\documentclass[reqno]{amsart}
\usepackage{amsfonts}
\usepackage{amsmath}
\usepackage{amsthm, amscd}
\usepackage{mathtools}
\usepackage{amssymb}
\usepackage{mathrsfs}
\usepackage{centernot}
\usepackage{mathabx}
\usepackage{hyperref}

\allowdisplaybreaks
\usepackage[alphabetic]{amsrefs}

\numberwithin{equation}{section}

\theoremstyle{plain}\newtheorem{thm}{Theorem}[section]
\theoremstyle{plain}\newtheorem{cor}[thm]{Corollary}
\theoremstyle{plain}\newtheorem{lem}[thm]{Lemma}
\theoremstyle{plain}\newtheorem{defn}[thm]{Definition}
\theoremstyle{plain}\newtheorem{prop}[thm]{Proposition}
\theoremstyle{plain}\newtheorem*{thm*}{Theorem}
\theoremstyle{remark}\newtheorem{rmk}[thm]{Remark}

\newcommand{\spinc}{\text{spin}^c}
\newcommand{\Spinc}{\text{Spin}^c}

\newcommand{\Hfrom}{\widehat{\mathit{HM}}}
\newcommand{\Hto}{\widecheck{\mathit{HM}}}
\newcommand{\Harrow}{\overrightarrow{HM}}
\newcommand{\Hred}{\mathit{HM}}
\newcommand{\id}{\textrm{id}}
\newcommand{\frs}{\mathfrak{s}}
\newcommand{\frt}{\mathfrak{t}}
\newcommand{\CSD}{\mathcal{L}}
\newcommand{\CSDpert}{\rlap{$-$}\CSD}

\newcommand{\moduli}[1]{\mathcal{C}_{#1}(X,\frs)}

\newcommand{\bC}{\mathbb{C}}
\newcommand{\bS}{\mathbb{S}}

\DeclareMathOperator{\inj}{inj}
\DeclareMathOperator{\Ric}{Ric}
\DeclareMathOperator{\supp}{supp}

\DeclareMathOperator{\Spin}{Spin}
\DeclareMathOperator{\SO}{SO}
\DeclareMathOperator{\SU}{SU}
\DeclareMathOperator{\UU}{U}

\DeclareMathOperator{\Hom}{Hom}
\DeclareMathOperator{\End}{End}
\DeclareMathOperator{\Vol}{Vol}

\title[Monopoles and foliations]{Monopoles and foliations without holonomy-invariant transverse measure}
\author{Boyu Zhang}
\begin{document}

\begin{abstract}
This article proves a uniform exponential decay estimate for Seiberg-Witten equations on non-compact 4-manifolds with exact symplectic ends of bounded geometry. This is an extension of the analysis for asymptotically flat almost K\"ahler (AFAK) structures by Kronheimer and Mrowka \cite{kronheimer1997monopoles}. As an application, we construct an invariant for smooth foliations without holonomy-invariant transverse measure, which takes value in the boundary-stable version of the monopole Floer homology group introduced by Kronheimer and Mrowka \cite{kronheimer2007monopoles}, without invoking the Eliashberg-Thurston perturbation.
\end{abstract}
\maketitle
\section{Introduction}

\subsection{Seiberg-Witten equations}
Let $(M,g)$ be an oriented Riemannian 4-manifold possibly with boundary. The manifold $M$ is allowed to be empty.
We say that $M$ is \emph{cylindrical}, if 
$M$ is isometric to $(-\infty,0]\times Y$
 for a closed, oriented Riemannian 3-manifold $Y$. We say that $M$ is endowed with an \emph{exact symplectic structure with bounded geometry} (or an \emph{ESBG structure}), if there is an exact symplectic form $\omega=d\theta$ on $M$, such that the following conditions hold:
 \begin{enumerate}
\item $\omega$ is compatible with $g$. In other words, $\omega$ is a self-dual 2-form with norm $\sqrt{2}$ under the metric $g$;
\item $\partial M$ is compact, $(M,g)$ is complete as a metric space;
\item The injectivity radius of $M$ has a positive lower bound on $M-N(\partial M)$, where $N(\partial M)$ is a tubular neighborhood of $\partial M$ with compact closure;
\item Let $R$ be the Riemann curvature tensor of $M$, let $\nabla$ be the Levi-Civita connection, then $\sup_M|\nabla^k R|<+\infty$ for every $k\ge 0$;
\item $\sup_M|\nabla^k \theta|<+\infty$ for every $k\ge 0$.
\end{enumerate}
If $M$ is empty, then by definition, $M$ is both cylindrical and has an (empty) ESBG structure.

Suppose $(M,g)$ is endowed with an $ESBG$ structure by the symplectic form $\omega=d\theta$, then $(\omega,g)$ induces an almost complex structure on $M$, and there is a \emph{canonical $\spinc$ structure} over $M$ such that
\begin{align*}
\mathbb{S}^+ &=T^{0,0}\, M \oplus T^{0,2}\, M, \\
\mathbb{S}^- &= T^{0,1}\, M.
\end{align*}
We will denote the canonical $\spinc$ structure by $\frs_{M,\omega}$.
Let $\Phi_0$ be the section of $\mathbb{S}^+$ given by $1\in\Gamma(M,T^{0,0}\,M)$, let $D$ be the Dirac operator, then there exists a unique $\spinc$ connection $A_0$ on $\frs_{M,\omega}$ such that $D_{A_0} \Phi_0=0$. We call $A_0$ the \emph{canonical $\spinc$ connection} on $\frs_{M,\omega}$. For more details on the canonical $\spinc$ structure and the canonical $\spinc$ connection, see, for example,  \cite[Sections 4.2-4.3]{hutchings1997introduction}.

\begin{defn}\label{def_cylindrical_ESBG_end}
Suppose $X$ is a complete oriented Riemannian 4-manifold. We say that \emph{$X$ has cylindrical and ESBG ends}, if there are two 4-dimensional submanifolds with boundary $M_c, M_s\subset X$ and an exact symplectic form $\omega=d\theta$ on $M_s$, such that the following holds. 
\begin{enumerate}
\item $M_s$ and $M_c$ are disjoint closed subsets of $X$, and the closure of $X-M_c-M_s$ is compact. $M_c$ and $M_s$ are allowed to be empty.
\item $M_c$ is cylindrical, and $\omega=d\theta$ is an ESBG structure on $M_s$.
\item If $M_s$ is nonempty, then the ESBG structure on $M_s$ can be extended to a neighborhood of $M_s$.
\item If $M_c$ is nonempty, then the cylindrical structure on $M_c$ can be extended to a neighborhood of $M_c$. Namely, there exists a neighborhood $N(M_c)$ of $M_c$ and a closed oriented Riemannian 3-manifold $Y$, such that $N(M_c)$ is isometric to $(-\infty,\epsilon]\times Y$ for $\epsilon>0$, where $M_c\subset N(M_c)$ is mapped to $(-\infty,0]\times Y$.
\end{enumerate}
\end{defn}

\begin{defn}\label{def_admissible_spinc_structure}
Let $X$ be a Riemannian 4-manifold with cylindrical and ESBG ends, where $M_c$ is the cylindrical end, and $(M_s,\omega=d\theta)$ is the ESBG end.
We say that a $\spinc$ structure $\mathfrak{s}$ on $X$ is \emph{admissible}, if there is an isomorphism from $\frs|_{M_s}$ to the canonical $\spinc$ structure $\frs_{M_s,\omega}$.
\end{defn}
\begin{rmk}
To simplify  notation, if $\mathfrak{s}$ is an admissible $\spinc$ structure on $X$, we will always assume that there is a fixed isomorphism from $\mathfrak{s}|_{M_s}$ to $\frs_{M_s,\omega}$, and we will identify the positive and negative spinor bundles of $\mathfrak{s}|_{M_s}$ with $T^{0,0}M_s \oplus T^{0,2}M_s$ and $T^{1,1}M_s$ respectively.
\end{rmk}

Let $\frs$ be an admissible $\spinc$ structure on $X$, let $\mathbb{S}=\mathbb{S}^+\oplus \mathbb{S}^-$ be the corresponding spinor bundle, and let $\rho$ be the Clifford multiplications.
Let $r>0$ be a constant; later we will require $r$ to be sufficiently large. 
This article studies a system of perturbed Seiberg-Witten equations on $X$ given by 
\begin{equation}\label{SW_intro}
\begin{cases}
D_A\phi = \eta_1,\\
F_A^+ = (\phi \phi^*)_0+\eta_2,
\end{cases}
\end{equation}
where $(\eta_1,\eta_2)= (0,-ir\omega/4 + F_{A_0}^+)$ on $M_s$, and $(\eta_1,\eta_2)$ is given by a tame perturbation introduced by \cite[Section 10]{kronheimer2007monopoles} on $M_c$. The precise definition of \eqref{SW_intro} will be given
in Section \ref{sec: SW equations}.

Let $A$ be $\spinc$ connection of $\frs$, and let $\phi$ be a section of $\bS^+$. Recall that we have chosen a fixed isomorphism from $\frs|_{M_s}$ to $\frs_{M_s,\omega}$ and use it to identify their spinor bundles, so we have $\mathbb{S}^+|_{M_s}= T^{0,0}M_s \oplus T^{0,2}M_s$, where the almost complex structure on $M_s$ is induced by $\omega$. Decompose $\phi|_{M_s}$ as
$$
\phi|_{M_s}= \sqrt{r} (\alpha+\beta),
$$
where $\alpha\in T^{0,0}M_s$, $\beta\in T^{0,2}M_s$.
 Let  $\nabla_A'$ be the projection of $\nabla_A|_{M_s}$ to $T^{0,2}M_s$. More precisely, $\nabla_A'$ is a connection of $T^{0,2}M_s$ such that for every section $s\in \Gamma(M_s,T^{0,2}M_s)$, the section $\nabla_A's$ is equal to the projection of $\nabla_A s$ to $T^{0,2}M_s$. Let $A_0$ be the canonical $\spinc$ connection on $\frs|_{M_s}$, then there exists a unique function $a$ on $M_s$ which takes values in $i\mathbb{R}$ such that $A|_{M_s}-A_0 = a \cdot \id_{\bS^+\oplus \bS^-}$. To simplify  notation, we will write $a=A|_{M_s}-A_0$ for the rest of the article. Notice that $a$ defines a unitary connection on the trivial bundle $T^{0,0}M_s$. We define the \emph{energy density function} of $(A,\phi)$ on $M_s$ as:
\begin{equation}
\label{eqn_def_Er(A,phi)}
E_r(A,\phi)=|1-|\alpha|^2-|\beta|^2|^2+|\beta|^2+|\nabla_a \alpha|^2+|\nabla_A'\beta|^2 + |F_a|^2.
\end{equation}

Define a function $d$ on $M_s$ as follows.
For each connected component $M_s^{(k)}$ of $M_s$, if $\partial M_s^{(k)}$ is nonempty, let $d$ on $M_s^{(k)}$ be the distance function to $\partial M_s^{(k)}$. Otherwise, fix a point $x^{(k)}\in M_s^{(k)}$, and let $d$ on $M_s^{(k)}$ be the distance function to $x^{(k)}$.
The main analytic result of this article is the following estimate.
\begin{thm} \label{thm_uniform_exponential_decay}
Let $X$ be a Riemannian 4-manifold with cylindrical and ESBG ends, where $M_c$ is the cylindrical end, and $(M_s,\omega=d\theta)$ is the ESBG end.
Then there exists a constant $r_0>0$ and a constant $z$, such that for every admissible $\spinc$ structure $\frs$ and every $r>r_0$, there exists a constant $C$ depending on $r$, with the following significance. If $(A,\phi)$ is a solution to \eqref{SW_intro} on $X$ such that 
$$\int_{M_s} E_r(A,\phi) <+\infty,$$
 then
\begin{equation}\label{eqn_uniform_exp_decay_intro}
 E_r(A,\phi)< C \,e^{-\sqrt{r}\cdot d/z}
\end{equation}
pointwise on $M_s$.
\end{thm}

Theorem \ref{thm_uniform_exponential_decay}  will be re-stated and proved  in Section \ref{thm: uniform decay}. The theorem is an extension of the analysis on asymptotically flat almost K\"ahler (AFAK) structures by Kronheimer and Mrowka \cite[Section 3(iii)]{kronheimer1997monopoles}. This estimate implies that the zero-dimensional component of a relevant moduli space of solutions to Seiberg-Witten equations on $X$ is compact, therefore one can define topological invariants for $X$ by counting the solutions of the Seiberg-Witten equations. In Section \ref{sec_Floer_chains}, we will follow the strategies of \cite{kronheimer2007monopolesAndLens} and construct an invariant for $X$ which takes value in the monopole Floer homology group.

\subsection{Taut foliations}
\label{subsec_taut_foliations}
Theorem \ref{thm_uniform_exponential_decay} and the construction in Section \ref{sec_Floer_chains} can be applied to the study of taut foliations.

Let $Y$ be a smooth, closed, oriented three-manifold, let $\mathcal{F}$ be an oriented foliation on $Y$. By definition, $\mathcal{F}$ is called \emph{taut} if the following condition is satisfied: for every point $p \in Y$, there exists an embedded circle in $Y$ that contains $p$ and is transverse to $\mathcal{F}$.

One of the fundamental problems in the study of taut foliations is their existence on a given 3-manifold. By the Roussarie-Thurston theorem, if $Y$ supports a taut foliation, then every embedded sphere in $Y$ is either nullhomotopic or is isotopic to a leaf. Reeb's stability theorem then implies that if $Y$ supports a taut foliation, then $Y$ is either irreducible, or is homeomorphic to $S^2\times S^1$ with the product foliation. Gabai \cite{gabai1983foliations} proved that every irreducible, closed, oriented three-manifold with $b_1\ge 1$ supports a taut foliation. The existence problem for taut foliations on irreducible manifolds with $b_1=0$ is still unsolved. It was proved in \cite{kronheimer2007monopolesAndLens} that if $Y$ is a rational homology sphere supporting a smooth taut foliation, then the reduced monopole Floer homology group $\Hred_{\bullet}(Y)$ must be nontrivial. This implies, for example, that the lens spaces do not support any smooth taut foliations except for $S^1\times S^2$.  The theorem was generalized to $C^{0}$-taut foliations by Bowden \cite{bowden2015approximating}. 

The flexibility of taut foliations has also been studied for years.  Eynard-Bontemps \cite{eynard2016connectedness} proved that if two taut foliations can be homotoped to each other via plane fields, then they can be homotoped to each other via foliations. On the other hand, Vogel \cite{vogel2013uniqueness} and Bowden \cite{bowden2016contact} have constructed examples of taut foliations that are homotopic as plane fields but cannot be homotoped to each other via taut foliations.

The proofs of the non-vanishing and non-flexibility results above rely on the following perturbation theorem, which is due to Eliashberg and Thurston for $C^2$ foliations and is generalized by Bowden to the $C^{0}$ case:

\begin{thm*}[Eliashberg-Thurston \cite{eliashberg1998confoliations}, Bowden \cite{bowden2015approximating}] \label{Thm:EliashThurs}
Let $\mathcal{F}$ be an orientable $C^{0}$ foliation on an oriented 3-manifold $Y$, and assume $(Y,\mathcal{F})$ is not homeomorphic to the product foliation on $S^2\times S^1$. Then $\mathcal{F}$ can be $C^0$ approximated both by a sequence of positive contact structures and a sequence
of negative contact structures. If $\mathcal{F}$ is taut, then the positive contact approximations are weakly semi-fillable.
\end{thm*}

If $Y$ supports taut foliations, then the non-vanishing of $\Hred_{\bullet}(Y)$ was proved by the theorem above and the non-vanishing property of semi-fillable contact structures \cite[Section 6.4]{kronheimer2007monopolesAndLens}. The examples in \cite{vogel2013uniqueness, bowden2016contact} were proved by first showing that the perturbed contact structures are unique up to isotopy, and then showing that the isotopy classes of the corresponding contact structures are different.\\

As an application of Theorem \ref{thm_uniform_exponential_decay}, if $\mathcal{F}$ is a smooth foliation on $Y$ that does not admit holonomy-invariant transverse measures,
we will construct two invariants $c_+(\mathcal{F})\in\Hto_\bullet(Y)$ and $c_-(\mathcal{F})\in\Hto_\bullet(-Y)$ without invoking the Eliashberg-Thurston perturbation. Here, $\Hto_\bullet(\cdot)$ is the boundary-stable version of the monopole Floer homology introduced by \cite{kronheimer2007monopoles}. We will then apply the invariants to the study of the existence and flexibility of taut foliations.

Notice that $\mathcal{F}$  is taut if and only if there exists a \emph{closed} 2-form $\hat\omega$ on $Y$, such that $\hat\omega$ is everywhere positive on the tangent plane field of $\mathcal{F}$ \cite[Proposition 10.4.1]{candel2000foliations}.
On the other hand, $\mathcal{F}$ has no holonomy-invariant transverse measure if and only if there exists an \emph{exact} 2-form $\hat \omega$ on $Y$, such that $\hat\omega$ is everywhere positive on the tangent plane field of $\mathcal{F}$ \cite[Theorem II.2]{sullivan1976cycles}. Therefore, if $Y$ is a rational homology sphere, then a foliation $\mathcal{F}$ is taut if and only if it has no holonomy-invariant transverse measure. 

In Section \ref{sec: monopole floer theory}, we will construct the invariants $c_+(\mathcal{F})$ and $c_-(\mathcal{F})$,  and show that the gradings of  $c_+(\mathcal{F})\in \Hto(Y)$ and $c_-(\mathcal{F})\in \Hto(-Y)$ are given by the homotopy classes of $\mathcal{F}$ as plane fields on $Y$ and $-Y$ respectively. We will also show that $c_\pm(\mathcal{F})$ have nonzero images in the reduced monopole Floer homology groups under the map $j_*$ defined by \cite[Proposition 22.2.1]{kronheimer2007monopoles}. Therefore, the existence of $c_+(\mathcal{F})$ gives an alternative proof for the nonvanishing theorem of $\Hred_{[\mathcal{F}]}(Y)$ for smooth taut foliations on rational homology spheres \cite[Theorem 2.1]{kronheimer2007monopolesAndLens}. 

The invariants $c_\pm(\mathcal{F})$ can also be used to study the flexibility of foliations. In Section \ref{sec:applications}, we will construct smooth foliations without holonomy-invariant transverse measure that are homotopic as plane fields but have different invariants $c_+$. Since $c_\pm(\mathcal{F})$ are invariant under smooth deformations, this gives examples of smooth foliations that are homotopic as plane fields but cannot be smoothly deformed to each other via foliations without holonomy-invariant transverse measure.

It should be pointed out that if $\mathcal{F}$ is a smooth foliation without holonomy-invariant transverse measure, then there exist linear deformations of $\mathcal{F}$ to both positive and negative contact structures \cite[Theorem 2.1.2]{eliashberg1998confoliations}. It is straightforward to verify that the space of all positive (negative) linear deformations is convex, so the contact structures obtained by linear deformations are unique up to isotopy. Therefore, the contact elements of the linearly deformed contact structures also give two invariants for $\mathcal{F}$ in the monopole Floer homology groups. The relation between $c_{\pm}(\mathcal{F})$ and the corresponding contact invariants is not clear to the author.

\subsection{Acknowledgements}
I would like to express my most sincere gratitude to my Ph.D. advisor, Clifford Taubes, for his patient guidance and encouragement. I would like to thank Peter Kronheimer and Tomasz Mrowka for helping me understand their work. I would like to thank Jonathan Bowden, Dan Cristofaro-Gardiner, Amitesh Datta, Mariano Echeverria, Chris Gerig, Jianfeng Lin, Cheuk Yu Mak, Jiajun Wang, and Yi Xie for many helpful discussions. Finally, I want to thank the anonymous referee for reading the manuscript carefully and giving numerous valuable suggestions.

\section{The Seiberg-Witten equations}\label{sec: SW equations}
This section briefly reviews the definition of Seiberg-Witten equations, and introduces a perturbation  on manifolds with cylindrical and ESBG ends. We will follow the notations from \cite{kronheimer2007monopoles}. The reader may refer to \cite{morgan1996seiberg, hutchings1997introduction} for more details.
\subsection{$\Spinc$ Structures}
For $n\ge 2$, let $\Spin(n)$ be the connected double cover of $\SO(n)$. Let  $\Spinc(n) = \big(\UU(1)\times \Spin(n)\big)/\{\pm1\}$, where $1\in \UU(1)\times \Spin(n)$ is the unit element, and the two coordinates of $-1\in \UU(1)\times \Spin(n)$ are given by $-1\in \UU(1)$ and the non-trivial element in the preimage of $1\in \SO(n)$. Let $X$ be an oriented Riemannian 4-manifold. By definition, a $\spinc$ structure $\frs$ on $X$ is a principal $\Spinc(4)$--bundle
which is a lift of the oriented orthonormal frame bundle via the surjection
$$
\Spinc(4) =\big( \UU(1)\times \Spin(4)\big)/\{\pm1\}\to \Spin(4)/\{\pm1\} \cong \SO(4).
$$

 The group $\Spinc(4)$ has a standard unitary representation on $\mathbb{C}^4$. Suppose $\frs$ is a $\spinc$ structure on $X$, then the spinor bundle of $\frs$ is defined as $\bS=\frs\times_{\Spinc(4)}\bC^4$. There is a Clifford multiplication $\rho:T^*X\to\Hom (\bS,\bS)$ which satisfies $\rho(v)^2=-\|v\|^2$ for all $v\in T^*X$, and the action $\rho$ extends to $\wedge^* T^*M$. 
 Let $d\Vol$ be the volume form of $X$, then $\rho(d\Vol)^2=\id|_{\bS}$. 
 Let $\bS^+$ and $\bS^-$ be the eigenspaces of $\rho(d\Vol)$ with eigenvalues $-1$ and $1$ respectively, then both $\bS^+$ and $\bS^-$ have rank $2$. Let $\Lambda^+(X)$ be the vector bundle of self-dual 2-forms on $X$, and let $\End_0(\bS^+)$ be the traceless endomorphisms of $\bS$, then $\rho$ maps $\Lambda^+(X)\otimes\mathbb{C}$ isomorphically to $\End_0(\bS^+)$.

A unitary connection $A$ on ${\bS}$ is called a \emph{$\spinc$ connection} if $\nabla_A \rho=0$, where $\nabla_A$ is the coupled connection of $A$ and the Levi-Civita connection on $TX\otimes \Hom(\bS,\bS)$.
Every $\spinc$ connection decomposes as two unitary connections on $\bS^+$ and $\bS^-$, and the connection on $\bS^+$ induces a connection on $\det(\bS^+)$. We use $A^t$ to denote the connection on $\det(\bS^+)$ induced by $A$, and use $D_A$ to denote the Dirac operator defined by $A$.

The definition of $\spinc$ structures on 3-manifolds is similar. 
A $\spinc$ structure on an oriented Riemannian 3-manifold $Y$ is a principal $\Spinc(3)$--bundle which is a lift of the oriented orthonormal frame bundle. Notice that $\Spinc(3) = \SU(2)\times\UU(1)/\{\pm1\} \cong \UU(2)$.  If $\frt$ is a $\Spinc$ structure on a 3-manifold $Y$, then the spinor bundle of $\frt$ is defined as $\bS=\mathfrak{t}\times_{\UU(2)}\mathbb{C}^2$, and there is a Clifford multiplication $\rho:T^*M\to\Hom (\bS,\bS)$. A unitary connection $B$ on the spinor bundle $\bS$ is called a \emph{$\spinc$ connection} if $\nabla_B \rho=0$.

\subsection{Configuration spaces}
For a smooth vector bundle $V$ over a smooth manifold $M$, we say that a section $s$ of $V$ is \emph{locally $L_k^p$}, if for every $p\in M$ there exists a neighborhood $U$ of $p$ and a (smooth) trivialization of $V|_U$, such that $s|_U$ is $L_k^p$ under this trivialization. We say that a connection $A$ of $V$ is \emph{locally $L_k^p$}, if there exists a smooth connection $\hat A$, such that $A-\hat A$ is a locally $L_k^p$ section of $T^*M\otimes V$.

We recall the following definitions of configuration spaces  from \cite{kronheimer2007monopoles}. 
\begin{defn}
Let $\mathfrak{t}$ be a $\spinc$ structure on a closed 3-manifold $Y$, let $\bS$ be the spinor bundle. Define $\mathcal{C}_{k}(Y,\mathfrak{t})$ to be the set of pairs $(B,\psi)$, where $B$ is a locally $L_k^2$ $\spinc$ connection of $\mathfrak{t}$, and $\psi$ is a locally $L_k^2$ section of $\bS$. Define $\mathcal{C}(Y,\mathfrak{t}) = \cap_{k\ge 1} \mathcal{C}_{k}(Y,\mathfrak{t})$.
\end{defn}
\begin{defn}
Let $\frs$ be a $\spinc$ structure on a compact 4-manifold $X$ possibly with boundary, let $\bS=\bS^+\oplus\bS^-$ be the spinor bundle. Define $\mathcal{C}_{k}(X,\frs)$ to be the set of pairs $(A,\phi)$ such that $A$ is a locally $L_k^2$ $\spinc$ connection of $\frs$, and $\phi$ is a locally $L_k^2$ section of $\bS^+$. Define $\mathcal{C}(X,\frs) = \cap_{k\ge 1} \mathcal{C}_{k}(X,\frs)$.
\end{defn}

Now let $X$ be a Riemannian 4-manifold with cylindrical and ESBG ends as given by Definition  \ref{def_cylindrical_ESBG_end}. Suppose $M_c$ is the cylindrical end, and $(M_s,\omega=d\theta)$ is the ESBG end.
Let $\frs$ be an admissible $\spinc$ structure on $X$ as in  Definition  \ref{def_admissible_spinc_structure}. We define the configuration space for $(X,\frs)$ as follows. Recall that for an oriented closed 3-manifold $Y$, the $\spinc$ structures on $Y$ are in one-to-one correspondence with the $\spinc$ structures on $(-\infty,0]\times Y$ (\cite[Section 4.3]{kronheimer2007monopoles}).

\begin{defn}\label{def: configuration space}
Let $X, M_c,M_s,\omega=d\theta,\frs$ be as above, let $r>0$ be a constant. For $k\ge 1$, define $\mathcal{C}_{k}(X,\frs)$ to be the set of pairs $(A,\phi)$ such that:
\begin{enumerate}
\item $A$ is a locally $L_k^2$ $\Spinc$-connection of $\frs$, and $\phi$ is a locally $L_{k}^2$ section of $\bS^+$;
\item $\displaystyle \int_{M_s} E_r(A,\phi)<+\infty$, where $E_r(A,\phi)$ is defined by \eqref{eqn_def_Er(A,phi)};
\item On the cylindrical end $M_c=(-\infty,0]\times Y$, let $\frt$ be the $\spinc$ structure on $Y$ induced by $\frs|_{M_c}$.  Then the restriction of $(A,\phi)$ on $M_c$ gives a path $(-\infty,0]\to \mathcal{C}_{k-1}(Y,\mathfrak{t})$ that is convergent at $-\infty$, in the $L_{k-1}^2$ topology of $\mathcal{C}_{k-1}(Y,\mathfrak{t})$.
\end{enumerate}
Define 
$\mathcal{C}(X,\frs) = \cap_{k\ge 1} \mathcal{C}_{k}(X,\frs).$
\end{defn}

\subsection{Strongly tame perturbations} \label{subsection: Strongly tame perturbations}

Let $Y$ be an oriented closed three-manifold, let $\frt$ be a $\spinc$ structure on $Y$, and let $B_0$ be a smooth $\spinc$ connection of $\frt$. Let $\CSD$ be the Chern-Simons-Dirac functional on $\mathcal{C}(Y, \mathfrak{t})$ defined by \cite[Definition 4.1.1]{kronheimer2007monopoles} with respect to $B_0$. A Banach space of ``tame'' perturbations of $\CSD$
was introduced and studied in \cite[Sections 10, 11]{kronheimer2007monopoles}. For the purpose of this article, we need to introduce a stronger condition on the perturbations. 

Recall that if $\mathfrak{q}$ is a perturbation of the Chern-Simons-Dirac functional, then the formal gradient of $\mathfrak{q}$ defines a perturbation $(\hat{\mathfrak{q}}^0,\hat{\mathfrak{q}}^1)$ for the Seiberg-Witten equations on the cylinder $[0,1] \times Y$ \cite[Secion 10.1]{kronheimer2007monopoles}.

\begin{defn}\label{def_strongly_tame}
Let $\frs$ be the 
$\spinc$ structure on $[0,1] \times Y$ induced by $\mathfrak{t}$.
A perturbation $\mathfrak{q}$ of $\CSD$ is called \emph{strongly tame} if
\begin{enumerate}
\item It is a tame perturbation as defined by  \cite[Definition 10.5.1]{kronheimer2007monopoles}.
\item There is a constant $m_0$ such that 
$$
\|\hat{\mathfrak{q}}^0 (A,\phi)\|_{C^0} \le m_0(\|\phi\|_{C^0}+1)
$$
for all $(A,\phi)\in\mathcal{C}([0,1] \times Y,\frs)$.
\item There is a constant $m_1$ such that 
$$
\|\hat{\mathfrak{q}}^1 (A,\phi)\|_{C^0} \le m_1
$$
for all $(A,\phi)\in\mathcal{C}([0,1] \times Y,\frs)$.
\end{enumerate}
\end{defn}

Using the calculations in \cite[p.~176]{kronheimer2007monopoles}, it is straight forward to verify that the cylindrical functions constructed in  \cite[Section 11.1]{kronheimer2007monopoles} are strongly tame.  The proof of \cite[Theorem 11.6.1]{kronheimer2007monopoles} defined a norm $\|\cdot\|_{\mathcal{P}}$ on the linear space generated by a seqeuence cylindrical functions, and proved that the completion with respect to $\|\cdot\|_{\mathcal{P}}$ gives a Banach space of tame perturbations. We consider a modified norm defined by
\begin{equation}
\label{eqn_def_norm_strongly_tame_perturbations}
\|\mathfrak{q}\|_{\hat{\mathcal{P}}} = \|\mathfrak{q}\|_{\mathcal{P}} + \sup_{(A,\phi)\in\mathcal{C}([0,1] \times Y,\frs)}  \Big(  \frac{\|\hat{\mathfrak{q}}^0 (A,\phi)\|_{C^0}}{\|\phi\|_{C^0}+1} + \|\hat{\mathfrak{q}}^1 (A,\phi)\|_{C^0} \Big).
\end{equation}
By the same argument as in \cite[Theorem 11.6.1]{kronheimer2007monopoles}, 
the completion with respect to $\|\cdot\|_{\hat{\mathcal{P}}}$ gives a Banach space of strongly tame perturbations that contains the given sequence of cylindrical functions. As a consequence, the transversality property \cite[Theorem 15.1.1]{kronheimer2007monopoles} still holds with respect to strongly tame perturbations.

\subsection{Perturbed Seiberg-Witten equations} \label{subsection: perturbation}
Let $X$ be a Riemannian 4-manifold with cylindrical and ESBG ends, where the cylindrical end is $M_c$ and the ESBG end is $(M_s,\omega=d\theta)$. Let $\frs$ be an admissible $\spinc$ structure on $X$, let $\bS=\bS^+\oplus\bS^-$ be the spinor bundle. Let $r>0$ be a constant. This section introduces a family of perturbations of Seiberg-Witten equations on $(X,\frs)$ parametrized by the constant $r$. Similar perturbations were used in \cite{taubes1996sw,kronheimer1997monopoles} and many other related works.

For $ (A,\phi) \in \moduli{k}$ with $k\ge 2$, define 
$$\mathfrak{F}(A,\phi)=(\rho(F_{A^t}^+)-(\phi\phi^*)_0, D_{A}\Phi), $$
where $(\phi\phi^*)_0$ is the traceless part of $\phi\phi^*$. By definition, $\mathfrak{F}(A,\phi)$ is a section of $i\mathfrak{su}({\bS}^+)\oplus{\bS}^-$. 

By Conditions (3), (4) of Definition \ref{def_cylindrical_ESBG_end}, the cylindrical and ESBG structures extend to neighborhoods of $M_c$ and $M_s$. Let ${M}_c'$ and $({M}_s',\omega'=d\theta')$ be the respective neighborhoods of $M_c$ and $M_s$ on which the cylindrical and ESBG structures extend. If $M_c=\emptyset$, then we take ${M}_c'=\emptyset$. By shrinking the neighborhoods, we may assume that the closures of ${M}_c'$ and ${M}_s'$ are disjoint. We also assume that ${M}_s'$ deformation retracts to $M_s$, therefore the isomorphism from $\frs|_{M_s}$ to $\frs_{M,\omega}$ extends to an isomorphism from $\frs|_{{M}_s'}$ to $\frs_{{M}_s',\omega'}$. In the following, we will fix such an isomorphism from $\frs|_{{M}_s'}$ to $\frs_{{M}_s',\tilde\omega}$. To simplify notation, we will use the fixed isomorphism to identify the spinor bundles of $\frs|_{{M}_s'}$ and $\frs_{{M}_s',\omega'}$.

Suppose $M_c$ is isometric to $(-\infty,0]\times Y$,
let $\frt$ be the $\spinc$ structure on $Y$ induced by $\frs|_{M_c}$. Let $\mathfrak{q}$ be a strongly tame perturbation on $(Y,\frt)$, then the flow line equation of the perturbed Chern-Simons-Dirac functional $\CSD+\mathfrak{q}$ can be written as $\mathfrak{F}(A,\phi) = \hat {\mathfrak{q}}(A,\phi)$, where $\hat {\mathfrak{q}}$ is the formal gradient of $\mathfrak{q}$. For the rest of this article, we will assume $\mathfrak{q}$ is strongly tame in the sense of Definition \ref{def_strongly_tame} and is admissible in the sense of \cite[Definition 22.1.1]{kronheimer2007monopoles}. Moreover, assume that $\|\mathfrak{q}\|_{\hat{\mathcal{P}}}\le 1$, where $\|\cdot\|_{\hat{\mathcal{P}}}$ is the norm defined by \eqref{eqn_def_norm_strongly_tame_perturbations}.

Recall that we have fixed an isomorphism from $\frs|_{{M}'_s}$ to $\frs_{{M}'_s,\omega'}$ and use it to identify the spinor bundles of $\frs|_{{M}'_s}$ and $\frs_{{M}'_s,\omega'}$.
There is a canonical section $\Phi_0$ of $\bS^+|_{{M}'_s}$ given by $1\in \Gamma({M}'_s,T^{0,0}{M}'_s)$, and a canonical $\spinc$ connection $A_0$ on $\frs|_{{M}'_s}$ characterized by $D_{A_0}\Phi_0=0$. Define a section $\hat{u}\in C^{\infty}(M_s', i\mathfrak{su}({\bS}^+)\oplus{\bS}^-)$ on ${M_s'}$ by
\begin{align}
\hat{u} &= (-r(\Phi_0\Phi_0^*)_0+\rho(F_{A_0^t}^+),0) \nonumber\\
          &= \big(-\frac{ir}{4}\rho( \omega')+\rho(F_{A_0^t}^+),0\big). 
              \label{eqn: perturbation on symplectic ends}
\end{align}

Let $\hat\tau\in C_0^\infty (Z-{M_c}'-{M_s}', i\mathfrak{su}({\bS}^+))$.
 Let $\eta\ge 0$ be a smooth cut-off function on $X$ such that $\supp\eta \subset {M_c}'\cup {M_s}'$, and $\eta=1$ on $M_c \cup M_s $.  Define
\begin{equation}\label{eqn: definition of perturbation}
\hat{\mu}=\eta\hat{\mathfrak{q}}+\eta\hat{u}+(\hat\tau,0).
\end{equation}
The Seiberg-Witten equation that will be studied in this article is the equation for $(A,\phi)\in\moduli{k}$ given by:
\begin{equation} \label{SW}
\mathfrak{F}(A,\phi)=\hat{\mu}(A,\phi).
\end{equation}

\subsection{Convergence on different manifolds}

This subsection defines a version of convergence for a sequence of connections and spinors on different manifolds, and gives a sufficient condition for the existence of a convergent subsequence. For a Riemannian manifold $X$, a point $p\in X$, and $d>0$, we use $B_p(d)$ to denote the set of points $x\in X$ such that the distance from $x$ to $p$ is no greater than $d$.

\begin{defn} \label{defn of bounded geometry}
A sequence of pointed Riemannian manifolds possibly with boundary
$$\{(X_n, g_n, p_n)\}_{n\ge 1}$$
is said to have \emph{uniformly bounded geometry}, if there exists a sequence of positive real numbers $\{r_n\}_{n\ge 1}$ such that the following conditions hold:
\begin{enumerate}
\item $\displaystyle \lim_{n\to\infty}r_n=+\infty;$
\item The exponential map of $X_n$ at $p_n$ is defined on the closed ball of radius $r_n$ for each $n$;
\item There exists $\epsilon_0>0$, such that for all $n$, the injectivity radius of $X_n$ is greater than $\epsilon_0$ for every point in $B_{p_n}(r_n)$;
\item For every integer $n\ge 0$, let $R^{(n)}$ be the Riemann curvature tensor of $X_n$, then the sequence
$$
\Big\{\sup_{B_{p_n}(r_n)}|\nabla^k R^{(n)}|\Big\}_{n\ge 1}
$$
is bounded for each $k$.
\end{enumerate}
\end{defn}

\begin{rmk}
\label{rmk_bounded_geometry}
	Suppose $\{(X_n, g_n, p_n)\}_{n\ge 1}$ is a sequence of pointed Riemannian manifolds with uniformly bounded geometry, then for each $N>0$, there exists a constant $C_N>0$ with the following property. For every $n$, suppose $x\in B_{p_n}(r_n)$, let $\varphi_x:B(\epsilon_0)\to X_n$ be the normal coordinate of $(X_n,g_n)$ centered at $x$ with radius $\epsilon_0$. Then $\|\varphi_x^*\, g_n\|_{C^N(B(\epsilon_0))}\le C_N$. As a consequence, for each $k\in \mathbb{Z}^+,\alpha\in(0,1),r>0$, there exists a constant $Q$ such that $\|(B_{p_n}(r_n),g_n)\|_{k+\alpha,r}\le Q$ for all $n$, where $\|\cdot\|_{k+\alpha,r}$ is the norm defined in \cite[Section 2]{petersen1997convergence}. This observation will be used in the proof of Proposition \ref{prop: properness}.
\end{rmk}

\begin{defn} \label{defn of convergence}
Suppose $\{(X_n, g_n, p_n)\}_{n\ge 1}$ is a sequence of oriented pointed Riemannian 4-manifolds with uniformly bounded geometry. For each $n$,  let $\frs_n$ be a $\spinc$ structure on $X_n$, let ${\bS}_n=\bS_n^+\oplus \bS_n^-$ be the corresponding spinor bundle, and let $\rho_n: T^*X_n\to \Hom( {\bS}_n,\,{\bS}_n)$ the Clifford multiplications. 
Let $A_n$ be a locally $L_k^2$ $\spinc$ connection of $\frs_n$, let $\phi_n$ be a locally $L_k^2$ section of $\bS_n^+$. 

The sequence $\{(X_n, g_n, p_n, \frs_n, A_n, \phi_n)\}_{n\ge 1}$ is said to be convergent to 
$$(X,g,p,\frs, A, \phi)$$ up to gauge transformations, if there exists a sequence 
$\{(d_n, U_n, V_n, \varphi_n, \tilde{\varphi}_n, u_n)\}_{n\ge 1}$
 such that the following conditions hold:

\begin{enumerate}
\item $(X,g)$ is a connected complete Riemannian 4-manifold, and $p\in X$.
 $\{d_n\}_{n\ge 1}$ is a sequence of positive real numbers such that $ \lim_{n\to\infty} d_n =+\infty$. The element $V_n$ is an open neighborhood of $p_n$ in $X_n$, and $U_n$ is an open neighborhood of $p$ in $X$. Both $V_n$ and $U_n$ have compact closures in $X_n$ and $X$ respectively.
 \item  The exponential map of $X_n$ at $p_n$ is defined on the closed ball of radius $d_n$ for each $n$, and 
$B_{p_n}(d_n)\subset V_n$ in $X_n$, $B_p(d_n)\subset U_n$ in $X$. The element $\varphi_n$ is a diffeomorphism from $U_n$ to $V_n$ mapping $p$ to $p_n$. Moreover, for every compact subset $K$ of $X$, we have
$$\lim_{n\to\infty}\|\varphi_n^*(g_n)- g\|_{C^m(K\cap U_n)}=0,\,\,\,\text{for all } m\in\mathbb{N}.$$
\item Let $\bS$ be the spinor bundle of $\frs$ and let $\rho:T^*X\to \Hom(\bS,\bS)$ be the Clifford multiplication.
The element $\tilde{\varphi}_n$ is a smooth unitary isomorphism
from ${\bS}_n|_{U_n}$ to ${\bS}|_{V_n}$ lifting $\varphi_n$. Let $\tilde{\varphi}_n^*(\rho_n):T^*X\to \Hom(\bS,\bS)$ be the pull-back of $\rho_n$ via $\tilde{\varphi}$ and the tangent map of $\varphi_n$.
For every compact subset $K$ of $X$, we have
$$\lim_{n\to\infty}\|\tilde{\varphi}_n^*(\rho_n) - \rho\|_{C^m(K\cap U_n)}=0,\,\,\,\text{for all } m\in\mathbb{N}.$$
\item The element $u_n$ is a gauge transformation of $\frs_n$ on $V_n$, such that for every compact subset $K$ of $X$, we have
$$\lim_{n\to\infty}\|\tilde{\varphi}_n^*(u_n(A_n,\phi_n)) - (A,\phi)\|_{C^m(K\cap U_n)}=0,\,\,\text{for all } m\in\mathbb{N}.$$
\end{enumerate}
\end{defn}

\begin{rmk}
By our definition, when $X_n$'s are not connected, the convergence of $$\{(X_n, g_n, p_n, \frs_n, A_n, \phi_n)\}_{n\ge 1}$$ only depends on the connected components containing $p_n$.
\end{rmk}

\begin{prop} \label{prop: properness}
Let $\{(X_n, g_n, p_n)\}_{n\ge 1}$ be a sequence of pointed oriented Riemannian 4-manifolds with uniformly bounded geometry, let $\epsilon_0$, $\{r_n\}_{n\ge 1}$ be the constants given by Definition \ref{defn of bounded geometry}. For each $n$, let $\frs_n$ be a $\spinc$ structure on $X_n$, let $\bS_n$ be the spinor bundle, let $A_n$ be a locally $L_1^2$ $\Spinc$-connection for $\frs_n$, and let $\phi_n$ be a locally $L_1^2$ section of $\bS_n^+$. Assume that there exists a constant $C>0$ such that for every $n$ and every point $x\in B_{p_n}(r_n)$, 
\begin{align}
\int_{B_x(\epsilon_0)}|F_{A_n}|^2 &< C, \label{assumption: local curvature L2 bound} \\
 |\phi_n(x)| &<C,  \label{assumption: uniform C0 bound}
\end{align}
and
\begin{equation} \label{assumption: dirac harmonic}
D_{A_n}(\phi_n)=0.
\end{equation}
 Moreover, assume that $\mathfrak{F}(A_n,\phi_n)$ is smooth for each $n$, and 
\begin{equation} \label{assumption: Ck bound of SW}
\sup_{n\ge 1}\|\mathfrak{F}(A_n,\phi_n)\|_{C^k}<+\infty, \,\,\text{ for all } k\ge 1.
\end{equation}
Then there exists a subsequence of $\{(X_n, g_n, p_n, \frs_n, A_n, \phi_n)\}_{n\ge 1}$ and a configuration $(X,g,p,\frs, A, \phi)$, such that the subsequence  converges to $(X,g,p,\frs, A, \phi)$ in the sense of Definition \ref{defn of convergence}.
\end{prop}

\begin{proof}
Since $\{(X_n,g_n,p_n)\}_{n\ge 1}$ have uniformly bounded geometry, it follows from \cite[Theorem 2.2]{petersen1997convergence} and Remark \ref{rmk_bounded_geometry} that after taking a subsequence, there exists a complete, connected, pointed Riemannian manifold $(X,g,p)$ and a sequence $\{(d_n, U_n, V_n, \varphi_n)\}_{n\ge 1}$, such that Conditions (1), (2) of Definition \ref{defn of convergence} are satisfied. Although \cite[Theorem 2.2]{petersen1997convergence} requires $(X_n,g_n,p_n)$ to be complete, the proof also works for non-complete manifolds as long as Conditions (1), (2) of Definition \ref{defn of bounded geometry} holds. By taking a further subsequence, we may assume that $U_n\subset U_m\subset X$ for all $n\le m$.

Now we construct a $\spinc$ structure $\frs$ on $X$.
By \eqref{assumption: local curvature L2 bound}, the sequence $$\|\varphi_n^*(F_{A_n})\|_{L^2(U_n\cap K)}$$ is bounded for every compact subset $K$ of $X$. Take an embedded closed oriented surface $\Sigma$ in $X$, let $N(\Sigma)$ be a tubular neighborhood of $\Sigma$. Then 
$$\sup_{\{n|N(\Sigma)\subset U_n\}} \int_{N(\Sigma)}|\varphi_n^*(F_{A_n})|<+\infty,$$   As a consequence,  
$$\sup_{\{n|N(\Sigma)\subset U_n\}}  |\langle 
\varphi_n^*(c_1(\frs_n)), [\Sigma] \rangle|<+\infty $$
 Let $\Spinc(U_n)$ be the set of isomorphism classes of $\spinc$ structures on $U_n$. 
Since the first Chern class determines the $\spinc$ structure up to torsion, there exists a finite set $\Lambda_n\subset \Spinc(U_n)$, such that $\varphi_m^*(\frs_m)|_{U_n}\in \Lambda_n$ for all $m\ge n$.
Therefore, after taking a further subsequence, we may assume that $\varphi_m^*(\frs_m)|_{U_n}$ is isomorphic to $\varphi_n^*(\frs_n)$ for all $m\ge n$.  For each $n$, let $\zeta_{n}: \varphi_{n}^*(\frs_n)\to \varphi_{n+1}^*(\frs_{n+1})|_{U_{n}}$ be an isomorphism of $\spinc$ structures on $U_n$, let $\iota_n:U_n\to U_{n+1}$ be the inclusion map, then $\{(\varphi_n^*(\frs_n),U_n,\zeta_n,\iota_n)\}_{n\ge 1}$ generates a direct system. Taking the limit of this direct system yields a $\spinc$ structure $\frs$ on $X$ and isomorphisms $\tilde{\varphi}_n:\frs|_{U_n} \to \frs_n$ that are lifts of $\varphi_n$.

The only thing remaining to prove is the existence of $(A,\phi)$ and the gauge transformations $u_n$ satisfying Condition (4) of Definition \ref{defn of convergence}. Without loss of generality, we may assume that the closures of $V_n\subset X_n$ and $U_n\subset X$ are compact manifolds with boundary.
 For a pair of positive integers $n\ge m$, let $\mathcal{E}^{an} (A_n|_{V_m},\phi_n|_{V_m})$ be the analytic energy of $(A_n,\phi_n)$ on $V_m$ as defined by \cite[Definition 4.5.4]{kronheimer2007monopoles}. We will show that for every $m$,
 $$\sup_{n\ge m+1}\mathcal{E}^{an} (A_n|_{V_m},\phi_n|_{V_m})< +\infty.$$ 

Since the closure of $U_m$ in $X$ is compact for every $m$, by taking a further subsequence if necessary, we may assume that $\overline{U_{m-1}}\subset U_m$ for all $m$. Moreover, since $X$ is complete, we may take a further subsequence, such that for each $m\ge 1$, there is a cut-off function $\chi_m\ge 0$ on $X$ such that $\supp \chi_m \subset U_{m+1}$, $\chi_m|_{U_{m}}=1$, and $|\nabla\chi_m|\le 1$ for all $m$. 

 For $n\ge m+1$, let $\phi_n^{(m)}=(\chi_m\circ \varphi_n^{-1})\cdot\phi_n$ be a spinor on $V_n$.
By \eqref{assumption: uniform C0 bound}, \eqref{assumption: Ck bound of SW}, and Condition (2) of Definition \ref{defn of convergence}, we have
\begin{align} 
\|\mathfrak{F}(A_n,\phi_n^{(m)})\|_{L^2(V_{m+1})} &\le C_1\,(\|\mathfrak{F}(A_n,\phi_n)\|_{L^2(V_{m+1})} + \|\phi_n\|_{C^0} + \|\phi_n\|_{C^0}^2) \nonumber
\\
& \le C_2
\label{L2 bound of mathcal(F)}
\end{align}
for constants $C_1, C_2$ depending on $m$.

Let $\mathcal{E}^{top} (A_n|_{V_{m+1}},\phi_n^{(m)}|_{V_{m+1}})$ be the topological energy of $(A_n,\phi_n^{(m)})$ on $V_{m+1}$ as defined by \cite[Definition 4.5.4]{kronheimer2007monopoles}.
Since $\phi_n^{(m)}$ is compactly supported on $V_{m+1}$, we have
\begin{equation} \label{eqn_topological energy on Vm}
\mathcal{E}^{top} (A_n|_{V_{m+1}},\phi_n^{(m)}|_{V_{m+1}})
 = \frac{1}{4}\int_{V_{m+1}}F_{A_n^t}\wedge F_{A_n^t},
\end{equation}
which is bounded by a constant depending on $m$ because of \eqref{assumption: local curvature L2 bound}.

By \cite[(4.16)]{kronheimer2007monopoles}, $$\mathcal{E}^{an}(A_n|_{V_{m+1}},\phi_n^{(m)}|_{V_{m+1}})=\mathcal{E}^{top}(A_n|_{V_{m+1}},\phi_n^{(m)}|_{V_{m+1}})+\|\mathfrak{F}(A_n,\phi_n^{(m)})\|_{L^2(V_{m+1})}^2,$$
therefore by \eqref{L2 bound of mathcal(F)} and \eqref{eqn_topological energy on Vm}, $\mathcal{E}^{an} (A_n|_{V_{m+1}},\phi_n^{(m)}|_{V_{m+1}})$ is bounded by a constant depending on $m$. Since 
$$\mathcal{E}^{an} (A_n|_{V_{m+1}},\phi_n^{(m)}|_{V_{m+1}}) \ge \mathcal{E}^{an} (A_n|_{V_{m}},\phi_n|_{V_{m}})-C_3,$$ 
for a constant $C_3$ depending on $m$, we conclude that  $$\sup_{n\ge m+1}\mathcal{E}^{an} (A_n|_{V_m},\phi_n|_{V_m})< +\infty.$$ 

By Condition (2) of Definition \ref{defn of convergence}, the statement above implies that
$$\{\mathcal{E}^{an} (\tilde{\varphi}^*(A_n)|_{U_m},\tilde{\varphi}^*(\phi_n)|_{U_m})\}_{n\ge m+1}$$
 is bounded for every $m$.
Therefore, by a diagonal argument, the existence of $(A,\phi)$ and $u_n$ satisfying Condition (4) of Definition \ref{defn of convergence} follows from \cite[Theorem 5.2.1]{kronheimer2007monopoles}. 
\end{proof}

\section{Exponential decay of $E_r(A,\phi)$}\label{sec: exponential decay}
This section proves a weak version of Theorem \ref{thm_uniform_exponential_decay}, which will be stated as Proposition \ref{prop: exponential decay on symplectic ends}.

Let $X$ be a Riemannian 4-manifold with cylindrical and ESBG ends, and suppose the cylindrical end is $M_c$ and the ESBG end is $(M_s,\omega=d\theta)$. Let $\frs$ be an admissible $\spinc$ structure on $X$, let $\bS=\bS^+\oplus\bS^-$ be the spinor bundle of $\frs$, let $\Phi_0$ be the canonical section of $\bS^+|{M_s}$, and let $A_0$ be the $\spinc$ connection on $\frs|_{M_s}$ such that $D_{A_0}\Phi_0=0$.

 Let
$(A,\phi)\in\moduli{k}$
be a solution to \eqref{SW}. By the standard elliptic regularity arguments, 
 $(A,\phi)$ is locally $C^\infty$ on $X$ after suitable gauge transformations. Since the perturbation $\mathfrak{q}$ in \eqref{eqn: definition of perturbation} is assumed to be admissible, it follows from \cite[Proposition 13.4.1]{kronheimer2007monopoles} that Condition (3) of Defintion \ref{def: configuration space} implies the $C^0$ convergence of $(A,\phi)$ on $M_c$ after gauge transformations. As a consequence, 
 \begin{equation} \label{eqn_sup_|phi|_finite}
 \|\phi\|_{C^0(M_c)}<+\infty.
 \end{equation}

 Let $\inj(X)$ be the injectivity radius of $X$. By Definition \ref{def_cylindrical_ESBG_end}, $\inj(X)>0$. Let 
\begin{equation}\label{eqn_def_epsilon0}
\epsilon_0=\min\{\frac{\inj(X)}{2}, 1\}.
\end{equation}

The following convention will be adopted for the rest of this article unless otherwise stated: the notations $z$ or $z_i$ will denote positive real numbers that only depend on $X, M_s, M_c, \theta$, and the terms $\hat \tau$ and $\eta$ in \eqref{eqn: definition of perturbation}. The notation $r_0$ will denote a positive real number that depends only on the same set of data, and we will assume that $r > r_0$ for the constant $r$ in \eqref{eqn: perturbation on symplectic ends}. The value of $r_0$ may increase as the proof proceeds.

\subsection{$C^0$ bound}
 This subsection proves the following $C^0$ estimate. Recall that $A_0$ is the canonical connection on $\frs|_{M_s}$, and let $a=A|_{M_s}-A_0$
\begin{prop} \label{uniform C0 bound}
There exist constants $z$, $r_0$, such that for all $r>r_0$ and $(A,\phi)\in\moduli{k}$
satisfying \eqref{SW}, we have
$\|\phi\|_{C^0(X)}\le z\cdot \sqrt{r}$, and
 $\|F_a^+\|_{C^0(X)}\le z\cdot{r}$.
\end{prop}

The proof starts with the following $C^0$ estimate on $M_s$, which is adapted from \cite[Lemma 3.23]{kronheimer1997monopoles}. Recall that $\epsilon_0$ is the constant defined by \eqref{eqn_def_epsilon0}.
\begin{lem}\label{C0 bound on the negative end}
Let $N(\partial M_s)$ be the $\epsilon_0$-neighborhood of $\partial M_s$.
There exist  constants $z, r_0$, such that if $r>r_0$ and $(A,\phi)\in\moduli{k}$
solves \eqref{SW}, we have
$$\|\phi\|_{C^0(M_s- N(\partial M_s))}\le z\cdot \sqrt{r}.$$
\end{lem}

\begin{proof}
By (2.2) of \cite{taubes1996sw}, the following inequality holds on $M_s$ for a constant $z_1$:
$$
\frac{1}{2}d^*d|\phi|^2+|\nabla_A\phi|^2+\frac{1}{4}|\phi|^2(|\phi|^2-r)-z_1\cdot |\phi|^2 \le 0.
$$
Take $r_0>4z_1$. For $r>r_0$, we have
\begin{equation}\label{eqn_c0_phi_1}
\frac{1}{2}d^*d|\phi|^2+\frac{1}{4}|\phi|^2(|\phi|^2-2r) \le 0.
\end{equation}
For $x\in M_s-N(\partial M_s)$, let $\rho$ be the distance function to $x$ on $B_x(\epsilon_0)$.  Let $f$ be the function on the interior of $B_x(\epsilon_0)$ defined by $f=1/(\epsilon_0^2-\rho^2)^2$. 
Since $\epsilon_0$ is less than the injectivity radius of $X$, let $(g_{ij})_{1\le i,j\le 4}$ be the metric matrix of the normal coordinates of $B_x(\epsilon_0)$ centered at $x$, and let $g$ be the determinant of $(g_{ij})_{1\le i,j\le 4}$. We have
$$
d^*df=-\frac{1}{\rho^{n-1}\sqrt{g}}\frac{\partial}{\partial\rho}\big(\rho^{n-1}\sqrt{g}\cdot\frac{\partial f}{\partial \rho}\big).
$$
Notice that since $X$ has bounded geometry, $\|g\|_{C^0}$ and $\|\nabla g\|_{C^0}$ are both bounded by constants independent of $x$, and $g$ is bounded away from $0$ on $B_x(\epsilon_0)$. A straightforward calculation shows that for some large constant $z_2>0$,
\begin{equation}\label{eqn_c0_phi_2}
\frac{1}{2}d^*d\big((z_2)^2rf\big)+\frac{1}{4}\big((z_2)^2rf\big)\big((z_2)^2rf-2r\big)\ge 0.
\end{equation}
Let $x'$ be a point in $B_x(\epsilon_0)$ where the function $(z_2)^2rf-|\phi|^2$ achieves the minimum value. Since the limit of $f$ on $\partial B_x(\epsilon_0)$ is $+\infty$, such a point $x'$ exists in the interior of $B_x(\epsilon_0)$. 
By \eqref{eqn_c0_phi_1} and \eqref{eqn_c0_phi_2}, we have
$$
|\phi(x')|^2(|\phi(x')|^2-2r) \le \big((z_2)^2rf(x')\big)\big((z_2)^2rf(x')-2r\big),
$$
therefore $|\phi(x')|^2\le \max\{2r, (z_2)^2rf(x')\}\le 2r+(z_2)^2rf(x')$. This implies $|\phi|^2 \le 2r+(z_2)^2rf$, hence $|\phi(x)|\le z\cdot\sqrt{r}$ for $z=\sqrt{2+(z_2)^2}$. 
\end{proof}

Now we prove Proposition \ref{uniform C0 bound}.

\begin{proof}[Proof of Proposition \ref{uniform C0 bound}]
Let $r_0$ be the constant given by Lemma \ref{C0 bound on the negative end}. Increase the value of $r_0$ if necessary such that $r_0\ge 1$, and assume $r>r_0$. 

By \eqref{eqn_sup_|phi|_finite} and Lemma \ref{C0 bound on the negative end}, we have $\sup_X|\phi|<+\infty$.
Let $x_0\in X$ be a point such that $|\phi({x_0})|\ge \frac12 \sup_{X}|\phi|$. Let $\epsilon<\epsilon_0$ be a positive constant that will be determined later. Notice that $(A,\phi)$ satisfies the following equations on $B_{x_0}(\epsilon)$:
\begin{align}
\rho(F_{A^t}^+) &=(\phi\,\phi^*)_0+\hat{\mu}^0(A,\phi),  \label{eqn: SW1}\\
D_A \phi&= \hat{\mu}^1(A,\phi), \label{eqn:SW2}
\end{align}
where $\hat\mu$ is given by \eqref{eqn: definition of perturbation}.
Since the perturbation $\mathfrak{q}$ is strongly tame and $|\phi({x_0})|\ge \frac12 \sup_{X}|\phi|$, there exists a constant $z_0$ such that the following holds on $B_{x_0}(\epsilon)$:
\begin{align}
\|\hat{\mu}^0(A,\phi)\|_{C^0} &\le z_0\,(1+|\phi({x_0})|)+z_0\,r, \label{eqn: bound1}\\
\|\hat{\mu}^1(A,\phi)\|_{C^0} &\le z_0. \label{eqn: bound2}
\end{align}
In the following, we will write $\hat{\mu}^0(A,\phi)$ as $\hat{\mu}^0$, and $\hat{\mu}^1(A,\phi)$ as $\hat{\mu}^1$.
Applying $D_A$ to both sides of \eqref{eqn:SW2} yields
$$D_{A}^2\phi=D_A(\hat{\mu}^1).$$
By the Weitzenb\"ock formula, this implies
\begin{equation}\label{eqn: Weitzenbock formula}
 \nabla_A^*\nabla_A\phi+\frac{1}{2}\rho(F_{A^t}^+)\phi+\frac{1}{4}s\phi=D_A(\hat{\mu}^1),
\end{equation}
where $s$ is the scalar curvature of $X$.
Plug in \eqref{eqn: SW1} to \eqref{eqn: Weitzenbock formula}, and take the inner product with $\phi$, we obtain
\begin{equation*}
\frac{1}{2}d^*d|\phi|^2+|\nabla_A\phi|^2+\frac{1}{4}|\phi|^4+\frac{1}{4}\langle s\phi,\phi\rangle +\frac{1}{2}\langle \hat{\mu}^0\phi, \phi \rangle = \langle D_A(\hat{\mu}^1), \phi\rangle.
\end{equation*}
Recall that $r>r_0\ge 1$, hence by \eqref{eqn: bound1}, there exists a constant $z_1$ such that
$$
\frac{1}{2}d^*d|\phi|^2+|\nabla_A\phi|^2+\frac{1}{4}|\phi|^4\le \langle D_A(\hat{\mu}^1), \phi\rangle
+z_1\,r |\phi|^2 + z_1 |\phi(x_0)|\,|\phi|^2.
$$
By the arithmetic-geometric mean inequality,
$$
-\frac{1}{16}|\phi|^4- 4\,z_1^2\,r^2 \le -z_1\,r\,|\phi|^2 ,
$$
$$
-\frac{1}{16}|\phi|^4- 4\,z_1^2\,|\phi(x_0)|^2 \le -z_1\,|\phi(x_0)|\,|\phi|^2 .
$$
Adding the above three inequalities, we obtain
$$\frac{1}{2}d^*d|\phi|^2+|\nabla_A\phi|^2+\frac{1}{8}|\phi|^4-4z_1^2(\,r^2+|\phi({x_0})|^{2})\le \langle D_A(\hat{\mu}^1), \phi\rangle.$$
Let $h\ge 0$ be a smooth function on $B_{x_0}(\epsilon)$ such that $h=1$ on $B_{x_0}(\epsilon/4)$ and $\supp h\subset B_{x_0}(\epsilon/2)$. Let $\chi=h^4$.  Let $G_{x_0}\ge 0$ be the Green's function on $B_{x_0}(\epsilon)$ that has a pole at ${x_0}$ and equals zero on $\partial B_{x_0}(\epsilon)$. Then:
\begin{multline*}
\int_{B_{x_0}(\epsilon)}\Big(\frac{1}{2}d^*d|\phi|^2+|\nabla_A\phi|^2+\frac{1}{8}|\phi|^4-4z_1^2(\,r^2+|\phi({x_0})|^{2})\Big)\cdot G_{x_0} \cdot \chi \\
\le \int_{B_{x_0}(\epsilon)} \langle D_A(\hat{\mu}^1), \phi\cdot G_{x_0}\chi \rangle ,
\end{multline*}
which gives
\begin{multline}
\int_{B_{x_0}(\epsilon)} \Big( \big(-\frac{1}{2}\Delta(|\phi|^2\chi)+\frac{1}{2}|\phi|^2\Delta\chi+\nabla|\phi|^2\cdot\nabla\chi\big)\cdot G_{x_0}+|\nabla_A\phi|^2\,G_{x_0}\,\chi\\
+\frac{1}{8}|\phi|^4\, G_{x_0}\,\chi-4z_1^2(\,r^2+|\phi({x_0})|^{2})\,G_{x_0}\,\chi \Big)
\le
\int_{B_{x_0}(\epsilon)}\langle \hat{\mu}^1, D_A(\phi\,G_{x_0}\,\chi) \rangle.
\end{multline}
Therefore
\begin{multline} \label{eqn: integral over Green kernel}
\frac{1}{2}|\phi({x_0})|^2\le \int_{B_{x_0}(\epsilon)} \Big((-\frac{1}{2}|\phi|^2\Delta\chi-\nabla|\phi|^2\cdot\nabla\chi)\cdot G_{x_0}-|\nabla_A\phi|^2\,G_{x_0}\,\chi\\
-\frac{1}{8}|\phi|^4\, G_{x_0}\,\chi+4z_1^2(\,r^2+|\phi({x_0})|^{2})\,G_{x_0}\,\chi +
\langle \hat{\mu}^1, D_A(\phi\,G_{x_0}\,\chi) \rangle\Big).
\end{multline}
Recall that $\chi=h^4$, hence $|\Delta\chi| \le  4h^3|\Delta h| +12h^2|\nabla h|^2$. By the arithmetic-geometric mean inequality, there exists a constant $z_2$ such that
\begin{align*}
-\frac{1}{2}|\phi|^2\Delta\chi
\le & |\phi|^2(2h^3|\Delta h| +6h^2|\nabla h|^2)\\
\le & z_2(|\Delta h|^2 h^2+|\nabla h|^4)+\frac{1}{16}|\phi|^4\,h^4\\
= & z_2(|\Delta h|^2 h^2+|\nabla h|^4)+\frac{1}{16}|\phi|^4\,\chi.
\end{align*}
Similarly, there exists a  constant $z_3$  such that
\begin{align*}
|\nabla|\phi|^2|\cdot|\nabla\chi|
\le & 2|\phi|\cdot|\nabla_A\phi|\cdot (4h^3 |\nabla h|)\\
\le & 32\, |\phi|^2h^2|\nabla h|^2+\frac12 |\nabla_A\phi|^2 h^4\\
\le & z_3 |\nabla h|^4+\frac{1}{16}|\phi|^4\, h^4 +\frac12 |\nabla_A\phi|^2 h^4\\
= & z_3 |\nabla h|^4+\frac12 |\nabla_A\phi|^2\chi+\frac{1}{16}|\phi|^4\,\chi.
\end{align*}
By \eqref{eqn: bound2}, there exists a constant $z_4$ such that
\begin{align*}
&\int_{B_{x_0}(\epsilon)} \langle \hat{\mu}^1, D_A\phi\rangle\,G_{x_0}\,\chi
\le
\int_{B_{x_0}(\epsilon)} z_0\, |D_A\phi|\,G_{x_0}\,\chi
\\
\le
&\int_{B_{x_0}(\epsilon)} \frac{1}{2}|\nabla_A\phi|^2G_{x_0}\chi + z_4\int_{B_{x_0}(\epsilon)} G_{x_0}\,\chi.
\end{align*}
Therefore \eqref{eqn: integral over Green kernel} and the three estimates above yield
\begin{multline} \label{eqn: integral over Green kernel 1}
\frac{1}{2}|\phi({x_0})|^2 \le  
\int_{B_{x_0}(\epsilon)} \Big(z_2(|\Delta h|^2 h^2+|\nabla h|^4)+z_3 |\nabla h|^4\Big)\,G_{x_0} 
 + z_4\int_{B_{x_0}(\epsilon)} G_{x_0}\,\chi  
 \\
 +\int_{B_{x_0}(\epsilon)} |\hat{\mu}^1||\phi||\nabla(G_{x_0}\,\chi)|+4z_1^2(\,r^2+|\phi({x_0})|^{2})\int_{B_{x_0}(\epsilon)} \,G_{x_0}\,\chi.
\end{multline}
Recall that $|\phi(x_0)|\ge \frac12 \sup_X |\phi|$, therefore by \eqref{eqn: bound2},
\begin{align}
 \int_{B_{x_0}(\epsilon)} |\hat{\mu}^1||\phi||\nabla(G_{x_0}\,\chi)| &
 \le 2\, |\phi({x_0})| \int_{B_{x_0}(\epsilon)} |\hat{\mu}^1||\nabla(G_{x_0}\,\chi)| \nonumber \\
 & \le \frac{1}{8}|\phi({x_0})|^2 + 8\, \Big(\int_{B_{x_0}(\epsilon)} |\hat{\mu}^1||\nabla(G_{x_0}\,\chi)|\Big) ^2
 \nonumber
 \\
  & \le \frac{1}{8}|\phi({x_0})|^2 + 8\,z_0^2\, \Big(\int_{B_{x_0}(\epsilon)} |\nabla(G_{x_0}\,\chi)|\Big) ^2. \label{eqn: plug in before rearrange}
\end{align}
Notice that the constants $z_i$ do not depend on the choice of $\epsilon$, and there exist constants $z_5, z_6$ such that
$$
z_5\, \epsilon^2\le \int_{B_{x_0}(\epsilon)}G_{x_0} \le z_6\, \epsilon^2.
$$
Take $\epsilon = 1/(z_7\,\sqrt{r})$, with $z_7$ sufficiently large such that
\begin{equation}\label{eqn: epsilon small}
\int_{B_{x_0}(\epsilon)} \,G_{x_0}\le z_6 \, \epsilon^2\le  \min\,\Big\{\frac{1}{r},\frac{1}{32 \, z_1^2}\Big\}.
\end{equation}
Plug in \eqref{eqn: plug in before rearrange} and \eqref{eqn: epsilon small} to \eqref{eqn: integral over Green kernel 1}, and rearrange, we have
\begin{multline*}
\frac{1}{4}|\phi({x_0})|^2 \le  
 \int_{B_{x_0}(\epsilon)} \Big(z_2(|\Delta h|^2 h^2+|\nabla h|^4)+z_3 |\nabla h|^4\Big)\,G_{x_0} 
 + z_4\int_{B_{x_0}(\epsilon)} G_{x_0}\,\chi  \\
  +8\,z_0^2\,\Big(\int_{B_{x_0}(\epsilon)} |\nabla(G_{x_0}\,\chi)|\Big) ^2 + 4\,z_1^2\,r^2\int_{B_{x_0}(\epsilon)}\,G_{x_0}\,\chi .
\end{multline*}
Since $\epsilon = 1/(z_7\,\sqrt{r})$, one can choose the function $h$ such that the right-hand side of the above inequality is bounded by $z_8\cdot r$ for some constant $z_8$, hence the estimate on $\|\phi\|_{C^0(X)}$ is proved. The upper bound on $\|F_a^+\|_{C^0(X)}$ then follows from \eqref{eqn: SW1} and \eqref{eqn: bound1}.
\end{proof}

\subsection{Exponential decay on $M_s$}
Recall that the spinor $\phi|_{M_s}$ decomposes as $\phi=\sqrt{r}(\alpha+\beta)$, with $\alpha\in\Gamma(M_s, T^{0,0}M_s)$, $\beta\in\Gamma(M_s, T^{0,2}M_s)$. The spinor bundle $\bS^+$ has a canonical section $\Phi_0$ on $M_s$ given by $1\in \Gamma(M_s, T^{0,0}M_s)$, and there is a unique $\spinc$ connection $A_0$ on $\frs|_{M_s}$ such that $D_{A_0}\Phi_0=0$. Take $a=A|_{M_s}-A_0$, and 
take  $\nabla_A'$ to be the projection of $\nabla_A|_{M_s}$ to $T^{0,2}M_s$. The energy density function $E_r(A,\phi)$ is defined on $M_s$ by
\eqref{eqn_def_Er(A,phi)}.
If $(A,\phi)\in\moduli{k}$, then we have
$$
\int_{M_s}E_r(A,\phi)<+\infty.
$$

Recall that the function $d$ on $M_s$ is defined as follows.
For each connected component $M_s^{(k)}$ of $M_s$, if $\partial M_s^{(k)}$ is nonempty, then $d$ is the distance function to $\partial M_s^{(k)}$ on $M_s^{(k)}$. Otherwise, fix a point $x^{(k)}\in M_s^{(k)}$, and $d$ is the distance function to $x^{(k)}$ on $M_s^{(k)}$.
The main result of this section is the following proposition.
\begin{prop} \label{prop: exponential decay on symplectic ends}
There exist constants $z, z', r_0$ such that the following holds. Suppose $r>r_0$ and $(A,\phi)\in\mathcal{C}_k(X,\frs)$ solves \eqref{SW}, then there is a constant $d_0$, which may depend on $r$ and $(A,\phi)$, such that
\begin{equation}\label{eqn_exponential_decay_Er}
E_r(A,\phi)(x)<z e^{-\sqrt{r}\cdot(d(x)-d_0)/z'}
\end{equation}
for every $x\in M_s$ with $d(x)>d_0$.
\end{prop}

We start the proof with the following lemma, which is adapted from \cite[Lemma 3.21]{kronheimer1997monopoles}.
\begin{lem} \label{lem: pointwise convergence}
Let $(A,\phi)$ be as in Proposition \ref{prop: exponential decay on symplectic ends}, then given $\delta>0$, there exists $d(\delta)>0$ depending on $(A,\phi)$, $r$ and $\delta$, such that for all $x\in M_s$ with $d(x)>d(\delta)$, we have
$$E_r(A,\phi)(x) <\delta.$$
\end{lem}

\begin{proof}
Assume the contrary, then there is a sequence $\{x_n\}_{n\ge 1}\subset M_s$ and a constant $\delta>0$, such that $d(x_n)\to +\infty$ and $E_r(A,\phi)(x_n)\ge\delta$ for all $n$. Let $\epsilon_0>0$ be given by \eqref{eqn_def_epsilon0}. After taking a subsequence of $\{x_n\}$ if necessary, we may assume that the balls $B_{x_n}(\epsilon_0)$ are pairwise disjoint and are all included in $M_s$. Let $g$ be the metric of $X$. Consider the sequence $(M_s,g,x_n,\frs, A,\phi)$. By Proposition \ref{prop: properness} and Lemma \ref{C0 bound on the negative end}, a subsequence converges to a limit 
$(\widetilde{M_s},\tilde{g},\tilde{x},\tilde{\frs},\tilde{A},\tilde{\phi})$. Recall that $\nabla^k \omega$
is bounded for all $k$. By a diagonal argument and the Arzel\`a-Ascoli theorem, after taking a further subsequence, the symplectic form $\omega$ converges to a limit symplectic form $\tilde{\omega}$ on $\tilde{M}_s$, in the $C^\infty$ topology on compact subsets. The symplectic form $\tilde{\omega}$ is compatible with $\tilde{g}$, and hence it defines an energy density function $\widetilde{E_r}(\tilde{A},\tilde{\phi})$ on $\tilde{M}_s$. By the assumptions on $x_n$, we have $\widetilde{ E_r}(\tilde{A},\tilde{\phi})(\tilde{x})\ge \delta$, thus 
$$
\int_{B_{\tilde{x}}(\epsilon_0)}\tilde{E}_r(\tilde{A},\tilde{\phi}) >0.
$$
Therefore, there exists a positive constant $\delta'>0$, such that 
$$\int_{B_{x_n}(\epsilon_0)} E_r(A,\phi) > \delta'$$ for sufficiently large $n$. This contradicts the assumption that $$\int_{M_s}E_r(A,\phi)<+\infty.\phantom\qedhere\makeatletter\displaymath@qed$$
\end{proof}

The following lemma is an extension of \cite[Lemma 3.24]{kronheimer1997monopoles}. Recall that for a point $p$ in a complete Riemannian manifold $M$, we use $B_p(r)$ to denote the set of points in $M$ whose distance to $p$ is no greater than $r$, and $r$ is allowed to be greater than the injectivity radius of $M$ at $p$.

\begin{lem} \label{lem: working horse}
Let $K, v_0, R>0$, $r\ge 1$ be constants. Let $M$ be an $n$-dimensional complete Riemannian manifold with $\Ric \ge -K$, let $x_0\in M$. Let
 $s$ be a $C^2$ function on $B_{x_0}(R)$. Suppose $s$ satisfies:
$$
\frac{1}{2}d^*d\,s+rVs\le h,
$$
where $h,V$ are $C^0$ functions, and $V\ge v_0$ on $B_{x_0}(R)$. 
Then there exists a positive constant $\epsilon$ depending only on $n$, $K$, $R$, and $v_0$, such that the following inequality holds:
$$
 s({x_0}) \le \Big(\sup_{B_{x_0}(R)} \Big|\frac{h}{rV}\Big|\Big )+\Big(\sup_{\partial B_{x_0}(R)}|s|\Big) e^{-\epsilon R\sqrt{r}}.
$$
If $\partial{B_{x_0}(R)}=\varnothing$, then $\sup_{\partial B_{x_0}(R)}|s|$ in the above inequality is defined to be $0$.
\end{lem}

\begin{proof}
Let $\rho$ be the distance function to ${x_0}$, let $k=\sqrt{K/(n-1)}$. By the distributional Laplacian comparison theorem, the following inequality holds on $X$ in the sense of distributions:
\begin{equation}\label{eqn_distributional_laplacian_comparison}
\Delta \rho\le \frac{n-1}{\rho}(1+k\,\rho).
\end{equation}
In other words, for every non-negative function $\varphi\in C_0^\infty (X)$, we have
$$
\int_X \rho \,\Delta \varphi \le \int \frac{n-1}{\rho}(1+k\,\rho)\varphi.
$$
The reader may refer to \cite[Corollary 2.12]{mantegazza2014distributional} for the proof of \eqref{eqn_distributional_laplacian_comparison} (see also \cite{ishii1995equivalence} and \cite[Theorem 3]{calabi1958extension}).

 Let $f(u)$ be a smooth, non-decreasing function on $\mathbb{R}$ such that $f(u)=0$ when $u\le R/4$ and $f(u)=u$ when $u\ge R/2$. Let $g=e^{\epsilon\sqrt{r}f(\rho)}$ be a function on $M$, where $\epsilon$ is a small positive constant that will be determined later. Notice that in the sense of distributions,
\begin{align*}
 d^*dg=-\Delta{g} &=-(\epsilon\sqrt{r}f''(\rho)+\epsilon^2 r(f'(\rho))^2+\epsilon \sqrt{r}f'(\rho)\Delta \rho)\,g \\
   & \ge -\Big(\epsilon\sqrt{r}f''(\rho)+\epsilon^2 r(f'(\rho))^2+\epsilon \sqrt{r}f'(\rho)\big(\frac{n-1}{\rho}+k(n-1)\big)\Big)\,g.
\end{align*}
Therefore, there exists a constant $\epsilon$ depending only on $n$, $K$, $R$, and $v_0$, such that  
$$\frac{1}{2}d^*dg+rVg\ge 0$$
in the sense of distributrions. Let 
$$
\tilde g = \sup_{B_{x_0}(R)} \Big|\frac{h}{rV}\Big| + \Big(\sup_{\partial{B_{x_0}(R)}}|s|\Big) \cdot g/ e^{\epsilon\sqrt{r}\,R},
$$
then $\frac{1}{2}d^*ds+rVs\le \frac{1}{2}d^*d\tilde{g}+rV\tilde{g}$ in the sense of distributions, and $s|_{\partial{B_{x_0}(R)}}
\le \tilde{g}|_{\partial{B_{x_0}(R)}}$.
By the maximum principle for weak solutions \cite[Theorem 8.1]{gilbarg2001elliptic}, we have $s\le g$ on the ball $B_{x_0}(R)$, hence the lemma is proved.
\end{proof}

\begin{proof}[Proof of Proposition \ref{prop: exponential decay on symplectic ends}]
Recall that we use the notation $z_i$ to denote constants that only depend on $X, M_s, M_c, \theta$, and the terms $\hat \tau$ and $\eta$ in \eqref{eqn: definition of perturbation}. In the following we will require $r_0\ge 1$.
The proof follows the strategy of \cite[Section 3]{kronheimer1997monopoles}, and is divided into 7 steps:

\paragraph{\bf Step 1.}
By Lemma \ref{lem: pointwise convergence}, there exists $d_1>0$ depending on $(A,\phi)$ and $r$, such that if $x\in M_s$ satisfies $d(x)>d_1$ then
\begin{equation} \label{eqn: property of T0}
|\alpha(x)|>\frac{1}{2},\quad E_r(A,\phi)(x)<1.
\end{equation}

\paragraph{\bf Step 2: Pointwise estimates of $\alpha$ and $\beta$.}
By \cite[Lemma 2.2]{taubes1996sw}, there exist constants $z_1,\,z_2,\,z_3\ge 1$, such that if $\zeta\in(0,\frac{r}{2z_1z_2})$, $r>z_1$, and $\delta > z_3$, let
\begin{equation}\label{eqn_deinition_u_from_alpha_beta}
u=(1-|\alpha|^2)-\zeta|\beta|^2+\frac{\delta}{\zeta r},
\end{equation}
then the following inequality holds:
$$
\frac{1}{2}d^*du+\frac{r}{4}|\alpha|^2u\ge 0.
$$
Notice that the $C^0$ norm of $u$ is bounded by \eqref{eqn: property of T0}, hence by Lemma \ref{lem: working horse} and \eqref{eqn: property of T0}, there are constants $z_5, z_6$ such that
\begin{equation} \label{eqn_pointwise_lower_bound_u}
u\ge -z_5e^{-\sqrt{r}\cdot (d-d_1)/{z_6}}, 
\end{equation}
on $\{x\in M_s|d(x)>d_1+1\}$.
Therefore there exists a constant $z_7$ such that 
\begin{align}
|\alpha|^2 &\le 1+ \frac{z_7}{r^2} \label{eqn: C0 of alpha},\\
|\beta|^2&\le \frac{z_7}{r}\big(1-|\alpha|^2+\frac{z_7}{r^2}\big), \label{eqn: C0 of beta}
\end{align}
on $\{x\in M_s|d(x)>d_1+1\}$.

\paragraph{\bf  Step 3: Pointwise estimates of $F_a$.}
On $M_s$, the curvature part of \eqref{SW} can be rewritten as (cf. \cite[(8),(9)]{kotschick1995seiberg})
\begin{equation} \label{eqn: F_a by tensor}
F_a^+=-\frac{i}{8}r\cdot\big(1-|\alpha|^2+|\beta|^2)\omega+\frac{ r}{4}(\alpha^*\beta-\alpha\beta^*).
\end{equation}
By \eqref{eqn: C0 of alpha} and \eqref{eqn: C0 of beta}, there exists a constant $z_{11}$ such that 
$$
|F_a^+| \le \frac{r}{4\sqrt{2}}(1+\frac{z_{11}}{r})(1-|\alpha|^2)+z_{11}.
$$

Now we estimate $|F_a^-|$. By \cite[Lemma 2.5]{taubes1996sw}, there exist constants $z_{12},\,z_{13}$, $z_{14}$, $z_{15}$ such that if $r>z_{15}$, then for
\begin{align*}
q_0 & =\frac{r}{4\sqrt{2}}(1+\frac{z_{12}}{r})(1-|\alpha|^2)-z_{13}\cdot r|\beta|^2+z_{14},\\
s &= |F_a^-|,
\end{align*}
we have
$$
\frac{1}{2}d^*d(s-q_0)+\frac{r}{4}|\alpha|^2(s-q_0)\le |\mathcal{R}|s,
$$
where $\mathcal{R}$ is a curvature term that is uniformly bounded on $M_s$.

Therefore, if $r>8\sup |\mathcal{R}|$, we have
$$
\frac{1}{2}d^*d(s-q_0)+\frac{r}{8}|\alpha|^2(s-q_0)\le |\mathcal{R}|\cdot |q_0|.
$$
By Lemma \ref{lem: working horse} and \eqref{eqn: property of T0}, there exists a constant $z_{16}$ such that on $\{x\in M_s|d(x)>d_1+1\}$, we have 
$$
|F_a^-|\le \frac{r}{4\sqrt{2}}(1+\frac{z_{16}}{r})(1-|\alpha|^2)+z_{16}.
$$
In conclusion, there is a constant $z_{17}$ such that on $\{x\in M_s|d(x)>d_1+1\}$,
$$
|F_a^{\pm}|\le \frac{r}{4\sqrt{2}}(1+\frac{z_{17}}{r})(1-|\alpha|^2)+z_{17}.
$$

\paragraph{\bf Step 4: Pointwise estimates of $|\nabla_a\alpha|$ and
$|\nabla_A'\beta|$.}
Let $$y=|\nabla_a\alpha|^2+r|\nabla_A'\beta|^2.$$
Recall that the function $u$ is defined by \eqref{eqn_deinition_u_from_alpha_beta}.
By \cite[(2.43)]{taubes1996sw}, there exists a constant $z_{18}$ such that
$$
\frac{1}{2}d^*d(y-z_{18}\cdot r\cdot u)+\frac{r}{4}|\alpha|^2(y-z_{18}\cdot r\cdot u) \le 0.
$$
By \eqref{eqn: C0 of alpha}, \eqref{eqn: C0 of beta}, and Lemma \ref{lem: working horse}, therer exists a constant $z_{19}$ such that
$$
|\nabla_a\alpha|^2+r|\nabla_A'\beta|^2 =y \le z_{19}\cdot r \cdot (1-|\alpha|^2)+z_{19}.
$$

\paragraph{\bf Step 5: Exponential decay of $|\nabla_a\alpha|$, $|\nabla_A'\beta|$, and $|\beta|$.} Let
$$
y_1=|\nabla_a\alpha|^2+\frac{r}{32}|\nabla_A'\beta|^2+\frac{r^2}{16\,z_{20}}|\beta|^2.
$$
By \cite[(4.15)]{taubes1996sw},
one can choose $z_{20}$ sufficiently large, such that there exists a constant $z_{21}$, such that\footnote{The derivation of \cite[(4.15)]{taubes1996sw} only used the pointwise estimates of $\alpha$, $\beta$, $F_a$, $\nabla_a\alpha$ and $\nabla_A '\beta$ from \cite[Section 2]{taubes1996sw}, and it does not depend on the refined pointwise estimate of $F_a^-$ developed in \cite[Section 3d]{taubes1996sw}. Therefore, the inequalities obtained from Step 2 to Step 4 are sufficient for deriving \eqref{eqn: step 5}.}
\begin{equation}\label{eqn: step 5}
\frac{1}{2}d^*dy_1+\frac{r}{4}|\alpha|^2y_1\le \big(z_{21}\cdot r\cdot(1-|\alpha|^2)+\frac{r}{8})y_1.
\end{equation}
By Lemma \ref{lem: pointwise convergence}, there exists a constant $d_2$ such that on $\{x\in M_s|d(x)>d_2\}$,
$$|1-|\alpha|^2|<\min \Big\{\frac{1}{16\,z_{21}},\frac{1}{8}\Big\}.$$
Then \eqref{eqn: step 5} implies that on $\{x\in M_s|d(x)>d_2\}$,
$$
\frac{1}{2}d^*dy_1+\frac{r}{32}y_1 \le 0.
$$
By Lemma \ref{lem: working horse}, there are constants $z_{22}$, $z_{23}$ such that on $\{x\in M_s|d(x)>d_2+1\}$,
\begin{equation} \label{decay estimate step 5}
y_1<z_{22}\cdot e^{\sqrt{r}\cdot(d-d_2)/z_{23}},
\end{equation}

\paragraph{\bf Step 6: Exponential decay of $|1-|\alpha|^2|$.}
By \cite[(2.3)]{taubes1996sw},
$$
\frac{1}{2}d^*d|\alpha|^2+|\nabla_a\alpha|^2+\frac{r}{4}|\alpha|^2(|\alpha|^2-1+|\beta|^2)
+ \alpha \boxtimes \nabla_A'\beta + \alpha \boxtimes \beta=0,
$$
where $\boxtimes$ are pointwise bilinear operators defined by the metric and the symplectic form.
A straight forward calculation shows
\begin{align*}
\frac{1}{4}d^*d|1-|\alpha|^2|^2 =& \big(\frac{1}{2}d^*d
                                                (1-|\alpha|^2)\big)
                                    \cdot (1-|\alpha|^2)
                                    -\frac{1}{2}|\nabla_a|\alpha|^2|^2 \\
                                 =& -\frac{r}{4}|\alpha|^2
                                     |1-|\alpha|^2|^2
                                 +|\nabla_a\alpha|^2\cdot(1-|\alpha|^2)\\
                                 &\quad+\frac{r}{4}|\alpha|^2|\beta|^2
                                     (1-|\alpha|^2)
                                  +(1-|\alpha|^2)\cdot(\alpha\boxtimes
                                      \nabla_A'\beta + \alpha\boxtimes
                                          \beta).
\end{align*}
The equation above and \eqref{decay estimate step 5} imply there are constants $z_{24}, z_{25}$, such that on $\{x\in M_s|d(x)>d_2+1\}$,
$$
\frac{1}{4}d^*d|1-|\alpha|^2|^2+\frac{r}{4}|\alpha|^2|1-|\alpha|^2|^2 \le z_{24}\cdot e^{(d-d_2)\sqrt{r}/z_{25}}.
$$
By Lemma \ref{lem: working horse}, there exist constants $z_{26}, z_{27}$ such that 
\begin{equation}\label{decay estimate step 6}
|1-|\alpha|^2|^2 < z_{26}\cdot e^{(d-d_2)\sqrt{r}/z_{27}},
\end{equation}
on $\{x\in M_s|d(x)>d_2+1\}$.

\paragraph{\bf Step 7: Exponential decay of $|F_a|$.}
The exponential decay for $|F_a^+|$ follows from \eqref{eqn: F_a by tensor}, \eqref{decay estimate step 5} and \eqref{decay estimate step 6}. Recall that $s=|F_a^-|$. By \cite[(2.19)]{taubes1996sw}, there exists a constant $z_{28}$ such that
\begin{multline*}
\frac{1}{2}d^*ds+\frac{r}{4}(|\alpha|^2+|\beta|^2)s\le |\mathcal{R}|s+
            \frac{r}{4\sqrt{2}}(|\nabla_a\alpha|^2+|\nabla_A'\beta|^2)\\
            +z_{28}\cdot r(|\alpha||\beta|+|\alpha||\nabla_A'\beta|
            +|\beta||\nabla_a\alpha|+|\beta|^2).
\end{multline*}
Therefore \eqref{decay estimate step 5} shows that there exist constants $z_{29}, z_{30}$, such that
 if $|\alpha|>7/8$, $r>16\sup|\mathcal{R}|$, and $d(x)>d_2+1$, then
$$
\frac{1}{2}d^*ds+\frac{r}{8}|\alpha|^2s\le z_{29}\cdot e^{(d-d_2)\sqrt{r}/z_{30}},
$$
By Lemma \ref{lem: working horse}, this implies there are constants $z_{31},z_{32}$, and a positive real number $d_3$ which may depend on $(A,\phi)$, such that
\begin{equation}\label{decay estimate step 7}
s< z_{31}\cdot e^{(d-d_3)\sqrt{r}/z_{32}}
\end{equation}
on $\{x\in M_s|d(x)>d_3\}$.

The proposition then follows from \eqref{decay estimate step 5}, \eqref{decay estimate step 6}, and \eqref{decay estimate step 7}.
\end{proof}

\section{Uniform exponential decay of $E_r(A,\phi)$}\label{sec: uniform exponential decay}
This section shows that the constant $d_0$ in Proposition \ref{prop: exponential decay on symplectic ends} can be chosen to depend only on $r$, not on the solution $(A,\phi)$. Let $X,Z,M_s,M_c,\frs,\bS,\omega,\theta, A,\phi,A_0$ be as in Section \ref{sec: exponential decay}. Recall that $z_i$ denotes constants that only depend on $X, M_s, M_c, \theta$, and the terms $\hat \tau$ and $\eta$ in \eqref{eqn: definition of perturbation}. The constant $r_0$ is a positive real number that depends on the same set of data, and the value of $r_0$ may increase as the proof proceeds. We will require $r > r_0$  in \eqref{eqn: perturbation on symplectic ends}.

\subsection{An energy identity on $M_s$}
Recall that $\phi|_{M_s}$ decomposes as $\phi=\sqrt{r}(\alpha+\beta)$, and $a=A|_{M_s}-A_0$.
For $F\in \Lambda^2T^*M_s\otimes\mathbb{C}$, define $F^\omega=\frac{1}{2}\langle\omega,F\rangle\in \mathbb{C}$.
The following lemma is a rescaled version of \cite[Equation (18)]{kronheimer1997monopoles}.
\begin{lem} \label{lem: integration by parts}
Let $\chi$ be a smooth cut-off function on $M_s$ such that $\supp\chi$ is contained in the interior of $M_s$, and $\chi =1$ for all $x\in M_s$ with $d(x)>1$. Then we have
\begin{multline} \label{eqn: integration by parts}
\int_{M_s}\Big( \frac{r}{2}|\bar{\partial}_a(\chi\alpha) + \bar{\partial}_a^*(\chi\beta)|^2+
2|iF_a^\omega-\frac{r}{8}(1-|\chi\alpha|^2+|\chi\beta|^2)|^2+2|F^{0,2}_a-\frac{r}{4}{(\chi\alpha)}^*(\chi\beta)|^2\\
+\frac{r}{2}iF_a^\omega-2|iF_a^\omega|^2-2|F_a^{0,2}|^2 \Big)\\
= \int_{M_s}\Big( \frac{r}{4}|\nabla_a(\chi\alpha)|^2+\frac{r}{4}|\nabla_{A_1+a}(\chi\beta)|^2+\frac{r}{2}(iF_{A_1}^\omega)|\chi\beta|^2 \\
+\frac{r^2}{32}(1-|\chi\alpha|^2-|\chi\beta|^2)^2+\frac{r^2}{8}|\chi\beta|^2-rRe\langle N \circ \partial_a(\chi\alpha),\chi\beta\rangle\Big).
\end{multline}
Where $A_1$ is the unique unitary connection on $T^{0,2}M_s$ such  that $\nabla_{A_1}^{1,0}=\partial$, and $N:T^{1,0}M_s\to T^{0,2} M_s$ is the Nijenhuis tensor.
\end{lem}

\begin{rmk}
If $\partial M_s=\emptyset$, we may take $\chi=1$ on $M_s$.
\end{rmk}

\begin{proof}
The identity follows from and integration by parts Weitzenb\"ock formulas.

For a constant $d_0>1$, let $\chi_{d_0}$  be a smooth function on $M_s$ such that $\chi_{d_0} = 1$ for all $x$ with $d(x)\le d_0$, and $\chi_{d_0} = 0$ for all $x$ with $d(x)\ge d_0+2$, and $|\nabla\chi_{d_0}|\le 1$, $|\chi_{d_0}|\le 1$. Then integration by parts yields
$$
\int_{M_s} \langle \chi_{d_0}\bar{\partial}_a(\chi\alpha) , \bar{\partial}_a^*(\chi\beta) \rangle 
= \int_{M_s} \langle \bar{\partial}_a\big(\chi_{d_0}\bar{\partial}_a(\chi\alpha)\big), \chi\beta \rangle. 
$$
On the other hand, there exists a constant $z_1$ such that
\begin{align*}
& \Bigg|
\int_{M_s} \langle \chi_{d_0}\bar{\partial}_a(\chi\alpha) , \bar{\partial}_a^*(\chi\beta) \rangle
  -
  \int_{\{x\in M_s|d(x)\le d_0+2\}} \langle \bar{\partial}_a(\chi\alpha) , \bar{\partial}_a^*(\chi\beta) \rangle
  \Bigg| 
\\
\le
&  \int_{\{x\in M_s|d_0\le d(x)\le d_0+2\}} |\bar{\partial}_a(\chi\alpha) |\cdot|\bar{\partial}_a^*(\chi\beta)| 
\\
\le &
z_1 \int_{\{x\in M_s|d_0\le d(x)\le d_0+2\}} E_r(A,\phi),
\end{align*}
and a constant $z_2$ such that
\begin{align*}
& \Bigg|
\int_{M_s} \langle \bar{\partial}_a\big(\chi_{d_0}\bar{\partial}_a(\chi\alpha)\big), \chi\beta \rangle
 - 
 \int_{\{x\in M_s|d(x)\le d_0+2\}} \langle\bar{\partial}_a\bar{\partial}_a(\chi\alpha), \chi\beta\rangle\Bigg|
\\
\le &\int_{\{x\in M_s|d_0\le d(x)\le d_0+2\}} |\nabla\chi|\cdot |\bar{\partial}_a(\chi\alpha)|\cdot|\chi\beta|+\langle\bar{\partial}_a\bar{\partial}_a(\chi\alpha), \chi\beta\rangle
\\
\le & z_2 \int_{\{x\in M_s|d_0\le d(x)\le d_0+2\}} E_r(A,\phi) + E_r(A,\phi)^{1/2}\cdot |\bar{\partial}_a\bar{\partial}_a(\chi\alpha)|.
\end{align*}

Let $\epsilon_0$ be the constant defined by \eqref{eqn_def_epsilon0}. 
By elliptic bootstrapping, there exist constants $z_3, z_4$ such that for all $x$ with $d(x)>\epsilon_0$,
$$
|\bar{\partial}_a\bar{\partial}_a(\chi\alpha)(x)| \le z_3 \Big(\int_{B_x(\epsilon_0)} E_r(A,\phi) + r+1\Big)^{z_4}.
$$
Let $d_0\to +\infty$, and suppose $r_0$ is sufficiently large.
It then follows from Proposition \ref{prop: exponential decay on symplectic ends}, the Bishop-Gromov volume comparison theorem, and the estimates above that 
\begin{equation}\label{eqn_integration_by_parts_no_boundary_term}
\int_{M_s} \langle \bar{\partial}_a(\chi\alpha) , \bar{\partial}_a^*(\chi\beta) \rangle 
= \int_{M_s} \langle \bar{\partial}_a\bar{\partial}_a(\chi\alpha), \chi\beta \rangle.
\end{equation}
 Similarly, for $r_0$ sufficiently large, we have the following identities:
\begin{align*}
&\int_{M_s} \langle \bar{\partial}_a(\chi\alpha) , \bar{\partial}_a^*(\chi\beta) \rangle 
= \int_{M_s} \langle \bar{\partial}_a\bar{\partial}_a(\chi\alpha), \chi\beta \rangle , \\
&\int_{M_s} \langle \bar{\partial}_a(\chi\alpha) , \bar{\partial}_a(\chi\alpha) \rangle 
= \int_{M_s} \langle \bar{\partial}_a^*\bar{\partial}_a(\chi\alpha), \chi\alpha \rangle ,\\
&\int_{M_s} \langle \bar{\partial}_a^*(\chi\beta) , \bar{\partial}_a^*(\chi\beta) \rangle 
= \int_{M_s} \langle \bar{\partial}_a\bar{\partial}_a^*(\chi\beta), \chi\beta \rangle, \\
&\int_{M_s} \langle \nabla_a(\chi\alpha) , \nabla_a(\chi\alpha) \rangle =
\int_{M_s} \langle \nabla_a^*\nabla_a(\chi\alpha), \chi\alpha \rangle, \\
&\int_{M_s} \langle \nabla_{A_1+a}(\chi\beta) , \nabla_{A_1+a}(\chi\beta) \rangle =
\int_{M_s} \langle \nabla_{A_1+a}^*\nabla_{A_1+a}(\chi\beta), \chi\beta \rangle .
\end{align*}

On the other hand, by the Weitzenb\"ock formulas \cite[(12), (13)]{kotschick1995seiberg},
\begin{align*}
\bar{\partial}_a^*\bar{\partial}_a(\chi\alpha) &=\frac{1}{2}(\nabla^*_a\nabla_a(\chi\alpha)-2iF_a^\omega (\chi\alpha)),\\
\bar{\partial}_a\bar{\partial}_a^*(\chi\beta) &=\frac{1}{2}(\nabla_{A_1+a}^*\nabla_{A_1+a}(\chi\beta)+2iF_{A_1+a}^\omega(\chi\beta)).
\end{align*}

 The lemma is then proved by a straightforward computation using the identities above and $\bar{\partial}_a^2(\chi\alpha) = F_a^{0,2}(\chi\alpha) - N\circ \partial_{a}(\chi\alpha)$.
\end{proof}

\subsection{Uniform energy bound}
\begin{prop}\label{prop: energy bound}
There exists a constant $r_0$, such that for all $r>r_0$, there exists a constant $C$ which may depend on $r$ with the following property. For all $(A,\phi)\in \moduli{k}$ that solves \eqref{SW}, we have
\begin{equation}\label{eqn: energy bound}
\int_{M_s}E_r(A,\phi)<C.
\end{equation}
\end{prop}

\begin{rmk}
A similar energy estimate was proved by \cite[Lemma 3.17]{kronheimer1997monopoles} for AFAK ends. The last paragraph of the proof of \cite[Lemma 3.17]{kronheimer1997monopoles} claimed that $$\int_{\partial K_3} a\wedge \omega$$ has a uniform bound without detailed explanation, and the detailed proof of this estimate was given by \cite{mrowka2006legendrian} after the proof of Lemma 2.2.7. However, the constant $C$ in the argument of \cite{mrowka2006legendrian} depends on the volume of the complement of the AFAK end.  If one applies the same argument from \cite{mrowka2006legendrian} to Proposition \ref{prop: energy bound} above, then the constant $C$ would be given by the volume of the complement of $M_s$, which is infinity when $M_c\neq \emptyset$. Therefore, the arguments in \cite{kronheimer1997monopoles} and \cite{mrowka2006legendrian} do not suffice in the context of this article.
\end{rmk}

If $M_c$ is non-empty, suppose $M_c=(-\infty,0]\times Y$,  let $t$ be the function on $X$ which is equal to the  projection to $(-\infty,0]$ on $M_c$, and is equal to zero on $M-M_c$. Let $\frt$ be the $\spinc$ structure on $Y$ induced by $\frs|_{M_c}$.

Suppose $\mathfrak{a}\in\mathcal{C}(Y,\mathfrak{t})$ is a critical point of the perturbed Chern-Simons-Dirac functional $\CSDpert=\mathcal{L}+\mathfrak{q}$ on $\mathcal{C}(Y,\frt)$, let $\gamma_{\mathfrak{a}}=(A_\mathfrak{a},\phi_\mathfrak{a})\in\mathcal{C}([-1,0]\times Y,\frs|_{[-1,0]\times Y})$ be the configuration on $[-1,0]\times Y$ which is in temporal gauge and represents the constant path at $\mathfrak{a}$. 
Recall that in Section \ref{subsection: perturbation}, the perturbation $\mathfrak{q}$ on $M_c$ is required to satisfy  $\|\mathfrak{q}\|_{\hat{\mathcal{P}}}\le 1$, 
where $\|\cdot\|_{\hat{\mathcal{P}}}$ is defined by \eqref{eqn_def_norm_strongly_tame_perturbations}. By \cite[Section 10.7]{kronheimer2007monopoles}, there is a constant $z_0$ such that
\begin{equation} \label{eqn: critical points being bounded}
\|F_{A_\mathfrak{a}^t}\|_{L^2}^2<z_0
\end{equation}
for all critical points $\mathfrak{a}$.

Choose a gauge representative of $(A,\phi)$ that is in temporal gauge on the cylindrical end $M_c$. Recall that $A_0$ is the canonical $\spinc$ connection on $M_s$. Extend $A_0$ to a smooth $\spinc$ connection on $(X,\frs)$, such that $A_0$ is in temporal gauge and is translation invariant on $M_c$. Let $a=A-A_0$, then $F_a=\frac12(F_{A^t}-F_{A_0^t})$.
By \eqref{eqn: critical points being bounded}, there exists $R_0>1$ depending on $(A,\phi)$, such that
\begin{equation}\label{eqn: L2 bound of curvature at R}
\int_{t(x)\in[-R_0-1,-R_0]} |F_a(x)|^2 <\frac12\,\Big( z_0+\int_{t(x)\in [-R_0-1,-R_0]}|F_{A_0^t}|^2+1\Big).
\end{equation}

\begin{lem}\label{lem: linear energy bound}
There are constants $z, r_0>0$ and a function $T:(\mathbb{R}^+)^2\to \mathbb{R}^+$ which depends on $X,M_c,M_s,\theta$ and the terms $\hat \tau$ and $\eta$ in  \eqref{eqn: definition of perturbation}, with the following property. Suppose $r>r_0$, and suppose there are constants $R>0, \kappa>0$ such that
\begin{equation} \label{eqn: assume L2 bound of curvature at R}
\int_{t(x)\in[-R-1,-R]} |F_a|^2 \le \kappa,
\end{equation}
then the following inequalities hold:
\begin{align}
&\int_{M_s} E_r(A,\phi) < T(\kappa,r)+z\,r^2\cdot R
            \label{eqn: linear energy bound on W-},\\
&\int_{[-R,0]\times Y} |F_a|^2 < T(\kappa,r)+z\,r^2\cdot R.  \label{eqn: linear energy bound on tube}
\end{align} 
\end{lem}

\begin{proof}
We use $T_i$ to denote the constants that may depend on $\kappa,r$ but are independent of $(A,\phi)$.

Recall that $F_a^\omega=\frac{1}{2}\langle\omega,F_a\rangle\in i\mathbb{R}$.
On $M_s$, equation \eqref{SW} decomposes as
\begin{align*}
& \bar\partial_a\alpha+\bar\partial_a^*\beta =0, 
\\
& F_a^\omega = -\frac{ir}{8}(1-|\alpha|^2+|\beta|^2), 
\\
& F_a^{0,2}=\frac{r}{4}\alpha^*\beta.
\end{align*}
By Proposition \ref{uniform C0 bound}, Lemma \ref{lem: integration by parts}, and the equations above, there exists a constant $T_1$ depending on $r$ such that
\begin{multline}
T_1+ \int_{M_s} \big(\frac{r}{2}iF_a^\omega-2|iF_a^\omega|^2-2|F_a^{0,2}|^2\big)\ge\int_{M_s}\Big( \frac{r}{4}|\nabla_a\alpha|^2
 +\frac{r}{4}|\nabla_{A_1+a}\beta|^2
 \\
 +\frac{r}{2}(iF_{A_1}^\omega)|\beta|^2 
+\frac{r^2}{32}(1-|\alpha|^2-|\beta|^2)^2 
+\frac{r^2}{8}|\beta|^2-rRe\langle N \circ \partial_a\alpha,\beta\rangle\Big).
\label{eqn_energy_inequality_before_rearrange}
\end{multline}
Suppose $r_0$ is sufficiently large, then for all $r>r_0$, we have
\begin{align}
\frac{r}{8}|\nabla_a\alpha|^2 + \frac{r^2}{32}|\beta|^2 &\ge |rRe\langle N \circ \partial_a\alpha,\beta\rangle|,
\label{eqn_control_Nijenhuis_tensor_in_energy_identity_for_rearrangement}
\\
\frac{r}{4}|\nabla_{A_1+a}\beta|^2 + \frac{r^2}{32}|\beta|^2 &\ge |\nabla_A'\beta|^2 + |\frac{r}{2}(iF_{A_1}^\omega)|\cdot |\beta|^2.
\label{eqn_control_beta_terms_in_energy_identity_for_rearrangement}
\end{align}
Therefore by \eqref{eqn_energy_inequality_before_rearrange}, for $r_0$ sufficiently large, we have
\begin{equation} \label{eqn: bound energy by integration by parts}
\int_{M_s}|1-|\alpha|^2-|\beta|^2|^2+|\beta|^2+|\nabla_a \alpha|^2+|\nabla_A'\beta|^2+|F_a^+|^2\le T_1+ \int_{M_s}\frac{r}{2}iF_a^\omega.
\end{equation}

By \eqref{eqn: assume L2 bound of curvature at R}, \cite[Lemma 5.1.2]{kronheimer2007monopoles}, and Coulomb gauge fixing, there exists a unitary connection $a'$ of the trivial $\mathbb{C}$--bundle on $\{x|t(x)\in[-R-1,-R]\}$, such that:
\begin{enumerate}
\item $\|a-a'\|_{L_1^2([-R-1,-R]\times Y)} < T_2$, for some constant $T_2$ depending on $\kappa$,
\item $a'=a$ when $t\in[-R-\frac{1}{3},-R]$,
\item $F_{a'}=0$, when $t\in[-R-1,-R-\frac{2}{3}]$.
\end{enumerate}
Extend $a'$ to $\{x|t(x)> -R\}$ by taking $a'=a$ when $t> -R$.

Recall that by Lemma \ref{lem: pointwise convergence}, there exists a constant $d_0$, which may depend on $(A,\phi)$, such that $|\alpha(x)|\ge \frac12$ when $d(x)\ge d_0$. Also recall that $\Phi_0$ is the canonical section of $\bS^+|_{M_s}$ given by $1\in\Gamma(M_s,T^{0,0}M_s)$. Therefore we can take a gauge representative of $(A,\phi)$ such that $\alpha\in\mathbb{R}\cdot\Phi_0$ when $d(x)\ge d_0$. Without loss of generality, assume $(A,\phi)$ satisfies the above property. Notice that $|\alpha(x)|\ge \frac12$ and $\alpha\in\mathbb{R}\cdot\Phi_0$ imply $|\nabla_a\alpha|\ge \frac12|a|$. Therefore by Proposition \ref{prop: exponential decay on symplectic ends}, there exist constants $z_1$, $z_2$, $d_1>0$, where $d_1$ may depend on $(A,\phi)$, such that 
\begin{equation}\label{eqn: exponential decay of a}
|a'|=|a|\le 2|\nabla_a\alpha|\le z_1 e^{-\sqrt{r}\cdot(d-d_1)/z_2} 
\end{equation}
for all $x\in M_s$ with $d(x)>d_1$.  

Recall that by Condition (3) of Definition \ref{def_cylindrical_ESBG_end}, the ESBG structure can be extended to a neighborhood of $M_s$, therefore we can smoothly extend $\theta$ to a smooth 1-form on $X$ such that $\theta=0$ outside an open neighborhood of $M_s$. Extend $\omega$ to $X$ by taking $\omega=d\theta$. The extensions of $\theta$ and $\omega$ do not depend on $r$ or $(A,\phi)$.

The region $\{x|t(x)\ge -R-1\}$ is the union of the cylinder $[-R-1,0]\times Y$ and $X-M_c$. For $r$ sufficiently large, \eqref{eqn: exponential decay of a} implies the following identities via the same argument as the proof of \eqref{eqn_integration_by_parts_no_boundary_term}. 
\begin{align}
&\int_{t\ge -R-1}F_{a'}\wedge \omega =
\int_{t\ge -R-1}F_{a'}\wedge d\theta=-\int_{t\ge -R-1}dF_{a'}\wedge \theta =0, 
    \label{eqn: integral of F wedge Omega is zero}\\
&\int_{t\ge -R-1}F_{a'}\wedge F_{a'} =\int_{t\ge -R-1}F_{a'}\wedge d{a'}=-\int_{t\ge -R-1}dF_{a'}\wedge a'=0.
\label{eqn: integral of F wedge F is zero}
\end{align}
Let $Z_R=\{x\in X|t(x)\ge -R-1\}-M_s$. Then there exists a constant $z_3$ such that
\begin{equation}\label{eqn_vol_ZR_bounded_linear_R}
\Vol(Z_R)\le z_3+ R\cdot\Vol(Y).
\end{equation}
By \eqref{eqn: integral of F wedge Omega is zero},      
$$
\int_{M_s}F_{a'}\wedge\omega +\int_{Z_R}F_{a'}\wedge\omega=0,
$$
by \eqref{eqn: integral of F wedge F is zero},
$$
\int_{M_s\cup Z_R}|F_{a'}^+|^2 = \int_{M_s\cup Z_R}|F_{a'}^-|^2,
$$
and by \eqref{eqn: bound energy by integration by parts},
\begingroup
\allowdisplaybreaks
$$
\int_{M_s}(E_r(A,\phi)-|F_a^-|^2) \le T_1+\frac{r}{4}\,\Big|\int_{M_s}(iF_{a})\wedge\omega\Big| = T_1+\frac{r}{4}\,\Big|\int_{M_s}(iF_{a'})\wedge\omega\Big|. 
$$
Therefore, there exists a constant $z_4$, and constants $T_3,T_4$ that may depend on $r$, such that
\begin{align*}
 \int_{M_s}(E_r(A,\phi)-|F_a^-|^2) \le
            & T_1+\frac{r}{4}\,\Big|\int_{M_s}F_{a'}\wedge\omega\Big|\\
            = & T_1 + \frac{r}{4}\Big|\int_{Z_R}
                        F_{a'}\wedge\omega\Big|\\
            \le & T_1  +
                 \frac{r^2}{16}\int_{Z_R}|\omega|^2+ \frac{1}{4}\int_{Z_R}|F_{a'}|^2 \\
            \le & T_3 + \frac{1}{4}\int_{Z_R}|F_{a'}|^2\\
            \le & T_3 + \frac{1}{4}\int_{Z_R\cup M_s}|F_{a'}|^2\\
            = & T_3 + \frac{1}{2} \int_{Z_R\cup M_s}|F_{a'}^+|^2\\
            \le & T_3+\frac{1}{2} \int_{Z_R}|F_a^+|^2+
                \frac{1}{2}\int_{M_s}(E_r(A,\phi)-|F_a^-|^2)\\
            \le & T_4+z_4\,r^2\cdot R+
                \frac{1}{2}\int_{M_s}(E_r(A,\phi)-|F_a^-|^2),
\end{align*}
\endgroup
where the last inequality follows from Proposition \ref{uniform C0 bound} and \eqref{eqn_vol_ZR_bounded_linear_R}.
Hence 
\begin{equation}\label{eqn: bound on E-|F_-^2|}
\int_{M_s}(E_r(A,\phi)-|F_a^-|^2) \le 2\,T_4+2\,z_4r^2\, R.
\end{equation}
Recall that by the definition of $a'$, we have $\|a-a'\|_{L_1^2}< T_2$, where $T_2$ is a constant depending on $\kappa$. Therefore, there is a constant $z_5$, and a constant $T_5$ that depends on $r$, such that
\begin{align}\label{eqn: bound on int_(t< R+1)|F_a|^2}
\int_{Z_R\cup M_s}|F_a|^2 &\le T_2+ \int_{Z_R\cup M_s} |F_{a'}|^2 \nonumber\\
                        &= T_2+2\int_{Z_R\cup M_s}|F_{a'}^+|^2 
                            \nonumber\\
                        & \le T_2 + 2\int_{Z_R}|F_a^+|^2+
                            2\int_{M_s}(E_r(A,\phi)-|F_a^-|^2)\nonumber\\
                        & \le T_5 + z_5r^2\cdot R,
\end{align}
where the last inequality follows from Proposition \ref{uniform C0 bound}, \eqref{eqn_vol_ZR_bounded_linear_R}, and \eqref{eqn: bound on E-|F_-^2|}.
The lemma then follows immediately from \eqref{eqn: bound on E-|F_-^2|} and\eqref{eqn: bound on int_(t< R+1)|F_a|^2}.
\end{proof}

We also have the following estimate on $M_c$:
\begin{lem} \label{lem: midpoint energy bound}
There are constants $z_1$, $z_2$, and a function 
$$R_0:(\mathbb{R}^+)^2\to \{x\in\mathbb{R}|x>1\}$$ which depends only on $Y$, such that the following holds.
If $A,B,R>0$ satisfy 
\begin{align}
\int_{[-R,0]\times Y}|F_a|^2 &\le A+B\cdot R,
\label{eqn: assume linear L2 bound for F}
\\
|F_a^+|^2 &\le B\,\,\text{pointwise on }[-R,0]\times Y, 
\label{eqn: assume pointwise bound for F+}
\end{align}
and 
$$R >R_0(A,B),$$
then 
\begin{equation}
\int_{[-\frac{R}{2}-\frac{1}{2},-\frac{R}{2}+\frac{1}{2}]\times Y}|F_a|^2< z_1\cdot B+z_2.
\end{equation}
\end{lem}

\begin{proof}
Put $a$ in temporal gauge on $[-R,0]\times Y$, and write $a$ as a function $a(t)$ of $t$ which takes value in $\Gamma(Y, iT^*Y)$. Then 
\begin{align}
|F_a^+|=\frac{\sqrt{2}}{2}\,\big|\dot a(t) + *da(t)\big|, 
\label{eqn_Fa+_in_*da}\\
|F_a^-|=\frac{\sqrt{2}}{2}\,\big|\dot a(t) - *da(t)\big|.
\label{eqn_Fa-_in_*da}
\end{align}

Since $Y$ is closed, it follows from standard Hodge theory that the operator $*d$ is a closed, self-adjoint operator with a discrete spectrum on $L^2(Y, iT^*Y)$. Let 
$$\cdots < \lambda_{-3} < \lambda_{-2}< \lambda_{-1} <\lambda_0 =0 < \lambda_1 < \lambda_{2}<\lambda_{3}<\cdots$$
 be the eigenvalues of $*d$.
Let 
\begin{equation}\label{eqn_def_k_by_spectrum}
k_0=\max\Big\{\frac{1}{|\lambda_{-1}|},\frac{1}{|\lambda_1|},1\Big\}.
\end{equation}
Decompose $a$ as
$$
a(t)=\sum_{n=-\infty}^{+\infty} a_n(t),
$$
where $*da_n(t)=\lambda_n a_n(t)$. Let 
\begin{equation}\label{eqn_define_bn}
b_n(t)=\dot a_n(t)+\lambda_n a_n(t).
\end{equation}
By \eqref{eqn: assume pointwise bound for F+}, for all $t\in[-R,0]$ we have
$$\int_{\{t\}\times Y} |F_{a}^+|^2 \le B\cdot \Vol(Y),$$
hence by \eqref{eqn_Fa+_in_*da},
\begin{equation} \label{eqn_sum_bn_L2_upper_bound}
\sum_{n=-\infty}^{\infty} \|b_n(t)\|_{L^2}^2 \le 2B\cdot\Vol(Y).
\end{equation}
By \eqref{eqn: assume linear L2 bound for F}, 
\begin{align*}
\int_{-R}^0\|*da(t)\|^2_{L^2} \,dt&\le \frac{1}{2}\int_{-R}^0
                        \big(\|\dot a(t) + *da(t)\|^2_{L^2} 
                        + \|\dot a(t) - *da(t)\|^2_{L^2} \big)\,dt\\
                  &= \int_{[-R,0]\times Y} |F_a|^2 \le A+B\cdot R.
\end{align*}
For $R>1$, the above inequality implies
$$
\int_{-R}^{-R+1}\|*da(t)\|^2_{L^2} \le A+B\cdot R,
$$
hence there exits $t_1\in[-R,-R+1]$ such that 
\begin{equation}\label{eqn: M+CR}
\sum_{n}\lambda_n^2 \, \|a_n(t_1)\|_{L^2}^2 =\|*da(t_1)\|^2_{L^2}\le A+B\cdot R.
\end{equation}
It follows from \eqref{eqn_define_bn} that 
$$
e^{\lambda_n t}\,b_n(t) = \frac{d}{dt} (e^{\lambda_n t} a_n(t)),
$$
therefore
$$
\lambda_n a_n(t) = a_n(t_1)\cdot e^{(t_1-t)\lambda_n}\cdot\lambda_n+
                                    \int_{t_1}^t e^{\lambda_n(s-t)}\cdot 
                                    \lambda_n \cdot b_n(s)\,ds.
$$
Recall that $k_0$ is the constant defined by \eqref{eqn_def_k_by_spectrum}. If $t>k_0+t_1$, we have
$$
\|\lambda_n a_n(t)\|_{L^2}^2  
    \le 3\big({X}^2_n(t)+{Y}^2_n(t)+{Z}^2_n(t)\big),
$$
where 
\begin{align*}
{X}_n(t)&=\|a_n(t_1) \, e^{(t_1-t)\lambda_n}\lambda_n\|_{L^2(Y)},\\
{Y}_n(t)&=\|\int_{t_1}^{t-k} b_n(s) \, e^{(s-t)\lambda_n}\lambda_n\,ds\|_{L^2(Y)},\\
{Z}_n(t)&=\|\int_{t-k}^{t} b_n(s) \, e^{(s-t)\lambda_n}\lambda_n\,ds\|_{L^2(Y)}\\
&=\|\int_{0}^{k} b_n(t-s) \, e^{-s\lambda_n}\lambda_n\,ds\|_{L^2(Y)}.
\end{align*}
If $n> 0$, then $\lambda_n\ge {1}/{k_0}$, hence
$$
X_n(t)\le \|\lambda_n a_n(t_1)\|_{L^2}\cdot e^{(t_1-t)/k_0}.
$$
If $R\ge 2k_0+3$, then $-R/2-1/2 > k_0+1$, hence \eqref{eqn: M+CR} and the above inequality imply
$$
\int_{-\frac{R}{2}-\frac{1}{2}}^{-\frac{R}{2}+\frac{1}{2}}\sum_{n\ge 1}X_n(t)^2\,dt
\le (A+B\cdot R)\cdot e^{-(R-3)/k_0}.
$$
By the Minkowski inequality, when $t>k_0+t_1$,
$$
Y_n(t)\le \int_{t_1}^{t-k_0} \|b_n(s)\|_{L^2} \, e^{(s-t)\lambda_n}\lambda_n \,ds.
$$
Notice that if $t-s>k_0$, then $e^{(s-t)\lambda}\cdot \lambda$ is decreasing with respect to $\lambda$ for all $\lambda\ge 1/k_0$,
hence 
\begin{align*}
\Big(\sum_{n\ge 1}Y_n(t)^2\Big)^{1/2} &\le \Big(\sum_{n\ge 1}\big(\int_{t_1}^{t-k_0} \|b_n(s)\|_{L^2} \, e^{(s-t)\lambda_n}\lambda_n \,ds\big)^2\Big)^{1/2} \\
 &\le \Big(\sum_{n\ge 1}\big(\int_{t_1}^{t-k_0} \|b_n(s)\|_{L^2} \, e^{(s-t)\lambda_1}\lambda_1 \,ds\big)^2\Big)^{1/2}.
 \end{align*}
 By the Minkowski inequality again and \eqref{eqn_sum_bn_L2_upper_bound}, we have
 \begin{align*} 
  & \Big(\sum_{n\ge 1}\big(\int_{t_1}^{t-k_0} \|b_n(s)\|_{L^2} \, e^{(s-t)\lambda_1}\lambda_1 \,ds\big)^2\Big)^{1/2}
  \\
 \le & \int_{t_1}^{t-k_0}\Big(\sum_{n\ge 1}\|b_n(s)\|_{L^2}^2\Big)^{1/2}e^{(s-t)\lambda_1}\lambda_1\,ds \\
 \le & \sqrt{2B\cdot\Vol(Y)} \int_{t_1}^{t-k_0} e^{(s-t)\lambda_1}\lambda_1\,ds \\
 \le & \sqrt{2B\cdot\Vol(Y)} \int_{-\infty}^{0} e^{s\lambda_1}\lambda_1\,ds = \sqrt{2B\cdot\Vol(Y)}.
\end{align*}
Therefore when $R>2k_0+3$, we have
$$
\int_{-\frac{R}{2}-\frac{1}{2}}^{-\frac{R}{2}+\frac{1}{2}}\sum_{n\ge 1}Y_n(t)^2\,dt
\le 2B\cdot \Vol(Y).
$$
As for $Z_n$, the Minkowski inequality gives the following estimates when $R>2k_0+3$:
\begin{align*}
\Big(\int_{-\frac{R}{2}-\frac{1}{2}}^{-\frac{R}{2}+\frac{1}{2}} Z_n(t)^2 \,dt \Big)^{1/2} &\le \Big(\int_{-\frac{R}{2}-\frac{1}{2}}^{-\frac{R}{2}+\frac{1}{2}} \big(\int_0^{k_0} \|b_n(t-s)\|_{L^2}\cdot \, e^{-s\lambda_n}\lambda_n\,ds\big)^2 \,dt \Big)^{1/2} \\
&\le\int_0^{k_0} \Big(\int_{-\frac{R}{2}-\frac{1}{2}}^{-\frac{R}{2}+\frac{1}{2}} \|b_n(t-s)\|_{L^2}^2\,dt\Big)^{1/2} e^{-s\lambda_n}\lambda_n\,ds\\
&\le \int_0^{k_0} \Big(\int_{-\frac{R}{2}-\frac{1}{2}-k_0}^{-\frac{R}{2}+\frac{1}{2}}\|b_n(t)\|_{L^2}^2\,dt \Big)^{1/2} e^{-s\lambda_n}\lambda_n\,ds\\
&\le \Big(\int_{-\frac{R}{2}-\frac{1}{2}-k_0}^{-\frac{R}{2}+\frac{1}{2}}\|b_n(t)\|_{L^2}^2\,dt \Big)^{1/2}\int_0^{+\infty} e^{-s\lambda_n}\lambda_n\, ds\\
&=\Big(\int_{-\frac{R}{2}-\frac{1}{2}-k_0}^{-\frac{R}{2}+\frac{1}{2}}\|b_n(t)\|_{L^2}^2\,dt \Big)^{1/2}.
\end{align*}
Therefore, when $R>2k_0+3$,
\begin{align*}
\int_{-\frac{R}{2}-\frac{1}{2}}^{-\frac{R}{2}+\frac{1}{2}}\sum_{n\ge 1}Z_n(t)^2\,dt
&\le \int_{-\frac{R}{2}-\frac{1}{2}-k_0}^{-\frac{R}{2}+\frac{1}{2}}\sum_{n\ge 1}\|b_n(t)\|_{L^2}^2\,dt \\
&\le (k_0+1)2B\cdot\Vol(Y).
\end{align*}
Combining the estimates above, when $R>2k_0+3$, we have
\begin{align*}
\int_{-\frac{R}{2}-\frac{1}{2}}^{-\frac{R}{2}+\frac{1}{2}}\sum_{n\ge 1}\|\lambda_n a_n(t)\|_{L^2}^2  \,dt
    &\le \int_{-\frac{R}{2}-\frac{1}{2}}^{-\frac{R}{2}+\frac{1}{2}}\sum_{n\ge 1}3\big({X}^2_n(t)+{Y}^2_n(t)+{Z}^2_n(t)\big) \\
    &\le 6(k_0+2)B\, \Vol(Y)+ 3(A+B\cdot R)\cdot e^{-(R-3)/k_0}.
\end{align*}
On the other hand, there exists a $t_2\in[-1,0]$ such that
$$
\sum_{n}\lambda_n^2 \|a_n(t_2)\|_{L^2}^2 \le A+B\cdot R.
$$
If $n<0$, we have the identity
$$
\lambda_n a_n(t) = a_n(t_2)\cdot e^{(t-t_2)(-\lambda_n)}\cdot\lambda_n-
                                    \int_{t}^{t_2} e^{(-\lambda_n)(t-s)}\cdot 
                                    \lambda_n \cdot b_n(s)\,ds.
$$
When $R>2k_0+3$, a similar argument gives
$$
\int_{-\frac{R}{2}-\frac{1}{2}}^{-\frac{R}{2}+\frac{1}{2}}\sum_{n\le -1}\|\lambda_n a_n(t)\|_{L^2}^2  \,dt
    \le 6(k_0+2)B\, \Vol(Y) +3(A+B\cdot R)\cdot e^{-(R-3)/k_0}.
$$
Therefore, when $R>2k_0+3$
\begin{align*}
\int_{-\frac{R}{2}-\frac{1}{2}}^{-\frac{R}{2}+\frac{1}{2}} \|*da\|_{L^2}^2\,dt&=
\int_{-\frac{R}{2}-\frac{1}{2}}^{-\frac{R}{2}+\frac{1}{2}}\sum_{n=-\infty}^{+\infty}\|\lambda_n a_n(t)\|_{L^2}^2 \,dt\\
&\le 12(k_0+2)B\, \Vol(Y) +6(A+B\cdot R)\cdot e^{-(R-3)/k_0}.
\end{align*}
Notice that by \eqref{eqn_Fa+_in_*da}, \eqref{eqn_Fa-_in_*da},
\begin{align*}
\int_{[-\frac{R}{2}-\frac{1}{2},-\frac{R}{2}+\frac{1}{2}]\times Y}\big||F_{a}^-|-|F_{a}^+|\big|^2 \, dt\le \int_{[-\frac{R}{2}-\frac{1}{2},-\frac{R}{2}+\frac{1}{2}]\times Y} 2\|*da(t)\|_{L^2}^2\,dt.
\end{align*}
Hence 
\begin{align*}
& \int_{[-\frac{R}{2}-\frac{1}{2},-\frac{R}{2}+\frac{1}{2}]\times Y}|F_a|^2
\\
\le \,&2\int_{[-\frac{R}{2}-\frac{1}{2},-\frac{R}{2}+\frac{1}{2}]\times Y}|2F_a^+|^2+ 2\int_{[-\frac{R}{2}-\frac{1}{2},-\frac{R}{2}+\frac{1}{2}]\times Y}\big||F_a^-|-|F_a^+|\big|^2\\
\le \, & 8B\Vol(Y)+48(k_0+2)B\Vol(Y)+24(A+B\cdot R)\cdot e^{-(R-3)/k_0},
\end{align*}
and the lemma follows from the inequality above by taking $R_0(A,B)$ sufficiently large such that $R_0(A,B)>2k_0+3$, and
$$
24(A+B\cdot R)\cdot e^{-(R_0(A,B)-3)/k_0} \le 1.\phantom\qedhere\makeatletter\displaymath@qed
$$
\end{proof}

\begin{proof}[Proof of Proposition \ref{prop: energy bound}]
Pick $r_0$ sufficiently large such that Lemma \ref{lem: linear energy bound} is valid for all $r>r_0$.
Let the function $T:\mathbb{R}^2\to \mathbb{R}^+$ and the constant $z$ be as in Lemma \ref{lem: linear energy bound}. Let $z'$ be the right-hand side of \eqref{eqn: L2 bound of curvature at R}. Let $z_0$ be the constant $z$ in Proposition \ref{uniform C0 bound}. Let the function $R_0$ and the constants $z_1$, $z_2$ be as in Lemma \ref{lem: midpoint energy bound}. 

Let $C_1=\max\{zr^2 ,z_0^2 \,r^2\}$, let $\kappa=\max\{z', \, z_1\cdot C_1+z_2\}$. 

Let $$R_{\min}=\inf \{R\ge 1|\int_{t\in[-R,-R+1]} |F_a|^2 < \kappa\}.$$
Since $\kappa\ge z'$, it follows from \eqref{eqn: L2 bound of curvature at R} that $R_{\min}$ exists and is finite. By \eqref{eqn: linear energy bound on tube},
\begin{align*}
\int_{[-R,0]\times Y}|F_a|^2 &<T(\kappa,r)+zr^2\cdot R\\
                        &\le T(\kappa,r)+C_1\cdot R.
\end{align*}
By the definitions of $z_0$ and $C_1$, we have the following pointwise estimate
$$
|F_a^+|^2\le z_0^2\, r^2\le C_1.
$$
If $R_{\min}>R_0(T(\kappa,r),C_1)$ and $R_{\min}>1$, take $R'={(R_{\min}+1)}/{2}$, then Lemma \ref{lem: midpoint energy bound} gives
$$
\int_{t\in[-R',-R'+1]}|F_a|^2<z_1\cdot C_1+z_2
\le \kappa.
$$
Since $R'<R$, this contradicts the definition of $R_{\min}$. Therefore,
\begin{equation}\label{eqn_upper_bound_Rmin}
R_{\min}\le \max\{R_0(T(\kappa,r),C_1), 1\},
\end{equation}
hence by \eqref{eqn: linear energy bound on W-},
\begin{align*}
\int_{M_s}E_r(A,\phi)&<T(\kappa,r)+z\,r^2\cdot R_{\min}\\
&\le T(\kappa,r)+z\,r^2\cdot 
\max\{R_0(T(\kappa,r),C_1), 1\},
\end{align*}
and Proposition \ref{prop: energy bound} is proved.
\end{proof}

\subsection{Uniform exponential decay of $E_r(A,\phi)$}

We can now prove the uniform exponential decay of $E_r(A,\phi)$. Recall that the function $d$ is defined on $M_s$ as follows. For each connected component $M_s^{(k)}$ of $M_s$, if $\partial M_s^{(k)}$ is nonempty, then $d$ is the distance function to $\partial M_s^{(k)}$ on $M_s^{(k)}$. Otherwise, fix a point $x^{(k)}\in M_s^{(k)}$, and $d$ is the distance function to $x^{(k)}$ on $M_s^{(k)}$. The following is a re-statement of Theorem \ref{thm_uniform_exponential_decay}.

\begin{thm} \label{thm: uniform decay}
There exist constants $r_0,z>0$ with the following significance. For every $r>r_0$, there is a constant $C$ depending on $r$, such that if $r>r_0$ and $(A,\phi)\in\moduli{k}$ solves \eqref{SW}, then the inequality
\begin{equation}\label{eqn: uniform energy bound}
E_r(A,\phi)< C\,e^{-\sqrt{r}\cdot d/z}
\end{equation}
holds on $M_s$. \qed
\end{thm}

\begin{proof}
First, we prove that the constant $d(\delta)$ given by Lemma \ref{lem: pointwise convergence} is uniform for all $(A,\phi)$. More precisely, for all $r>r_0$ and $\delta>0$, there exists a constant $d(\delta)$ depending on $r$ and $\delta$, such that for all $x\in M_s$ with $d(x)>d(\delta)$ and $(A,\phi)\in\moduli{k}$ solving \eqref{SW}, we have 
$$E_r(A,\phi)(x)\le \delta.$$ 

Assume the contrary, then there exist $\delta_0>0$, a sequence of solutions $(A_n,\phi_n)$, and a sequence of points
$x_n\in M_s$ with 
$\lim_{n\to\infty}d(x_n)=+\infty$, such that 
$$E_r(A_n,\phi_n)(x_n)\ge \delta_0$$ for all $n$.

Let $g$ be the metric on $M_s$. By Proposition \ref{prop: properness} and the $C^0$ bound of $|\phi_n|$ given by Lemma \ref{C0 bound on the negative end}, a subsequence of $\{(M_s,g,\frs,x_n,A_n,\phi_n)\}_{n\ge 1}$ converges to a limit $(\widetilde{M},\tilde{g},\tilde{\frs},\tilde{x},\tilde{A},\tilde{\phi})$. The limit manifold $(\widetilde{M},\tilde{g})$ is complete and has bounded geometry. 
Since $\nabla^k \theta$ is bounded for all $k$, by a diagonal argument and the Arzel\`a–Ascoli theorem, after taking a further subsequence, we may assume that $\theta$ converges to a 1-form $\widetilde{\theta}$ on $\widetilde{M}$, in the $C^\infty$ topology on compact subsets.
 Therefore, $\widetilde{\omega} = d\widetilde{\theta}$ is a compatible symplectic structure on $\widetilde{M}$, and $\tilde\frs$ is the canonical $\spinc$ structure on $\widetilde{M}$ induced by $\widetilde{\omega}$. In conclusion, $\widetilde{M}$ is an ESBG end with empty boundary, and $(\tilde{A},\tilde{\phi})$ is a solution to \eqref{SW} on $\widetilde{M}$.

Let $\widetilde{E}_r(\tilde{A},\tilde{\phi})$ be the energy density function of $(\tilde{A},\tilde{\phi})$ on $\widetilde{M}$. By the previous assumptions, we have
$\widetilde{E}_r(\tilde{A},\tilde{\phi})(\tilde{x})\ge \delta_0.$
The uniform energy bound given by Proposition \ref{prop: energy bound} implies
\begin{equation}\label{eqn_uniform_energy_bound_tilde_E}
\int_{\widetilde{M}} \widetilde{E}_r(\tilde{A},\tilde{\phi}) < +\infty.
\end{equation}

Apply Lemma \ref{lem: integration by parts} to $\widetilde{M},\tilde{A},\tilde{\phi}$. Since $\widetilde{M}$ is a symplectic end itself with empty boundary, we can take $\chi = 1$ on $\widetilde{M}$. Therefore when $r_0$ is sufficiently large and $r>r_0$, we have
\begin{multline}\label{eqn: integration by parts on tilde W}
\int_{\widetilde M}\Big( \frac{r}{2}|\bar{\partial}_{\tilde{a}}{\tilde{\alpha}} + \bar{\partial}_{\tilde{a}}^*{\tilde{\beta}}|^2+
2|iF_{\tilde{a}}^{\tilde{\omega}}-\frac{r}{8}(1-|{\tilde{\alpha}}|^2+|{\tilde{\beta}}|^2)|^2+2|F^{0,2}_{\tilde{a}}-\frac{r}{4}{\tilde{\alpha}}^*{\tilde{\beta}}|^2\\
+\frac{r}{2}iF_{\tilde{a}}^{\tilde{\omega}}-2|iF_{\tilde{a}}^{\tilde{\omega}}|^2-2|F_{\tilde{a}}^{0,2}|^2 \Big)\\
= \int_{\widetilde M}\Big( \frac{r}{4}|\nabla_{\tilde{a}}{\tilde{\alpha}}|^2+\frac{r}{4}|\nabla_{\tilde A_1+\tilde a}{\tilde{\beta}}|^2+\frac{r}{2} (iF_{\tilde A_1}^{\tilde{\omega}}) |\tilde{\beta}|^2  \\
+\frac{r^2}{32}(1-|{\tilde{\alpha}}|^2-|{\tilde{\beta}}|^2)^2+\frac{r^2}{8}|{\tilde{\beta}}|^2-rRe\langle \widetilde N \circ \partial_{\tilde{a}}{\tilde{\alpha}},{\tilde{\beta}}\rangle\Big),
\end{multline}
where $\tilde{A}_1$ is the unique unitary connection on $T^{0,2}\widetilde{M}$ such  that $\nabla_{\tilde{A}_1}^{1,0}=\partial$, and $\widetilde{N}:T^{1,0}\widetilde{M}\to T^{0,2} \widetilde{M}$ is the Nijenhuis tensor.

Since $(\tilde{A},\tilde{\phi})$ solves \eqref{SW} on $\widetilde{M}$, we have
\begin{align*}
& \bar{\partial}_{\tilde{a}}{\tilde{\alpha}}
+\bar{\partial}_{\tilde{a}}^*{\tilde{\beta}} =0, 
\\
& F_{\tilde{a}}^{\tilde{\omega}} = -\frac{ir}{8}(1-|\tilde{\alpha}|^2+|\tilde{\beta}|^2), 
\\
& F_{\tilde{a}}^{0,2}=\frac{r}{4}{\tilde{\alpha}}^*{\tilde{\beta}}.
\end{align*}
Therefore \eqref{eqn: integration by parts on tilde W} gives
\begin{multline} \label{eqn: energy identity on tilde W}
\int_{\widetilde M}\big(\frac{r}{2}iF_{\tilde{a}}^{\tilde{\omega}}-2|iF_{\tilde{a}}^{\tilde{\omega}}|^2-2|F_{\tilde{a}}^{0,2}|^2 \big)
= \int_{\widetilde M}\Big( \frac{r}{4}|\nabla_{\tilde{a}}{\tilde{\alpha}}|^2+\frac{r}{4}|\nabla_{\tilde A_1+\tilde a}{\tilde{\beta}}|^2
\\
+\frac{r}{2}(iF_{\tilde A_1}^{\tilde{\omega}}) |\tilde{\beta}|^2 
+\frac{r^2}{32}(1-|{\tilde{\alpha}}|^2-|{\tilde{\beta}}|^2)^2+\frac{r^2}{8}|{\tilde{\beta}}|^2-r Re \langle\tilde N  \circ \partial_{\tilde{a}}{\tilde{\alpha}},{\tilde{\beta}}\rangle\Big).
\end{multline}
When $r_0$ is sufficiently large, by \eqref{eqn_uniform_energy_bound_tilde_E} and the same proof of \eqref{eqn_integration_by_parts_no_boundary_term}, we have,
\begin{align}
\int_{\widetilde M}F_{\tilde a}^{\tilde{\omega}}=\frac{1}{2}\int_{\widetilde M}F_{\tilde a}\wedge \widetilde\omega =\frac{1}{2}\int_{\widetilde M}F_{\tilde a}\wedge d\widetilde\theta =-\frac{1}{2}\int_{\widetilde M}d(F_{\tilde a})\wedge \widetilde\theta=0,&
\\
\int_{\widetilde M}|F_{\tilde a}^+|^2
-
\int_{\widetilde M}|F_{\tilde a}^-|^2
= \int_{\widetilde M}F_{\tilde a}\wedge F_{\tilde a} = 0.&
\label{eqn_integrate_|F^+|^2=|F^-|^2_on_limit}
\end{align}
Therefore equation \eqref{eqn: energy identity on tilde W} gives
\begin{multline}\label{eqn: no nontrivial solution on symplectic manifold}
0=\int_{\widetilde M}\Big( 2|iF_{\tilde{a}}^{\tilde{\omega}}|^2+2|F_{\tilde{a}}^{0,2}|^2+\frac{r}{4}|\nabla_{\tilde{a}}{\tilde{\alpha}}|^2+\frac{r}{4}|\nabla_{\tilde A_1+\tilde a}{\tilde{\beta}}|^2+\frac{r}{2}(iF_{\tilde A_1}^{\tilde{\omega}}) |\tilde{\beta}|^2 \\
+\frac{r^2}{32}(1-|{\tilde{\alpha}}|^2-|{\tilde{\beta}}|^2)^2+\frac{r^2}{8}|{\tilde{\beta}}|^2-rRe\langle \widetilde N \circ \partial_{\tilde{a}}{\tilde{\alpha}},{\tilde{\beta}}\rangle\Big).
\end{multline}
When $r_0$ is sufficiently large, by  \eqref{eqn_energy_inequality_before_rearrange} and \eqref{eqn_control_beta_terms_in_energy_identity_for_rearrangement}, and \eqref{eqn: no nontrivial solution on symplectic manifold}, we have
\begin{equation}\label{eq: energy at infinity is zero}
0\ge \int_{\widetilde M}|\nabla_{\tilde{a}}{\tilde{\alpha}}|^2+|\nabla_{\tilde A}'{\tilde{\beta}}|^2
+(1-|{\tilde{\alpha}}|^2-|{\tilde{\beta}}|^2)^2+|{\tilde{\beta}}|^2+|F_{\tilde{a}}^+|^2.
\end{equation}
As a consequence, the integrand of the right-hand side of $\eqref{eq: energy at infinity is zero}$ is identically $0$ on $\widetilde{M}$. By \eqref{eqn_integrate_|F^+|^2=|F^-|^2_on_limit}, $|F_a^-|$ is also identically zero on $\widetilde{M}$, hence $\widetilde E_r(\tilde{A},\tilde{\phi})$ is identically zero. This contradicts the assumption that $\widetilde E_r(\tilde{A},\tilde{\phi})(\tilde x)\ge \delta_0$. In conclusion, the constant $d(\delta)$ given by Lemma \ref{lem: pointwise convergence} is uniform for all $(A,\phi)$.

Now we finish the proof of Theorem \ref{thm: uniform decay}.
Since the constant $d(\delta)$ given by Lemma \ref{lem: pointwise convergence} is uniform for all $(A,\phi)$, the constants $d_1$, $d_2$, $d_3$ in the proof of Proposition \ref{prop: exponential decay on symplectic ends} depend only on $r$ but not on $(A,\phi)$. 
Therefore the constant $d_0$ in the statement of Proposition \ref{prop: exponential decay on symplectic ends} depends only on $r$ but not on $(A,\phi)$. This proves \eqref{eqn: uniform energy bound} when $d>d_0$. By standard elliptic bootstrapping, there exists a constant $C_0$ depending on $r$ such that $E_r(A,\phi)<C_0$ pointwise. This proves the estimate for \eqref{eqn: uniform energy bound} when $d\le d_0$. Hence the theorem is proved.
\end{proof}

\subsection{Uniform decay with neck stretching}
\label{subsec_Uniform decay with neck stretching}
For $i=1,2$, suppose $X^{(i)}$ is a manifold with  cylindrical and ESBG ends, where the cylindrical end is $M_c^{(i)}$, and the ESBG end is given by $(M_s^{(i)},\omega^{(i)}=d\theta^{(i)})$. Let $Z^{(i)} = X^{(i)}-M_c^{(i)}-M_s^{(i)}$. 
Moreover, suppose there exists a non-empty oriented closed Riemannian 3-manifold $Y$, such that $M_c^{(1)}$ is given by $(-\infty,0]\times Y$, and $M_c^{(2)}$ is given by $(-\infty,0]\times (-Y)$. 

For each constant $R>0$, we can define a Riemannian manifold $X_R$ as follows. Let $X_R^{(1)}$ be the subset of $X^{(1)}$ given by $Z^{(1)}\cup M_s^{(1)}\cup  [-R,0]\times Y$, let $X_R^{(2)}$ be the subset of $X^{(2)}$ given by $Z^{(2)}\cup M_s^{(2)}\cup [-R,0]\times (-Y)$. Let $X_R$ be the manifold obtained from $X_R^{(1)}\sqcup X_R^{(2)}$ by gluing $[-R,0]\times Y$ with $[-R,0]\times (-Y)$ via $(t,x)\sim (-R-t,x)$. Then $X_R$ is a manifold with ESBG end $M_s^{(1)}\cup M_s^{(2)}$. This subsection proves that the exponential decay estimate on $X_R$ given by Theorem \ref{thm: uniform decay} is uniform for all $R$.

Let $M_s=M_s^{(1)}\cup M_s^{(2)}$, let $\theta$ be the union of $\theta^{(1)}$ and $\theta^{(2)}$ on $M_s$, and extend $\theta$ to a smooth 1-form on $X_R$ such that the support of $\theta$ is contained in $M_s\cup Z^{(1)}\cup Z^{(2)}$. Let $\omega= d\theta$ be a 2-form on $X_R$. 

Let $\frs$ be an admissible $\spinc$ structure on $X_R$, and let $(A,\phi)\in \mathcal{C}_k(X_R,\frs)$ be a solution to \eqref{SW}. Let $A_0$ be the canonical connection of $\frs|_{M_s}$, and extend $A_0$ to a smooth connection of $\frs$ that is translation-invariant on the glued image of $[-R,0]\times Y\subset M_c^{(1)}$ and $[-R,0]\times (-Y)\subset M_c^{(2)}$. 

Let $a=A-A_0$. Decompose $\phi=\sqrt{r}(\alpha+\beta)$ on $M_s$ such that $\alpha\in\Gamma(M_s,T^{0,0}M_s)$, $\beta\in\Gamma(M_s,T^{0,2}M_s)$ as before, and let $E_r(A,\phi)$ be defined by \eqref{eqn_def_Er(A,phi)}.

In this subsection, we will use $r_0$ $z$, and $z_i$ to denote the constants that depend on $X^{(1)},M_c^{(1)},M_s^{(1)},\theta^{(1)},X^{(2)},M_c^{(2)},M_s^{(2)},\theta^{(2)}$,$\frs$, and the perturbation terms of the Seiberg-Witten equations. We will use $C_i$ to denote the constants that depend on the same set of data above and also $r$, but are independent of $R$.

\begin{lem}\label{C0 estimate with neck}
There is a constant $z$, such that $|\phi|<z\cdot\sqrt{r}$, $|F_a^+|<z\cdot r$. 
\end{lem}
\begin{proof}
The proof is the same as Proposition \ref{uniform C0 bound}.
\end{proof}

\begin{lem}\label{lem: linear energy bound with neck}
There exist constants $r_0$ and $z$ with the following property. Suppose $r>r_0$, $R\ge 1$,
then there is a constant $C_0$ depending on $r$, such that
\begin{align}
&\int_{M_s} E_r(A,\phi) < C_0+z\,r^2\cdot R,
            \label{eqn: linear energy bound on symplectic end with neck}\\
&\int_{X_R-M_s} |F_a|^2 < C_0+z\,r^2\cdot R.
            \label{eqn: linear energy bound on cylinder with neck}
\end{align} 
\end{lem}

\begin{proof}
By the same argument as in \eqref{eqn: bound energy by integration by parts}, there is a constant $C_1$ depending on $r$, such that
\begin{equation} \label{eqn: bound energy by integration by parts with neck}
\int_{M_s}(E_r(A,\phi)-|F_a^-|^2)\le C_1+ \frac{r}{2}\int_{M_s}iF_a^\omega.
\end{equation}
Similar to \eqref{eqn: integral of F wedge Omega is zero} and \eqref{eqn: integral of F wedge F is zero}, we have
$$c_0:=-\int_{X_R}F_a\wedge F_a = -\int_{X_R}|F_a^+|^2+\int_{X_R}|F_a^-|^2$$
 is a topological invariant that only depends on $\frs$, and 
 $$\int_{X_R}F_a\wedge \omega=0.$$
  Therefore, there exists a constant $z_1$, and $C_2, C_3$ depending on $r$, such that 
\begin{align*}
\int_{M_s}(E_r(A,\phi)-|F_a^-|^2) & \le C_1+ \frac{r}{2}\Big|\int_{M_s}iF_a^\omega\Big|
\\
& =   C_1+ \frac{r}{4}\Big|\int_{M_s}F_a\wedge \omega\Big|
\\
                  &= C_1+\frac{r}{4}\Big| \int _{X_R-M_s}
                      F_a\wedge \omega \Big| \\
                  &\le C_1 +
                      \frac{r^2}{16} \int_{X_R-M_s}|\omega|^2
                      + \frac{1}{4} \int_{X_R- M_s} |F_a|^2 \\
                  &\le C_2 + \frac{1}{4} \int_{X_R-M_s} |F_a|^2\\
                  &\le C_2 + \frac{1}{4} \int_{X_R} |F_a|^2\\
                  &=   C_2 + \frac{c_0}{4}+
                      \frac{1}{2} \int_{X_R} |F_a^+|^2\\
                  &=   C_2 + \frac{c_0}{4}+\frac12 \int_{X_R-M_s}|F_a^+|^2 +
                      \frac{1}{2} \int_{M_s} |F_a^+|^2\\
                  &\le C_3 + z_1\,r^2\cdot R +
                      \frac{1}{2}\int_{M_s}(E_r(A,\phi)-|F_a^-|^2),
\end{align*}
where the last inequality follows from Lemma \ref{C0 estimate with neck}.
Therefore
\begin{equation}
\label{eqn_linear_energy_bound_without_F-_with_neck}
\int_{M_s}(E_r(A,\phi)-|F_a^-|^2) \le 2C_3+2z_1\,r^2\cdot R.
\end{equation}
On the other hand, by Lemma \ref{C0 estimate with neck} again and \eqref{eqn_linear_energy_bound_without_F-_with_neck}, there exist a constant $z_2$, and a constant $C_4$ depending on $r$, such that
\begin{align*}
\int_{X_R}|F_a|^2 &\le |c_0|+2\int_{X_R}|F_a^+|^2\\
                  &\le |c_0|+2\int_{X_R-M_s} |F_a^+|^2 + 
                      2\int_{M_s}(E-|F_a^-|^2)\\
                  &\le C_4+z_2\,r^2\cdot R.
\end{align*}
The lemma is then proved by combining the two inequalities above.
\end{proof}

\begin{lem}\label{thm: uniform energy bound for stretching neck}
There exists a constant $r_0>0$, such that for every $r>r_0$ there is a constant $C>0$ depending on $r$ but independent of $R$, such that if $(A,\phi)\in\mathcal{C}_k(X_R,\frs)$ solves \eqref{SW}, then we have
\begin{equation}\label{eqn_uniform_energy_bound_with_neck}
\int_{M_s}E_r(A,\phi)< C.
\end{equation}

\end{lem}

\begin{proof}
The proof follows from Lemma \ref{lem: linear energy bound with neck} and an argument similar to the proof of Proposition \ref{prop: energy bound}. 

Let the constants $z_1,\,z_2$ and the function $R_0$ be as in Lemma \ref{lem: midpoint energy bound}. Let $C_0$ and $z$ be the constants in Lemma \ref{lem: linear energy bound with neck}. Let $z'$ be the constant given by Lemma \ref{C0 estimate with neck}. Let $C_1=\max\{zr^2,z'r^2\}$, let $\kappa= z_1C_1+z_2+1$. If $R\le R_0(C_0,C_1)$ then by Lemma \ref{lem: linear energy bound with neck},
$$
\int_{M_s}E_r(A,\phi)<C_0+zr^2\,R_0(C_0,C_1),
$$ 
hence \eqref{eqn_uniform_energy_bound_with_neck} holds when $C>C_0+zr^2\,R_0(C_0,C_1)$.

If $R>R_0(C_0,C_1)$, recall that $[-R,0]\times Y\subset M_c^{(1)}$ is a subset of $X_R$. By Lemma \ref{lem: linear energy bound with neck},
$$
\int_{[-R,0]\times Y} |F_a|^2 \le C_0+zr^2\,R\le C_0+C_1\cdot R,
$$
by Lemma \ref{C0 estimate with neck},
$$
|F_a^+|\le z'r^2\le C_1,
$$
 hence Lemma \ref{lem: midpoint energy bound} gives
\begin{equation}\label{eqn_Fa_midpoint_L2_bound_on_neck}
\int_{[-R/2-1/2,-R/2+1/2]\times Y}|F_a|^2\le z_1 C_1 + z_2 < \kappa.
\end{equation}
Take 
$$
R_{\min}=\inf\{\hat R|\hat R\ge 0, \int_{[-\hat R-1,-\hat R]\times Y}|F_a|^2< \kappa\},
$$
then \eqref{eqn_Fa_midpoint_L2_bound_on_neck} implies that $R_{\min}$ exists.
Let $z^{(1)}>0$ and $T^{(1)}:(\mathbb{R}^+)^2\to\mathbb{R}^+$ be the constant and the function given by Lemma \ref{lem: linear energy bound} when applied to $X^{(1)}$, let  $C_1^{(1)} = \max\{z^{(1)}r^2,z'r^2\}$.
Then by the same proof of \eqref{eqn_upper_bound_Rmin}, we have 
$$R_{\min}\le \max \big\{ R_0(T^{(1)}(\kappa,r),C_1^{(1)}), 1\big\}.$$
Now apply Lemma \ref{lem: linear energy bound} to  $X^{(1)}$, we obtain 
\begin{align*}
\int_{M_s^{(1)}}E_r(A,\phi) &< T^{(1)}(\kappa,r)+z^{(1)}r^2\cdot R_{\min}
\\
&\le T^{(1)}(\kappa,r)+z^{(1)}r^2\cdot \max\Big\{R_0\big(T^{(1)}(\kappa,r),C_1^{(1)}\big),1\Big\},
\end{align*}
which yields a uniform upper bound for the integration of $E_r(A,\phi)$ on $M_s^{(1)}$. The upper bound for  $M_s^{(2)}$
follows from a similar argument.
\end{proof}

\begin{thm}\label{thm: uniform exponential decay for stretching neck}
There exist constants $r_0,z>0$ with the following significance. For every $r>r_0$, there is a constant $C>0$ depending on $r$ but independent of $R$, such that if $(A,\phi)\in\mathcal{C}_k(X_R,\frs)$ solves \eqref{SW} on $X_R$, then the followsing inequality holds on $M_s$:
$$
E_r(A,\phi)< C e^{-\sqrt{r}\cdot d/z}.
$$ 
\end{thm}

\begin{proof}
This follows from Lemma \ref{thm: uniform energy bound for stretching neck} and the same argument as the proof of Theorem \ref{thm: uniform decay}. 
\end{proof}

\section{Floer chains from ESBG structures}
\label{sec_Floer_chains}
Let $X$ be a Riemannian 4-manifold with cylindrical and ESBG ends, such that $M_c\cong (-\infty, 0]\times Y$ is the cylindrical end, and $(M_s,\omega=d\theta)$ is the ESBG end.
If $M_s$ is compact and $b_2^+(X)\ge 2$, after removing a small ball from $X$, we can view $X-B^4$ as a cobordism from $S^3$ to $-Y$. By the construction of \cite[Section 25]{kronheimer2007monopoles}, the cobordism $X-B^4$ induces a map
$$
\Harrow(X-B^4):\Hfrom(S^3) \to \Hto(-Y).
$$
The map $\Harrow(X-B^4)$ is only well-defined up to a sign, which can be fixed by a choice of the homology orientation of $X$. The construction above defines an element 
$$\Harrow(X-B^4)(\hat 1)\in \Hto(-Y)/\{\pm 1\},$$ 
where $\hat 1\in \Hfrom(S^3)$ is the generator of $\Hfrom(S^3)$ as a $\mathbb{Z}[U_\dagger]$--module. 
The condition $b_2^+(X)\ge 2$ is necessary to gurantee that under a generic perturbation, the solutions of the Seiberg-Witten equations on $X$ are all irreducible, namely the spinor part is not identically zero.

In this section, we will show that when $M_s$ is not compact, it is still possible to define an element in $\Hto(-Y)/\{\pm 1\}$ by counting solutions of \eqref{SW} on $X$. This is a straightforward generalization of \cite[Section 6.3]{kronheimer2007monopolesAndLens}. Theorem \ref{thm: uniform decay} implies that the moduli space analogous to the space $M(Z^+,\mathfrak{a})$ in \cite[Section 6.3]{kronheimer2007monopolesAndLens} compact. Since $M_s$ is non-compact, Condition (2) of Definition \eqref{def: configuration space} implies that every element in $\moduli{k}$ is irreducible. The sign ambiguity of $\Hto(-Y)/\{\pm 1\}$ is essential and cannot be resolved by a choice of homology orientation. This will be explained in Remark \ref{rmk_contact_case}. 

For each $\spinc$ structure $\frt$ on $Y$, fix a stongly tame perturbation $\mathfrak{q}_\frt$ that is admissible in the sense of \cite[Definition 22.1.1]{kronheimer2007monopoles} and satisfies $\|\mathfrak{q}_{\mathfrak{t}}\|_{\hat{\mathcal{P}}}\le 1$, where $\|\cdot\|_{\hat{\mathcal{P}}}$ is the norm defined by \eqref{eqn_def_norm_strongly_tame_perturbations}. Let $\mathfrak{q}$ be the set of all $\mathfrak{q}_{\mathfrak{t}}$. Let $\mathfrak{C}$ be set of isomorphism classes of critical points of the Chern-Simons-Dirac functional perturbed by $\mathfrak{q}_\mathfrak{t}$ in the blown-up configuration space $\mathcal{B}^\sigma(-Y,\frt)$ as defined in \cite[Section 4.1]{kronheimer2007monopolesAndLens} for all $\frt$ (see also \cite[Section 6]{kronheimer2007monopoles}).
 
\begin{defn}
\label{def_moduli_space_for_c(X)}
Suppose $[\mathfrak{a}]\in\mathfrak{C}$, and $\frs$ is an admissible $\spinc$ structure on $X$. Let $\mathcal{M}(X,[\mathfrak{a}],\frs)$ be the moduli space of $(A,\phi)\in\mathcal{C}(X,\frs) $ that solves \eqref{SW} and is asymptotic to $[\mathfrak{a}]$ on the cylindrical end $M_c$ after lifting to the blown-up configuration space.
\end{defn}

\begin{rmk}
For the definition of asymptoticity in the blown-up configuration space, see \cite[Definition 13.1.1]{kronheimer2007monopoles} and the paragraph above it.
\end{rmk}

For a generic choice of $\hat \tau$ in \eqref{eqn: definition of perturbation}, the moduli space $\mathcal{M}(X,[\mathfrak{a}],\frs)$ is regular for all $\mathfrak{a}$ and $\mathfrak{s}$. By the construction of \cite[Sections 20, 25.2]{kronheimer2007monopoles}, there are two systems of compatible orientations for the moduli spaces $\mathcal{M}(X,[\mathfrak{a}],\frs)$ that are different by an overall sign. Let $\mathfrak{o}$ and $-\mathfrak{o}$ be the two systems of compatible  orientations. Let $\mathcal{M}_0(X,[\mathfrak{a}],\frs)$ be the zero-dimensional components of $\mathcal{M}(X,[\mathfrak{a}],\frs)$. By Theorem \ref{thm: uniform decay} and the compactness results of Seiberg-Witten equations \cite[Lemma 25.3.1]{kronheimer2007monopoles}, there are only finitely many $[\mathfrak{a}]$ and $\frs$ such that $\mathcal{M}_0(X,[\mathfrak{a}],\frs)$ is nonempty.

\begin{defn}\label{def_Floer_chain}
Let 
$$\check{\psi}_\mathfrak{o}(X) = \sum_{\frs} \sum_{[\mathfrak{a}]\in \mathfrak{C}}\# \mathcal{M}_0(X,[\mathfrak{a}],\frs)\cdot [\mathfrak{a}]\in \mathbb{Z}[\mathfrak{C}],$$
where the elements of $\mathcal{M}_0(X,[\mathfrak{a}],\frs)$ are counted with signs using the orientation $\mathfrak{o}$, and the summation of $\frs$ goes over the isomorphism classes of admissible $\spinc$ structures over $X$ relative to $M_s$.
\end{defn}
If $M_s$ is non-compact, by Condition (2) of Definition \eqref{def: configuration space}, all the elements of $\mathcal{M}_0(X,[\mathfrak{a}],\frs)$ are irreducible. Therefore, in order for $\mathcal{M}_0(X,[\mathfrak{a}],\frs)$  to be non-empty, the critical point $\mathfrak{a}$ has to be either irreducible or boundary-stable. 
As a consequence, $\check{\psi}_\mathfrak{o}(X)$ is an element of the boundary-stable monopole Floer chain group $\check{C}_*(-Y)$ defined by \cite[(22.3)]{kronheimer2007monopoles}, and we have the following lemma.
\begin{lem}
The element $\check{\psi}_\mathfrak{o}(X)\in \check{C}_*(-Y)$ satisfies $\check{\partial}(\check{\psi}_\mathfrak{o}(X))=0$.
\end{lem}
\begin{proof}
This follows from the same proof as \cite[Lemma 6.6]{kronheimer2007monopolesAndLens}.
\end{proof}

Therefore, the homology class of $\check{\psi}_\mathfrak{o}(X)$ defines an element in the monopole Floer homology group $\Hto_\bullet(-Y)$. 
Define
\begin{equation}\label{eqn_def_c(X)}
c(X) = [\check{\psi}_\mathfrak{o}(X)] \in \Hto_\bullet(-Y)/\{\pm 1\},
\end{equation}
then $c(X)$ does not depend on the choice of the orientation $\mathfrak{o}$.
\begin{prop}
\label{prop_invariant_under_deformation}
$c(X)\in \Hto_\bullet(-Y)/\{\pm 1\}$ does not depend on the choices of $r>r_0$, the perturbations $\mathfrak{q}$, $\hat \tau$, $\eta$ in \eqref{eqn: definition of perturbation}, or the metric on $X-M_s$. Moreover, $c(X)$ is invariant under smooth deformations of the ESBG structures on $M_s$ with uniformly bounded geometry.
\end{prop}
\begin{proof}
For $i=1,2$, let $g_i$ be a metric on $X$ that is cylindrical on $M_c$, let $\omega_i=d\theta_i$ be an exact symplectic form on $M_s$, suppose $(X,g_i)$ is a manifold with cylindrical end $M_c$ and ESBG end $(M_s,\omega_i=d\theta_i)$. Assume $(X, g_0,\omega_0=d\theta_0)$ can be smoothly deformed to $(X, g_1,\omega_1=d\theta_1)$ via manifolds with cylindrical end $M_c$ and ESBG end $M_s$, such that the deformation has uniformly bounded geometry.

For $i=0,1$,
let $\hat\tau_i$, $\eta_i$ be a choice of perturbation terms in \eqref{eqn: definition of perturbation}. Let $r_i$ be a sufficiently large constact such that Theorem \ref{thm: uniform decay} holds for $(g_1,\omega_i=d\theta_i,\hat\tau,\hat\eta)$. Let $g_i^Y$ be the metric on $Y$ induced by the restriction of $g_i$ to $M_c$. Let $\mathfrak{q}_i$ be a collection of strongly tame, admissible perturbations on $(Y,g_i^Y)$ for all isomorphism classes of $\spinc$ structures, and let $\check C_*(-Y,g_i^Y,\mathfrak{q}_i)$ be the corresponding boundary-stable Floer chain. 

Let $\mathfrak{o}$ be a choice of the orientation,
let $\check{\psi}_\mathfrak{o}(X)(i)\in \check C_*(-Y,g_i^Y,\mathfrak{q}_i)$ be the element defined by Definition \ref{def_Floer_chain} with respect to $g_i, \omega_i=d\theta_i, \mathfrak{q}_i, \hat\tau_i, \eta_i,r_i$. Let $\mathcal{M}(X,[\mathfrak{a}],\frs)(i)$ be the moduli space given by Definition \ref{def_moduli_space_for_c(X)} with respect to the same choice of geometric data, and let $\mathcal{M}_0(X,[\mathfrak{a}],\frs)(i)$ be the zero-dimensional components of $\mathcal{M}(X,[\mathfrak{a}],\frs)(i)$.

Let $\hat{g}$ be a metric on $\mathbb{R}\times (-Y)$, and $\hat{\mathfrak{q}}$ be a collection of perturbations of the Seiberg-Witten equations on $(\mathbb{R}\times (-Y),\hat{g})$ for all isomorphism classes of $\spinc$ structures, such that 
\begin{enumerate}
\item $\hat{g}$ is the cylindrical metric given by $g_1^Y$ on $(-\infty , 1]\times (-Y)$, and is the cylindrical metric given by $g_2^Y$ on $[2 , +\infty)\times (-Y)$;
\item $\hat{\mathfrak{q}}$ is given by the formal gradient of $\mathfrak{q}_1$ with respect to $g_1^Y$ on $(-\infty , 1]\times (-Y)$, and is given by the formal gradient of $\mathfrak{q}_2$ with respect to $g_2^Y$ on $[2,+\infty)\times (-Y)$.
\end{enumerate}
For a generic choice of $\hat{\mathfrak{q}}|_{[1,2]\times (-Y)}$, the Seiberg-Witten equations on $(\mathbb{R}\times (-Y),\hat{g})$ with perturbation $\hat{\mathfrak{q}}$ defines a chain map 
$$\check C_*\big(\mathbb{R}\times (-Y), \hat{\mathfrak{q}}\big):\check C_*(-Y, g_1^Y,\mathfrak{q}_1)\to \check C_*(-Y,g_2^Y,\mathfrak{q}_2).$$ 
We only need to prove that 
$$
[\check C_*\big(\mathbb{R}\times (-Y), \mathfrak{q}\big)\big(\check{\psi}_\mathfrak{o}(X)(1)\big)] = \pm [\check{\psi}_\mathfrak{o}(X)(2)]
$$
in the homology of $(\check C_*(-Y,g_2^Y,\mathfrak{q}_2), \check\partial)$.

For $t\ge 1$, consider the Seiberg-Witten equations on $X$ where the metric and the perturbation are given by $g_1,\omega_1,\hat\tau_1,\eta_1,\mathfrak{q}_1$ on $X-M_c$, are given by $g_1^Y$ and $\mathfrak{q}_1$ on $[0,t]\times (-Y)$, are given by $g_2^Y$ and $\mathfrak{q}_2$ on $[t+1,+\infty)\times (-Y)$, and are given by $(\hat{g},\hat{\mathfrak{q}})|_{[1,2]\times (-Y)}$ on $[t,t+1]\times (-Y)$. Concatenate this family of equations with a smooth family of equations parametrized by $t\in[0,1]$, such that at $t=1$ the two equations coincide, at $t=0$ the equation coincides with the equation defined by $g_2,\omega_2,\mathfrak{q}_2,\hat\tau_2,\eta_2,r_2$ on $X$. The family of equations can be chosen to be independent of $t$ on $[2,+\infty)\times(-Y)$ for $t\in[0,1]$. Moreover, by the assumptions on the ESBG structures, we may choose the family such that for each $t\in [0,1]$, the ESBG end of $X$ is $M_s$, and the family of metrics on $X$ for $t\ge 0$ has uniformly bounded geometry.

For $t\ge 0$, let $g(t)$, $\omega(t)$,  $\mathfrak{q}(t)$, $\hat \tau(t)$, $\eta(t)$, $r(t)$ be the corresponding geometric data as given above. 
Let $\mathfrak{C}(\mathfrak{q}_1, g_1^Y)$ be the set of critical points in the blown-up configuration space give by $(\mathfrak{q}_1, g_1^Y)$.
For $\mathfrak{a}\in \mathfrak{C}$ and $\frs$ an admissible $\spinc$ structure on $X$, define $\mathcal{M}(X,[\mathfrak{a}],\frs)(t)$ to be the moduli space given by Definition \ref{def_moduli_space_for_c(X)} with respect to $g(t)$, $\omega(t)$, $\mathfrak{q}(t)$, $\hat \tau(t)$, $\eta(t)$, $r(t)$. 
Define $$\widetilde{\mathcal{M}}(X,[\mathfrak{a}],\frs) = \bigcup_{t\ge -1}\mathcal{M}(X,[\mathfrak{a}],\frs)(t).$$ 
For a generic choice of $\hat\tau(t)$, the moduli space $\widetilde{\mathcal{M}}(X,[\mathfrak{a}],\frs)$ is regular.  Let $\widetilde{\mathcal{M}}_0(X,[\mathfrak{a}],\frs)$ be the zero-dimensional components of $\widetilde{\mathcal{M}}(X,[\mathfrak{a}],\frs)$, and let $\widetilde{\mathcal{M}}_1(X,[\mathfrak{a}],\frs)$ be the one-dimensional components of $\widetilde{\mathcal{M}}(X,[\mathfrak{a}],\frs)$. Then by increasing $r(t)$ if necessary, we have that $\widetilde{\mathcal{M}}_0(X,[\mathfrak{a}],\frs)$ is compact, and so we can define an element 
$$
h_\mathfrak{o} = \sum_{\frs} \sum_{[\mathfrak{a}]\in \mathfrak{C}}\# \widetilde{\mathcal{M}}_0(X,[\mathfrak{a}],\frs)\cdot \mathfrak{a}\in \check{C}_*(-Y,g_1^Y,\mathfrak{q}_1),
$$
where the elements of $\widetilde{\mathcal{M}}_0(X,[\mathfrak{a}],\frs)$ are counted with signs using the orientation $\mathfrak{o}$, and the summation of $\frs$ goes over the isomorphism classes of admissible $\spinc$ structures over $X$ relative to $M_s$.

The boundary of the compactification of $\widetilde{M}_1(X,[\mathfrak{a}],\frs)$ consists of three parts: (i) elements of $\mathcal{M}_0(X,[\mathfrak{a}],\frs)(2)$, (ii) broken trajectories given by an element of $\mathcal{M}_0(X,[\mathfrak{a}],\frs)(1)$ and a solution of the blown-up Seiberg-Witten equations on $(\mathbb{R}\times (-Y),\hat{g})$ with respect to the perturbation $\hat{\mathfrak{q}}$, (iii) broken trajectories given by an element of  $\widetilde{M}_0(X,[\mathfrak{a}],\frs)$ and a solution of the blown-up Seiberg-Witten equations on $\mathbb{R}\times (-Y)$ with respect to $g_2^Y,\mathfrak{q}_2$. This implies 
$$\check C_*\big(\mathbb{R}\times (-Y), \mathfrak{q}\big)\big(\check{\psi}_\mathfrak{o}(X)(1)\big)\pm \check{\psi}_\mathfrak{o}(X)(2) \pm  \check{\partial} h_\mathfrak{o}=0,$$
and the proposition is proved.
\end{proof}

\begin{rmk}\label{rmk_contact_case}
Suppose $Y$ is endowed with a contact structure $\xi$, let $X=\mathbb{R}\times Y$. Let $M_s=[1,+\infty)\times Y$ be given by the symplectization of $\xi$ (cf. \cite[(1)]{kronheimer1997monopoles}), and let $M_c=(-\infty,0]\times Y$ be endowed with a cylindrical metric. Then the invariant 
$$c(X)\in \Hto(-Y)/\{\pm 1\}$$ coincides with the contact element of  $\xi$ defined by \cite[Section 6.3]{kronheimer2007monopolesAndLens}. In this case, it was proved by \cite[Theorem H]{lin2018fr} that it is impossible to lift the contact class to $\Hto(-Y)$ such that it is still an isotopy invariant.
\end{rmk}

If $Y=\emptyset$, then for each admissible $\frs$, counting the elements of zero-dimensional moduli space of solutions $(A,\phi)\in\mathcal{C}(X,\frs) $ to \eqref{SW} as in \cite[Definition 2.5]{kronheimer1997monopoles} gives a numerical invariant $SW(X,\frs)\in \mathbb{Z}/\{\pm 1\}$.
 The sign of $SW(X,\frs)$ can be fixed by a choice of homology orientation of $X$ following the same argument as in \cite[Appendix]{kronheimer1997monopoles}. By the compactness properties, there are only finitely many isomorphism classes of $\frs$ such that $SW(X,\frs)\neq 0$. If $\partial M_s$ is a contact manifold and $(M_s,\omega)$ is the symplectization, then $SW(X,\frs)$ coincides with the monopole invariant defined by \cite{kronheimer1997monopoles}.

\begin{lem}\label{lem: only one solution on X}
Suppose $(X,\omega=d\theta)$ is an ESBG end without boundary\footnote{In other words, $X$ is a complete manifold without boundary that satisfies all the conditions in the definition of ESBG ends (with $\partial M=\emptyset$).}, let $Z\subset X$ be a 4-dimensional compact submanifold with boundary, let $M_s=X-Z$. View $X$ as a manifold with an ESBG end $(M_s, \omega|_{M_s})$. Then there exists $r_0>0$ with the following property. Suppose $r>r_0$, $\frs$ is an admissilble $\spinc$ structure relative to $M_s$, and
$$(A,\phi)\in\mathcal{C}(X,\frs)$$
is a solution to \eqref{SW}. Then $\frs$ is isomorhphic to $\frs_{X,\omega}$, and $(A,\phi)$ is gauge equivalent to $(A_0,\sqrt{r}\Phi_0)$ over $X$. Moreover, the moduli space of solutions, which is a point, is regular.
\end{lem}
\begin{proof}
Recall that by our convention, not only $\frs$ is isomorphic to $\frs_{X,\omega}$ on $M_s$, but there is also a fixed isomorphism from $\frs|_{M_s}$ to $\frs_{X,\omega}|_{M_s}$.
Therefore, there is a complex line bundle $E$ over $X$ with a hermitian metric and a fixed isomorphism from $E$ to $\underline{\mathbb{C}}$ on $M_s$, such that $\frs = \frs_{M,\omega}\otimes E$. To simplify the notation, we will identify $\frs$ with $\frs_{X,\omega}$ over $M_s$, and identify $E$ with $\underline{\mathbb{C}}$ over $M_s$, using the fixed isomorphisms.

There is a unitary connection $a$ on $E$, which is equal to the trivial connection of $\underline{\mathbb{C}}$ on $M_s$, such that $A$ is equal to the coupling of $A_0$ and $a$. 
Decompose $\phi$ as $\sqrt{r}(\alpha+\beta)$ such that $\alpha\in T^{0,0}X\otimes E$, $\beta\in T^{0,2}X\otimes E$, where $T^{*,*}X$ is defined with respect to the almost complex structure induced by $(X,\omega)$. 
The same integration by parts as Lemma \ref{lem: integration by parts} gives
\begin{multline}\label{eqn: integration by parts on X}
\int_{X}\Big( \frac{r}{2}|\bar{\partial}_a\alpha + \bar{\partial}_a^*\beta|^2+
2|iF_a^\omega-\frac{r}{8}(1-|\alpha|^2+|\beta|^2)|^2+2|F^{0,2}_a-\frac{r}{4}\bar{\alpha}\beta|^2\\
+\frac{r}{2}iF_a^\omega-2|iF_a^\omega|^2-2|F_a^{0,2}|^2 \Big)\\
= \int_{X}\Big( \frac{r}{4}|\nabla_a\alpha|^2+\frac{r}{4}|\nabla_{A_1+a}\beta|^2+\frac{r}{2} (iF_{A_1+a}^\omega) |\beta|^2  \\
+\frac{r^2}{32}(1-|\alpha|^2-|\beta|^2)^2+\frac{r^2}{8}|\beta|^2-rRe\langle N \circ \partial_a\alpha,\beta\rangle\Big).
\end{multline}
For $r_0$ sufficiently large, the same argument as in \eqref{eqn: bound energy by integration by parts} then gives
$$
\int_{X}|1-|\alpha|^2-|\beta|^2|^2+|\beta|^2+|\nabla_a \alpha|^2+|\nabla_A'\beta|^2+|F_a^+|^2\le\int_{X}\frac{r}{2}iF_a^\omega.
$$
On the other hand, by the same argument that leads to \eqref{eqn_integration_by_parts_no_boundary_term},
$$
\int _X F_a^\omega =\frac12 
\int_X F_a\wedge \omega=\int_X F_a\wedge d\theta=-\int_X dF_a\wedge \theta=0.
$$
Therefore 
$$
|1-|\alpha|^2-|\beta|^2|^2+|\beta|^2+|\nabla_a \alpha|^2+|\nabla_A'\beta|^2+|F_a^+|^2=0
$$
 on $X$, hence $E$ is the trivial bundle, $a$ is the trivial connection, and $(A,\phi)$ is gauge equivalent to $(A_0,\sqrt{r}\,\Phi_0)$ over $X$.  The regularity of the moduli space follows from a straightforward generalization of \cite[Lemma 3.11]{kronheimer1997monopoles}.
\end{proof}

\begin{cor}\label{cor_SW=pm1_symplectic}
Suppose $(X,\omega=d\theta)$ is an ESBG end without boundary, let $Z\subset X$ be a 4-dimensional compact submanifold with boundary, let $M_s=X-Z$. View $X$ as a manifold with an ESBG end $(M_s, \omega|_{M_s})$. Then
$$\sum_{\frs}SW(X,\frs)=\pm 1,$$
where the summation of $\frs$ goes over the isomorphism classes of admissible $\spinc$ structures over $X$ relative to $M_s$. \qed
\end{cor}

Now let $(X^{(1)},Z^{(1)},M_c^{(1)},M_s^{(1)},\theta^{(1)})$ and $(X^{(2)},Z^{(2)},M_c^{(2)},M_s^{(2)},\theta^{(2)})$ be as in Section \ref{subsec_Uniform decay with neck stretching}. Assume that both $M_s^{(1)}$ and $M_s^{(2)}$ are non-compact, and that $M_c^{(1)}$ is given by $(-\infty,0]\times Y$, and $M_c^{(2)}$ is given by $(-\infty,0]\times (-Y)$. For each constant $R>0$, define $X_R$ as in Section \ref{subsec_Uniform decay with neck stretching}. Let $$c(X^{(1)})\in \Hto_\bullet(-Y)/\{\pm 1\},\,\, c(X^{(2)})\in \Hto_\bullet(Y)/\{\pm 1\}$$ 
be given by \eqref{eqn_def_c(X)}. Then we have the following gluing result.
\begin{prop}\label{prop_gluing_Floer_homology}
Let $j_*: \Hto_\bullet(Y)\to \Hfrom_\bullet(Y)$ be the map defined by \cite[Proposition 22.2.1]{kronheimer2007monopoles}, let $\langle\cdot,\cdot \rangle$ be the pairing of $\Hfrom_\bullet(-Y)$ and $\Hfrom_\bullet(Y)$ as given by \cite[Corollary 22.5.11]{kronheimer2007monopoles}. Then 
$$\langle c(X^{(1)}), j_*c(X^{(2)})\rangle = \pm \sum_\frs SW(X_R,\frs),$$
where the summation of $\frs$ goes over the isomorphism classes of admissible $\spinc$ structures over $X_R$ relative to $M_s^{(1)}\cup M_s^{(2)}$.
\end{prop}
\begin{proof}
The proposition follows from Theorem \ref{thm: uniform exponential decay for stretching neck} and the gluing argument of \cite[Section 27]{kronheimer2007monopoles}.
\end{proof}

\section{Monopoles Floer invariants of foliations}\label{sec: monopole floer theory}
This section defines the invariants $c_\pm(\mathcal{F})$ for a smooth oriented foliation $\mathcal{F}$ on a closed oriented 3-manifold $Y$, where $\mathcal{F}$ does not admit holonomy-invariant transverse measure.

\subsection{Symplectizations of smooth taut foliations}\label{sec: setting the stage}
Let $Y$ be a smooth closed oriented 3-manifold, let $\mathcal{F}$ be a smooth oriented foliation on $Y$. The orientations of $Y$ and $\mathcal{F}$ induce a co-orientation of $\mathcal{F}$. Take a smooth non-zero 1-form $\lambda$ such that $\mathcal{F}=\ker \hat\lambda$ and $\hat\lambda$ is positive on the positive side of $\mathcal{F}$. By Frobenius theorem, $\hat\lambda \wedge d\hat\lambda=0$. Since $\mathcal{F}$ has no holonomy-invariant transverse measure, by Sullivan \cite{sullivan1976cycles}, there exists an exact 2-form $\hat{\omega}$ such that $\hat{\omega} \wedge \hat\lambda >0$ everywhere on $Y$. Take a smooth 1-form $\hat\theta$ such that $d\hat{\theta}=\hat{\omega}$.

Consider the cylinder $\mathbb{R}\times Y$, let $t$ be the coordinate of the $\mathbb{R}$-component. Let $\pi_Y:\mathbb{R}\times Y\to Y$ be the projection onto $Y$. Let $\omega=\pi_Y^*(\hat\omega)+d(t\pi_Y^*(\hat\lambda))$, let $\theta=\pi_Y^*(\hat\theta)+t\pi_Y^*(\hat\lambda)$, then $\omega$ is a symplectic form on $\mathbb{R}\times Y$, and $\omega=d\theta$. Let $\lambda = \pi_Y^*(\hat\lambda)$. 

Fix a metric $g_0$ on $Y$ such that $|\hat\lambda|_{g_0}=1$ and $\hat\lambda=*\hat\omega$. Locally $\hat\omega$ can be written as $\hat\omega=e^1\wedge e^2$ where $e^1$ and $e^2$ are orthonormal cotangent vector fields on $Y$.
Since $\hat\lambda \wedge d\hat\lambda=0$, there is a unique 1-form $\mu_1$ such that $d\hat\lambda=\mu_1\wedge \hat\lambda$ and $\langle \mu_1 ,\hat \lambda \rangle_{g_0} =0$. 

We have $d\mu_1\wedge\hat\lambda=d(\mu_1\wedge \hat\lambda)=d(d\hat\lambda) =0$, hence there is a unique 1-form  $\mu_2$ such that $d\mu_1=\mu_2\wedge \hat\lambda$ and $\langle \mu_2, \hat\lambda \rangle _{g_0}=0$. 

Now we define a Riemannian metric on $\mathbb{R}\times Y$ that is compatible with $\omega$ as follows. Notice that locally $\omega = e^1\wedge e^2 + dt \wedge \lambda + t \mu_1 \wedge \lambda$. Take 
\begin{equation} \label{eqn: definition of g}
g= e^1\otimes e^1+ e^2\otimes e^2 + (1+t^2) \lambda\otimes\lambda + \frac{1}{1+t^2} (dt+t\mu_1)\otimes(dt+t\mu_1).
\end{equation}
It is easy to verify that $g$ does not depend on the choice of $e^1$ and $e^2$, and that it is compatible with $\omega$.  Denote $\mathbb{R}\times Y$ by $X$.

\begin{lem} \label{lem: X has bounded geometry}

The manifold $(X,g,\omega=d\theta)$ has the following properties:

\begin{enumerate}
\item $X$ is complete.
\item The injectivity radius of $X$ is bounded from below by a positive number.
\item Let $R$ be the curvature tensor of $X$, and $\nabla$ be the Levi-Civita connection, then $\sup_X|\nabla^k R|<+\infty$ for each $k$.
\item $\sup_X|\nabla^k \theta|<+\infty$ for each $k$.

\end{enumerate}
\end{lem}

\begin{proof}
Suppose $x$ is a real number and $u$ is a vector tangent to the $Y$ component of $X$, let $v=x\cdot\frac{\partial}{\partial{t}}+u$ be a tangent vector of $X$. By the definition of $g$ and Cauchy's inequality:

\begin{align*}
 & |v|\cdot \sqrt{t^2|\mu_1|^2+t^2+1 } \\
\ge & \sqrt{|u|^2+\frac{1}{1+t^2}\big(x+t\cdot\mu_1(u)\big)^2}
    \cdot\sqrt{t^2|\mu_1|^2+(1+t^2)}\\
\ge & |t||\mu_1||u|+|x+t\cdot\mu_1(u)|\\
\ge & |x|.
\end{align*}

Therefore $|v|\ge {|x|}/{\sqrt{1+z\cdot t^2}}$, where $z=\sup|\mu_1|^2+1$. The length of a curve from the slice $t=-T$ to $t=T$ is therefore at least 
$$\int_{-T}^{T}{1}/{\sqrt{1+z\cdot t^2}} \,dt.$$ Since $$\int_{-\infty}^{\infty} {1}/{\sqrt{1+z\cdot t^2}}\,dt=+\infty,$$ 
this implies the completeness of $X$.

To prove the boundedness of $|\nabla^k R|$ and $|\nabla^k \theta|$, we use the moving frame method. Take an arbitrary point $q$ on $Y$, choose local chart $U_q$ of $q$, and fix a choice of $e^1$ and $e^2$ on $U_q$. Let 
\begin{align*}
e^3 &= \sqrt{1+t^2}\cdot \lambda ,\\
e^4 &= \frac{1}{\sqrt{1+t^2}}(dt+t\cdot\mu_1).
\end{align*}
Then $\{e^1, e^2, e^3, e^4\}$ form an orthonormal basis of the cotangent bundle on $U_q\times \mathbb{R}$.
There exist smooth functions $\nu_i$ on $U_q$ ($i=1,2,...,10$), such that
\begin{align*}
   d\,e^1  &=  
          \nu_1\,e^1\wedge e^2+\nu_2\,e^1\wedge \lambda+\nu_3\,e^2\wedge \lambda ,\\
   d\,e^2  &=  
          \nu_4\,e^1\wedge e^2+\nu_5\,e^1\wedge \lambda+\nu_6\,e^2\wedge \lambda ,\\
   \mu_1   &=  \nu_7\,e^1+\nu_8\,e^2 ,\\
   \mu_2 &=  \nu_9\,e^1+\nu_{10}\,e^2 .
\end{align*}
By shrinking $U_q$ if necessary and identifying $U_q$ with a subset of $\mathbb{R}^3$, we have $\|\nu_i\|_{C^m(U_q)}<+\infty$ for all $m$. A straightforward calculation shows:
\begin{align}\label{eqn: calculation of moving frames}
\begin{cases}
   d\,e^1  &=  
   \nu_1\,e^1\wedge e^2+\frac{\nu_2}{\sqrt{1+t^2}}\,e^1\wedge e^3+\frac{\nu_3}{\sqrt{1+t^2}}\,e^2\wedge e^3 ,\\
   d\,e^2  &=  \nu_4\,e^1\wedge e^2+\frac{\nu_5}{\sqrt{1+t^2}}\,e^1\wedge e^3+\frac{\nu_6}{\sqrt{1+t^2}}\,e^2\wedge e^3 ,\\
   d\,e^3  &=  \frac{t}{\sqrt{1+t^2}}\,e^4\wedge e^3 + \frac{\nu_7}{1+t^2}\,e^1\wedge e^3+\frac{\nu_8}{1+t^2}\,e^2\wedge e^3 ,\\
   d\,e^4 &= \frac{1}{1+t^2} e^4\wedge(\nu_7\,e^1+\nu_8\,e^2)-\frac{t}{1+t^2}e^3\wedge(\nu_9\,e^1+\nu_{10}\,e^2).
\end{cases}
\end{align}
Write 
$$de^i=\sum_{j\neq k}a_{jk}^i\,e^j\wedge e^k,$$
such that $a_{jk}^i=-a_{jk}^i$, then the equations above imply that $\|a_{jk}^i\|_{C^m(\mathbb{R}\times U_q)}<+\infty$ for each $m$.
 
Suppose $\nabla e^i=\omega_j^i\otimes e^j$, where $\nabla$ is the Levi-Civita connection. Then the connection matrix 
$\{\omega_i^j\}$ can be calculated from $\{a_{jk}^i\}$ by the formula
$$\omega_j^i= \sum_k (-a_{ji}^k+a_{kj}^i+a_{ik}^j)e^k,$$
and the curvature matrix under the basis $\{e^i\}$ is given by $d\omega_i^j-\omega_i^k\wedge \omega_k^j$. Since $a_{jk}^i$ and their exterior derivatives are bounded, it follows that under the basis 
$\{e^i\}$, every component of $\nabla^m R$ is bounded on $\mathbb{R}\times U_q$ for all $m\ge 1$.
This proves the boundedness of $|\nabla^m R|$ on $\mathbb{R}\times U_q$. Since $Y$ is compact, it can be covered by finitely many such $U_q$'s, therefore $|\nabla^m R|$ is bounded on $X=\mathbb{R}\times U_q$ for every $m$.

For the estimates on $\theta$, write $\theta$ as
$$
\theta=\nu_{11}e^1+\nu_{12}e^2+\nu_{13}\lambda,
$$
then 
$$\theta=\nu_{11}e^1+\nu_{12}e^2+\frac{t+\nu_{13}}{\sqrt{1+t^2}}e^3,$$
and the same calculation proves the boundedness of $|\nabla^m\theta|$.

For the lower bound on injectivity radius, we need the following theorem:

\begin{thm*}[{\cite[Theorem 4.7(i)]{cheeger1982finite}}]
Let $(M^n, g)$ be a complete Riemannian manifold, let $R$ be the Riemannian curvature tensor, let $K>0$ be a constant such that $|R|\le K$ on $M$. 
Let $0<r<\frac{\pi}{4\sqrt{K}}$.
Then the injectivity radius at each point $p\in M$ satisfies the following inequality:
\begin{equation} \label{eqn_lower_bound_inj_radius}
\inj(p) \ge \frac{r}{2}\cdot \frac{1}{1+V_{2r}^{-K}/\Vol(B_p(r))},
\end{equation}
where $V_{2r}^{-K}$ is the volume of a geodesic ball of radius $2r$ on the hyperbolic $n$--space with constant curvature $-K$. 
\end{thm*}
\begin{proof}
This is a special case of \cite[Theorem 4.7(i)]{cheeger1982finite} with $H=-K$, $x=p$, and $r_0=s=r$.
\end{proof}

Back to the proof of Lemma \ref{lem: X has bounded geometry}. 
Let $p=(t,q)\in X$. The following argument will show that $\Vol(B_p(r))$ is bounded from below by a positive constant independent of $p$. Without loss of generality, assume $|t|>1$.

Let $K>0$ be an upper bound of $|R|$. For each point $q\in Y$, let $L_q$ be the leaf of $\mathcal{F}$ through $q$, the metric on $L_q$ is taken to be the restriction from $g_0$. Let $\epsilon=\inf_{q\in Y} \inj(L_q)$. Since $Y$ is compact, $\epsilon$ is positive. 
Let $r=\frac12 \min\{\frac{\pi}{4\sqrt{K}},\epsilon\}$. 

Let $D(q,r/3)$ be the open disk of radius $r/3$ on $L_q$ centered at $q$, 
Let $$U=\{x\in Y| \textrm{dist}_{g_0} (x, D(q,r/3)) < \frac{r}{3\sqrt{1+t^2}} \}.$$
Then the distance from each point in $U$ to $D(q,r/3)$ under the metric $g|_{Y\times \{t\}}$ is less than $r/3$, thus the distance from each point of $U$ to $q$ is less than $2r/3$. Therefore,
$$ B_p(r) \supseteq (e^{-r/3}\,t,e^{r/3}\,t)\times U. $$
The volume of $U$ under the metric $g_0$ is bounded from below by a constant multiple of $r/(3\sqrt{1+t^2})$, where the constant depends only on $g_0$ and $\mathcal{F}$. Therefore the volume of $U\times (e^{-r/3}\,t,e^{r/3}\,t)$ under the product metric $\mathbb{R}\times (Y,g_0)$ is bounded from below by a positive constant. 
Notice that the volume form of the product metric on $\mathbb{R}\times (Y,g_0)$  is the same as the volume form of $g$. Therefore
\begin{equation} \label{sec 2 eq 3}
\Vol (B_p(r)) > \frac{1}{z_2}
\end{equation}
for some positive constant $z_2$ depending on $\mathcal{F}$ and $g_0$.
The lower bound of injectivity radius of $X$ then follows immediately from \eqref{eqn_lower_bound_inj_radius}.
\end{proof}

\begin{rmk}
The fact that the injectivity radius of $X$ is bounded from below could be counter intuitive because of the factor $\frac{1}{1+t^2}$ in the definition of $g$. In fact, by the proof of Lemma \ref{lem: X has bounded geometry}, one can visualize the geometry of $X$ as follows. First consider the three manifold $Y$ with the metric $g_0$. For any $x\in Y$, $r, \epsilon>0$, let $L_x$ be the leaf of $\mathcal{F}$ containing $x$ with the induced metric from $g_0$, let $D_r$ be the $r$-neighborhood of $x$ in $L_x$, and let $D_r(\epsilon)$ be the $\epsilon$ neighborhood of $D_r$ in $Y$. 
When $r$ is fixed and $\epsilon$ is small, $D_r(\epsilon)$ is a thin slice near $D_r$. Now let $r_0>0$ be a lower bound of the injectivity radius, then a normal neighborhood of $X$ centering at $(t,q)$ with radius $r_0$ contains the set $D_{r_0/3}\big(\frac{\,r_0\,}{\,3\sqrt{1+t^2}\,}\big)\times (e^{-r_0/3}t,e^{r_0/3}t)$. When $t$ is large, this is $($a much thinner slice near $D_{r_0/3})$ $\times$ $($a long interval$)$.
\end{rmk}

\subsection{The definition of $c_{\pm}(\mathcal{F})$}
This subsection defines the monopole Floer invariants $c_{\pm}(\mathcal{F})$ for the smooth foliation $\mathcal{F}$.

Let $X=\mathbb{R}\times Y$, and let $g$ be the metric on $X$ defined by \eqref{eqn: definition of g}. Let $\omega=d\theta$ be the compatible symplectic form on $X$ as defined in Section \ref{sec: setting the stage}.

Let $g^+$ be a Riemannian metric on $X$ that is equal to $g$ on $(-\infty,-1]\times Y$, and is cylindrical on $[1,+\infty)\times Y$. Let $g^-$ be a Riemannian metric on $X$ that is equal to $g$ on $[1,+\infty)\times Y$, and is cylindrical on $(-\infty,-1]\times Y$. Let $X_{g^+}$ be the Riemannian manifold $(X,g^+)$, and let $X_{g^-}$ be the Riemannian manifold $(X,g^-)$. By Lemma \ref{lem: X has bounded geometry}, $X_{g^+}$ is a manifold with cylindrical and ESBG ends, where the ESBG structure is given by $\omega=d\theta$ on $(-\infty,-1]\times Y$. 
Similarly, $X_{g^-}$ is a manifold with cylindrical and ESBG ends, where the ESBG structure is given by $\omega=d\theta$ on $[1,+\infty)\times Y$. 
\begin{defn}
Define 
\begin{align*}
c_+(\mathcal{F}) &= c(X_{g^+})\in \Hto_\bullet(Y)/\{\pm1\},\\
c_-(\mathcal{F}) &= c(X_{g^-})\in \Hto_\bullet(-Y)/\{\pm1\},
\end{align*}
 where $c(\cdot)$ is given by \eqref{eqn_def_c(X)}.
\end{defn}

 By Proposition \ref{prop_invariant_under_deformation}, $c(X_{g^\pm})$ are invariant under deformations of the ESBG structures, it follows that $c_\pm(\mathcal{F})$ are independent of the choice of $g_0$ and $\hat\lambda$, and are invariant under smooth deformations of $\mathcal{F}$ via foliations without holonomy-invariant transverse measure.

Let $j_* : \Hto_{\bullet}(Y)\to\Hfrom_{\bullet}(Y)$ be the map in the long exact sequence of monopole Floer homologies introduced by \cite[Proposition 22.2.1]{kronheimer2007monopoles}. The next theorem proves the nonvanishing of $j_*c_{\pm}(\mathcal{F})$.
 
\begin{thm}\label{thm: nonvanishing}
Let $\mathcal{F}$ be a smooth foliation on $Y$ with no holonomy-invariant transverse measure, then 
\begin{align*}
j_*c_+(\mathcal{F}) &\neq 0\in \Hfrom_\bullet(Y)/\{\pm1\},\\
j_*c_-(\mathcal{F}) &\neq 0\in \Hfrom_\bullet(-Y)/\{\pm1\}.
\end{align*}
\end{thm}

\begin{proof}
By Proposition \ref{prop_gluing_Floer_homology} and Corollary \ref{cor_SW=pm1_symplectic}, we have
\begin{equation}\label{eqn: gluing}
\langle c_\mp(\mathcal{F}), j_*c_{\pm}(\mathcal{F}) \rangle = \pm \sum_\frs SW(X,\frs) = \pm 1.
\end{equation}
Hence $j_*c_{\pm}(\mathcal{F})\neq 0$.
\end{proof}

\begin{thm} \label{thm:grading}
The grading of $c_{\pm}(\mathcal{F})\in \Hto_\bullet(\pm Y)$ is represented by the homotopy class of the tangent plane field of $\mathcal{F}$.
\end{thm}

\begin{proof}
The grading of $c_+(\mathcal{F})$ is represented by a nowhere vanishing section $\psi\in\Gamma(Y\times\{0\},\bS^+)$ such that it extends to a nowhere vanishing section of $\bS^+|_{(-\infty,0]\times Y}$ that is asymptotic to the canonical section $\Psi_0$ at $t\to-\infty$. Therefore, we can take $\psi$ to be $\Phi_0|_{Y\times\{0\}}$. A straightforward calculation then shows that the plane field corresponding to $(\bS^+,\psi)$ is homotopic to $\ker \alpha=\mathcal{F}$.
\end{proof}

\section{Topological applications} \label{sec:applications}

\begin{cor}[{\cite[Theorem 2.1]{kronheimer2007monopolesAndLens}, \cite[Theorem 41.4.1]{kronheimer2007monopoles}}]\label{cor: nonvanishing of reduced homology}
Let $Y$ be an oriented three-manifold. If $\mathcal{F}$ is a smooth foliation on $Y$ without holonomy-invariant transverse measure, let $[\mathcal{F}]$ be the homotopy class of the tangent plane field of $\mathcal{F}$, let $\Hred_{[\mathcal{F}]}(Y)$ be the reduced monopole Floer homology at the degree represented by $[\mathcal{F}]$. Then $\Hred_{[\mathcal{F}]}(Y)\neq 0$. 
\end{cor}
\begin{proof}
This is an immediate consequence of Theorem \ref{thm: nonvanishing} and Theorem \ref{thm:grading}.
\end{proof}

\begin{cor}[{\cite[Corollary 1.5]{kronheimer1997monopoles}}]
\label{cor_finitely_many_homotopy_classes}
There are only finitely many homotopy classes of plane fields on $Y$ that can be realized by the tangent plane field of a smooth foliation without holonomy-invariant transverse measure.
\end{cor}
\begin{proof}
By \cite[Proposition 22.2.3]{kronheimer2007monopoles}, $\Hred_\bullet(Y)$ has finite rank, hence the result follows from Corollary \ref{cor: nonvanishing of reduced homology}.
\end{proof}

Since every foliation without holonomy-invariant transverse measure is a taut foliation, the  corollaries above are special cases of the non-vanishing and finiteness results in \cite{kronheimer1997monopoles,kronheimer2007monopolesAndLens}. On the other hand, by the discussions in Section \ref{subsec_taut_foliations}, on a rational homology sphere every foliation without holonomy-invariant transverse measure is a taut foliation. Therefore, Corollary \ref{cor: nonvanishing of reduced homology} and Corollary \ref{cor_finitely_many_homotopy_classes} yield alternative proofs for the non-vanishing and finiteness results of smooth taut foliations on rational homology spheres, without making reference to the Eliashberg-Thurston perturbation.

We can improve Corollary \ref{cor: nonvanishing of reduced homology} to a more general class of three-manifolds. The following lemma shows that in many cases, smooth folaitions without holonomy-invariant transverse measure are ``generic'' among smooth taut foliations. The result was explained to the author by Jonathan Bowden.

\begin{lem}[{\cite{bowden_2016}}]
\label{lem_generic_perturb_taut_foliation}
Let $Y$ be an atoroidal manifold and $\mathcal{F}$ a smooth taut foliation on $Y$. Then either $\mathcal{F}$ can be $C^0$ isotoped to smooth folaition $\mathcal{F}'$ without holonomy-invariant transverse measure, or $Y$ is diffeomorphic to a surface bundle over $S^1$.
\end{lem}

\begin{proof}
By \cite{bonatti1994feuilles}, the foliation $\mathcal{F}$ can be $C^0$ approximated by a smooth taut folaition $\mathcal{F}_1$, such that every closed leaf of $\mathcal{F}_1$ has genus 0 or 1. If $Y\cong S^2\times S^1$, then $\mathcal{F}$ is homeomorphic to the product foliation, and the statement of the lemma is verified. If $Y\ncong S^2\times S^1$, by Reeb's stability theorem the foliation $\mathcal{F}_1$ has no closed leaf with genus 0. Since every closed leaf of a taut foliation is incompressible and $Y$ is assumed to be atoroidal, the foliation $\mathcal{F}_1$ has no torus leaf. This proves that $\mathcal{F}_1$ has no closed leaf.

If $\mathcal{F}_1$ has a holonomy-invariant transverse measure $\mu$, let $A$ be a minimal set contained in the support of $\mu$. The existence of $A$ follows from \cite[Corollary 4.1.13]{candel2000foliations}. Since $\mathcal{F}_1$ has no closed leaf, the minimal set $A$ is either equal to $Y$ or is exceptional as defined in \cite[Definition 4.1.4]{candel2000foliations}.  
If $A$ is exceptional, by Sacksteder's theorem \cite[Theorem 8.2.1]{candel2000foliations}, there exists a leaf $L$ in $A$ containing a curve of contracting linear holonomy. Since $L$ is in the support of $\mu$, on a neighborhood of $L$ the measure $\mu$ has to be a constant multiple of the delta measure of $L$. This implies that $L$ is a closed leaf, which is a contradiction. Therefore $A=Y$. By \cite[Proposition 9.5.8]{candel2000foliations}, in this case $Y$ is diffeomorphic to a surface bundle over $S^1$.
\end{proof}
 
\begin{cor}[{\cite[Theorem 2.1]{kronheimer2007monopolesAndLens}}]\label{cor:atoroidal}
Suppose $Y\neq S^1\times S^2$, $Y$ is atoroidal, and $Y$ supports a smooth taut foliation, then $\Hred_\bullet(Y)\neq 0$.
\end{cor}
\begin{proof}
If $\mathcal{F}$ can be $C^0$ approximated by a smooth taut folaition $\mathcal{F}'$ such that $\mathcal{F}'$ has no holonomy-invariant transverse measure, then the result follows from Corollary \ref{cor: nonvanishing of reduced homology}. Otherwise, by Lemma \ref{lem_generic_perturb_taut_foliation}, $\mathcal{F}$ can be $C^0$ approximated by a smooth taut folaition $\mathcal{F}'$ such that $(Y,\mathcal{F}')$ is homeomorphic to a surface bundle over $S^1$ foliated by the fibers. Since $Y$ is atoroidal and $Y\neq S^1\times S^2$, the genus of the fiber is at least 2. In this case, the desired result follows from \cite[Theorem 3.1]{kronheimer2010knots} and \cite[Lemma 2.2]{kronheimer2010knots}.
\end{proof}

Recall that by Theorem \ref{thm: nonvanishing} and Theorem \ref{thm:grading}, $c_{\pm}(\mathcal{F})$ are non-zero and are graded by the homotopy class of $\mathcal{F}$. It turns out that the invariants $c_{\pm}(\mathcal{F})$ are stronger than the homotopy class itself. The rest of this section constructs examples of foliations $\mathcal{F}_1$ and $\mathcal{F}_2$ such that they are homotopic as plane fields but $c_+(\mathcal{F}_1)\neq c_+(\mathcal{F}_2)$, $c_+(\mathcal{F}_1)\neq c_+(\mathcal{F}_2)$. Since $c_\pm(\mathcal{F})$ are invariant under smooth deformations, this gives examples of smooth foliations without holonomy-invariant transverse measure that are homotopic as plane fields, but cannot be smoothly deformed to each other via foliations without holonomy-invariant transverse measure.

\begin{prop} \label{prop:reverseOrientationChangesInvariant}
Suppose $M$ is a compact oriented 4-manifold with boundary, and let $Y=\partial M$, where $Y$ is oriented such that there is an orientation-preserving diffeomorphim from $[0,1)\times Y$ to a neighborhood of $Y$ in $M$. Let $\mathcal{F}$ be a smooth co-oriented foliation on $Y$ that has no holonomy-invariant transverse measure. Assume there is an exact symplectic form $\omega$ on $M$ such that $\omega|_{Y}$ is positive on $\mathcal{F}$. Assume further that $2\,c_1(\omega)\neq 0$. Let $-\mathcal{F}$ be the same foliation as $\mathcal{F}$ but with reversed orientation. Then $c_+(\mathcal{F})$ and $c_+(-\mathcal{F})$ are linearly independent in $\Hto_\bullet(Y)\otimes \mathbb{Q}$, and $c_-(\mathcal{F})$ and $c_-(-\mathcal{F})$ are linearly independent in $\Hto_\bullet(-Y)\otimes \mathbb{Q}$.
\end{prop}

\begin{proof}
Remove a small ball in $M$, the remaining part of $M$ forms a cobordism from $Y$ to $S^3$. For any $\Spinc$ structure $\frs$ on $M$, it induces a map $\Hfrom(M-B^3,\frs):\Hfrom_*(S^3)\to\Hfrom_*(Y)$. Let $\hat{1} \in \Hfrom_*(S^3)\cong \mathbb{Z}[U_\dagger]$ be a generator as $\mathbb{Z}[U_\dagger]$ module, then $\Hfrom(M-B^3,\frs)(\hat{1})\in \Hfrom_\bullet(-Y)$.

Write $\mathcal{F}=\ker \hat \lambda$, where $\hat \lambda$ is a 1-form on $Y$ such that $\hat \lambda$ is positive on the positive side ot $\mathcal{F}$. Let $\lambda$ be the pull-back of $\hat\lambda$ to $\mathbb{R}\times Y$.

Let $\widetilde{M} = (-\infty,0]\times Y\cup_{\partial M} M$. We can define an exact symplectic form on $\widetilde M$ as follows. Let Let $M_1 = [-1,0]\times Y\cup_{\partial M} M$. Let $\eta:(-\infty,0]\to[-1,0]$ be a smooth non-decreasing function such that $\eta(t)=-1$ when $t\le -1$, $\eta(t)=t$ when $t\ge -1/2$. Let $\chi:(-\infty,0]\to (-\infty,0]$ be a non-decreasing function such that $\chi(t)=t$ when $t\le -1/2$, and $\chi(t)=0$ when $t\ge -1/4$. Let $\pi:\widetilde{M}\to M_1$ be the map defined by $\pi=\eta \times \id_Y$ on $(-\infty,0]\times Y$, and $\pi=\id_M$ on $M$. Let $\varphi:M_1\to M$ be a diffeomorphism that is maps $(-1,x)$ to $(0,x)$ for all $x\in Y$ on the boundary. Let $\epsilon>0$ be a constant. Define $\tilde{\omega} = (\varphi\circ\pi_1)^*\omega +  d(\chi (\epsilon\lambda))$, where $\chi (\epsilon\lambda)$ is defined to be zero on $M$. It is straightforward to verify that when $\epsilon$ is sufficiently small, $\tilde{\omega}$ is symplectic on $\widetilde{M}$. If we endow $\widetilde{M}$ with a compatible metric such that it is equal to the metric given by \eqref{eqn: definition of g} on $(-\infty,-1]$, then $\widetilde{M}$ is a ESBG end without boundary.

By the gluing property and Lemma \ref{lem: only one solution on X},
$$\langle \Hfrom(M-B^3,\frs)(\hat{1}), c_+(\mathcal{F}) \rangle=\begin{cases}\pm 1 \qquad \rm{if}\,\,\frs\cong \frs_{M,\omega},\\
0\qquad\,\,\, \rm{  otherwise.}\end{cases}$$

If we change $\mathcal{F}$ to $-\mathcal{F}$ and change the symplectic form on $M$ from $\omega$ to $-\omega$, the canonical $\Spinc$ structure is then changed to the conjugation of $\frs_0$, hence we have,
$$\langle \Hfrom(M-B^3,\frs)(\hat{1}), c_+(-\mathcal{F}) \rangle=\begin{cases}\pm 1 \qquad \rm{if}\,\,\frs\cong \frs_{M,-\omega},\\
0\qquad\,\,\, \rm{otherwise.}\end{cases}$$

Since $2\,c_1(\omega)\neq 0$, the $\Spinc$ structures $\frs_{M,\omega}$ and $\frs_{M,-\omega}$ are not isomorhpic, therefore $c_+(\mathcal{F})$ and $c_+(-\mathcal{F})$ are linearly independent. The proof for $c_-$ follows from a similar argument.
\end{proof}

The next lemma provides examples that satisfy the conditions of Proposition \ref{prop:reverseOrientationChangesInvariant}. The result was explained to the author by Cheuk-Yu Mak. Recall that a contact form $\alpha$ on $Y$ is said to have a strong symplectic filling if $Y$ bounds a compact symplectic 4-manifold $(M,\omega)$, such that there is a vector field $v$ near $\partial M$ with $(\iota_v\omega)|_{Y} = \alpha$.

\begin{lem} \label{lem:exactFillingOfCircleBundle}
Let $Y$ be an $S^1$ bundle over a compact surface of genus $g$ with Euler number $e<0$ and $e\neq 2-2g$. Then there exists a contact form $\alpha$ on $Y$, such that $\alpha$ has an exact strong symplectic filling with a non-torsion first Chern class, and such that the Reeb vector field of $\alpha$ is the positive unit tangent vector field of the $S^1$-fibers.
\end{lem}

\begin{proof}
Let $E$ be a holomorphic line bundle with Euler number $e$ over a Riemann surface of genus $g$, and let $J$ be the complex structure on $E$. Let $h$ be an Hermittian metric on $E$ such that its Chern connection has negative curvature.  Let $E_1$ be the unit disk bundle of $E$ with respect to the metric $h$, then $E_1$ is a complex manifold with a $J$-convex boundary as defined in \cite[Section 2.3]{cieliebak2012stein}. The circle bundle $\partial E_1$ is a principal $U(1)$-bundle and the Chern connection of $E$ induces a connection on $\partial E_1$. Let $\alpha_0$ be the connection form on $\partial E_1$, then $\ker \alpha_0=T\partial E_1 \cap J (T\partial E_1)$ is a contact structure on $\partial E_1$, and the Reeb vector field of $\alpha_0$ is the positive unit tangent vector field of the $S^1$-fibers. For more details of this computation, the reader may refer to \cite[Section 2.5]{cieliebak2012stein}

By \cite[Theorem (2')]{bogomolov1997stein}, there exists a smooth family of integrable almost complex structures $J_t$, $t\in(0,1)$ on $E_1$, such that $J_0=J$ and $(E_1, J_t)$ is Stein when $t>0$.

Let $f$ be a $J_0$-convex function defined near $\partial E_1$, such that $\partial E_1=f^{-1}(1)$, the value $1$ is a regular value of $f$, and that $f<1$ in the interiori of $E_1$. Then there exists $\epsilon_0>0$, such that for all $0<\delta<\epsilon_0$, the function $f$ is $J_{\delta}$-convex. Let $\alpha_{\delta}:=df\circ J_{\delta}$ be a 1-form on $f^{-1}(1)=\partial E_1$, then $\alpha_{1-\delta}$ is a contact form on $\partial E_1$. 

For sufficiently small $\delta$, the contact structure $\ker \alpha_{\delta}$ is $C^{\infty}$ close to $\ker \alpha_0$, hence by Gray's stability theorem there exists a diffeomorphism $\iota:\partial E_1 \to \partial E_1$ which is isotopic to the identity, and a positive function $u$ on $\partial E_1$, such that $\iota^*(u\cdot\alpha_{\delta})=\alpha_0$. The Reeb vector field of $\iota^*(u\cdot\alpha_{\delta})$ is therefore the positive unit tangent vector field of the $S^1$-fibers. Notice that for a sufficiently large constant $C$, there exists a strong symplectic cobordism from $(\partial E_1, u\cdot\alpha_{\delta})$ to $(\partial E_1, \alpha_{\delta}/C)$. Since $(\partial E_1, \alpha_{\delta})$ is Stein fillable, this implies that the contact form $u\cdot\alpha_{\delta}$ is has a strong exact filling, therefore $\iota^*(u\cdot\alpha_{\delta})$ has a strong exact filling. The first Chern class of the filling is equal to the first Chern class of the complex manifold $(E,J)$, which is not torsion when $e\neq 2-2g$.
Since $Y\cong \partial E_1$, this proves the lemma.
\end{proof}

Let $Y$ be an $S^1$ bundle over a compact surface of genus $g>1$ with Euler number $e$, such that $2-2g<e<0$. By \cite{wood1971bundles}, there exists an oriented smooth foliation $\mathcal{F}$ on $Y$ which is transverse to the $S^1$ fibers. Let $-\mathcal{F}$ be the same foliation as $\mathcal{F}$ but with the opposite orientation.

\begin{prop}
Let $Y$, $e$, $\mathcal{F}$, and $-\mathcal{F}$ be as above, and assume $e|2g-2$. Then $\mathcal{F}$, $-\mathcal{F}$ are foliations without holonomy-invariant transverse measure, then $\mathcal{F}$ and $-\mathcal{F}$ are homotopic as oriented plane fields, but $c_+(\mathcal{F})\neq c_+(-\mathcal{F})$, and $c_-(\mathcal{F})\neq c_-(-\mathcal{F})$.
\end{prop}

\begin{proof}
By Lemma \ref{lem:exactFillingOfCircleBundle}, there exists a contact form $\alpha$ on $Y$ with a strong exact symplectic filling $(M,\omega)$, such that $c_1(\omega)$ is not torsion on $M$, and the Reeb vector field of $\alpha$ is positively transverse to $\mathcal{F}$. Notice that the Reeb vector field being positively transverse to $\mathcal{F}$ is equivalent to the form $\omega$ being positive on $\mathcal{F}$. Since $\omega$ is exact, this implies that $\mathcal{F}$ and $-\mathcal{F}$ have no holonomy-invariant transverse measure. Moreover, by Proposition \ref{prop:reverseOrientationChangesInvariant}, $c_+(\mathcal{F})$ and $c_+(-\mathcal{F})$ are linearly independent, $c_-(\mathcal{F})$ and $c_-(-\mathcal{F})$ are linearly independent.

It remains to prove that $\mathcal{F}$ and $-\mathcal{F}$ are homotopic as plane fields. Let $S^1\to Y\stackrel{\pi}{\to} \Sigma$ be the bundle structure of $Y$, let $e(Y)\in H^2(\Sigma)$ be the Euler class of the bundle. By the Gysin exact equence, 
$$ H^0(\Sigma) \stackrel{\cup e(Y)}{\longrightarrow } H^2(\Sigma) \stackrel{\pi^*}{\longrightarrow} H^2(Y) $$ 
is exact. Notice that $\mathcal{F}$ is isomorphic to $\pi^*(T\Sigma)$ as a plane bundle, therefore the assumption $e|2g-2$ implies that the Euler class of $\mathcal{F}$ is zero, hence $\mathcal{F}$ has a global basis $\{e_1,e_2\}$. Let $e_3$ be the positively oriented normal vector field of $\mathcal{F}$, then for $t\in [0,1]$ the family of plane fields $\mathcal{F}_t=\rm{span}\big\{e_1, \cos(\pi t)\,e_2+\sin(\pi t)\,e_3\big\}$ defines a homotopy from $\mathcal{F}$ to $-\mathcal{F}$.
\end{proof}

\bibliographystyle{amsalpha}
\bibliography{references}

\end{document}